\documentclass[10pt,twoside]{article}

\usepackage[utf8]{inputenc}
\usepackage[T1]{fontenc}
\usepackage[numbers,sort&compress]{natbib}

\usepackage{amsmath,amsfonts,amssymb,amscd,mathrsfs,array,subeqnarray}
\usepackage{amsthm}
\usepackage{booktabs}
\usepackage{tabularx}
\usepackage{enumitem}
\usepackage{stmaryrd}
\usepackage{bm}
\usepackage{xcolor}

\definecolor{lightrefblue}{RGB}{55,126,168}
\usepackage[
  colorlinks=true,
  linkcolor=blue,
  citecolor=lightrefblue,
  urlcolor=lightrefblue
]{hyperref}

\usepackage{geometry}
\theoremstyle{plain}
\newtheorem{thm}{\bf Theorem}[section]
\newtheorem{coro}[thm]{\bf Corollary}
\newtheorem{lem}[thm]{\bf Lemma}
\newtheorem{prop}[thm]{\bf Proposition}
\newtheorem{defn}[thm]{\bf Definition}

\newtheorem{remark}[thm]{\bf Remark}
\newtheorem{exam}{\bf Example}[section]

\makeatletter \@addtoreset{equation}{section} \makeatother \makeatletter
\@addtoreset{thm}{section} \makeatother

\newcommand\assref[1]{\hyperref[#1]{Assumption \ref*{#1}}}
\newcommand\lemref[1]{\hyperref[#1]{Lemma \ref*{#1}}}
\newcommand\defref[1]{\hyperref[#1]{Definition \ref*{#1}}}
\newcommand\thmref[1]{\hyperref[#1]{Theorem \ref*{#1}}}
\newcommand\proref[1]{\hyperref[#1]{Proposition \ref*{#1}}}
\newcommand\corref[1]{\hyperref[#1]{Corollary \ref*{#1}}}
\newcommand\figref[1]{\hyperref[#1]{Figure \ref*{#1}}}
\newcommand\tabref[1]{\hyperref[#1]{Table \ref*{#1}}}
\newcommand\secref[1]{\hyperref[#1]{Section \ref*{#1}}}
\newcommand\appref[1]{\hyperref[#1]{Appendix \ref*{#1}}}
\newcommand\remref[1]{\hyperref[#1]{Remark \ref*{#1}}}
\newcommand\exaref[1]{\hyperref[#1]{Example \ref*{#1}}}

\newcommand*{\dd}{\mathop{}\!\mathrm{d}}

\allowdisplaybreaks[4]

\newcommand{\SLQh}{\ensuremath{\widehat{\mathrm{SLQ}\text{-}\mathrm{J}}}}
\begin{document}

\title{\Large\bfseries Stochastic Riccati equations with jumps for indefinite linear-quadratic control: A stochastic value-flow approach\thanks{JW is supported by National Natural Science Foundation of China (Grant No. 12571478) and Guangdong Basic and Applied Basic Research Foundation (Grant No. 2025B1515020091), and Shenzhen Science and Technology Program (Grant No. JCYJ20230807093309021).
JX is supported by the National Natural Science Foundation of China (Grant No. 12471418).
XZ is supported by the National Natural Science Foundation of China (Grant Nos. 12171086 and 12371472), Tianyuan Fund for Mathematics (Grant No. 12426652), and Jiangsu Provincial Scientific Research Center of Applied Mathematics (Grant No. BK20233002).}}

\author{
Kai Ding\thanks{Department of Mathematics, Southern University of Science and Technology, Shenzhen, 518055, China (Email: {\tt dingk@sustech.edu.cn}).}~,~~~~
Jiaqiang Wen\thanks{Department of Mathematics and SUSTech International Center for Mathematics, Southern University of Science and Technology, Shenzhen, 518055, China (Email: {\tt wenjq@sustech.edu.cn}).}~,~~~~
Jie Xiong\thanks{Department of Mathematics and SUSTech International Center for Mathematics,
Southern University of Science and Technology, Shenzhen, 518055, China (Email: {\tt xiongj@sustech.edu.cn}).}~,~~~~
Xin Zhang\thanks{School of Mathematics, Southeast University, Nanjing, 211189, China (Email: {\tt x.zhang.seu@gmail.com}).}
}

\date{}

\maketitle

\noindent{\bf Abstract:}
This paper studies the nonlinear matrix-valued stochastic Riccati equation with jumps (SRE-J) arising from indefinite stochastic linear-quadratic optimal control problems with random coefficients and Poisson jumps.
We demonstrate that the presence of jumps may cause the matrix-valued optimal state process $\mathbf{X}$ to become singular, even under the standard condition. 
Consequently, the classical representation $P=\mathbf{Y}\mathbf{X}^{-1}$ used in the diffusion case is generally unavailable.
To overcome this difficulty, we construct the Riccati process directly from the stochastic value flow under the uniform convexity condition on the cost functional, an approach that requires no invertibility condition on the jump coefficient. 
We then prove that the SRE-J admits a maximal strongly regular solution, which yields a closed-loop representation of the unique open-loop optimal control.
Furthermore, we establish verifiable sufficient conditions for uniform convexity and provide examples of indefinite stochastic linear-quadratic problems with jumps.

\medskip



\section{Introduction}

\subsection{Problem formulation}
Let $(\Omega,\mathcal F,\mathbb F,\mathbb P)$ be a complete filtered
probability space on which are defined a one-dimensional standard Brownian motion $W$
and a Poisson random measure $N(\mathrm dt,\mathrm dz)$ with
finite characteristic measure $\lambda(\mathrm dz)$ on a
measurable space $(Z,\mathcal B(Z))$, i.e., $\lambda(Z)<\infty$.
% Hence $N$ has finitely many jumps on $[0,T]$ almost surely.
We assume that $W$ and $N$ are independent, and that $\mathbb{F}= \{\mathcal{F}_t\}_{t\ge 0}$ is the natural filtration of $W$  and  $N$ augmented by all the $\mathbb{P}$-null sets in $\mathcal{F}$, that is,
\[
\mathcal{F}_t=\sigma \left\{ W(s) ;s\le t \right\} \bigvee{\sigma \left\{ \int_0^s{\int_A{N\left( \mathrm{d}u,\mathrm{d}z \right)}}; s\le t,A\in \mathcal{B}\left( Z \right) \right\}}\bigvee{\mathcal{N}},
\]
where $\mathcal{N}$ denotes the collection of all $\mathbb{P}$-null sets in $\mathcal F$. 
% Therefore, $\mathbb{F}$ automatically satisfies the usual conditions. We can and do
% take all semimartingales to have right-continuous paths with left limits (c\`{a}dl\`{a}g).
Furthermore, the compensated version of $N (\mathrm{d} t,\mathrm{d} z)$ is defined as
$
\widetilde{N}\left( \mathrm{d}t,\mathrm{d}z \right) =N\left( \mathrm{d}t,\mathrm{d}z \right) -\lambda \left( \mathrm{d}z \right) \mathrm{d}t.
$

Consider the following controlled linear stochastic differential equation on the finite horizon $[t, T]$:
\begin{equation}\label{state}
\begin{cases}
	\mathrm{d}X(s) =\left[ A(s) X\left( s- \right) +B(s) u(s)  \right] \mathrm{d}s+\left[ C(s) X\left( s- \right) +D(s) u(s) \right] \mathrm{d}W(s) \\
	\hspace{4em} +\int_{Z}{\left[ E\left( s,z \right) X\left( s- \right) +F\left( s,z \right) u(s) \right] \widetilde{N}\left( \mathrm{d}s,\mathrm{d}z \right)}\\
	X\left( t \right) =\xi\\
\end{cases}
\end{equation}
where  $A, C:[0, T] \times \Omega \rightarrow \mathbb{R}^{n \times n}$, $B, D:[0, T] \times \Omega \rightarrow \mathbb{R}^{n \times m}$,  and $E:[0, T] \times \Omega \times Z \rightarrow \mathbb{R}^{n \times n} $, $F:[0, T] \times \Omega \times Z \rightarrow \mathbb{R}^{n \times m} $ are given matrix-valued processes, called the coefficients of the \emph{state equation} \eqref{state}. The initial time $t$ ranges over $[0,T]$, and the initial state $\xi$
is an $\mathbb R^n$-valued, $\mathcal F_t$-measurable random variable
satisfying $\mathbb E[|\xi|^2]<\infty$.
% In \eqref{state},  $X(\cdot)$ is the $\mathbb{R}^{n}$-valued state process and $u(\cdot)$ is the $\mathbb{R}^{m}$-valued control process. 
We call $u(\cdot)$ an \emph{admissible control} on $[t,T]$ if it belongs to the Hilbert space
\begin{equation}\label{admissible control}
\mathcal{U}[t,T]=L_{\mathbb{F}}^2(t,T;\mathbb{R}^m)=\bigg\{u:[t,T]\times\Omega\to\mathbb{R}^{m}\mid u(\cdot)~\text{is}~\mathbb{F}\text{-predictable},~\mathbb{E}\bigg[\int_{t}^{T}|u(s)|^{2}\dd s\bigg]<\infty\bigg\}.
\end{equation}
The cost functional associated with \eqref{state} is 
\begin{equation}\label{costL}
J\left( t,\xi ;u\left( \cdot \right) \right) =\mathbb E\bigg[ \left< GX\left( T \right) ,X\left( T \right) \right> +\int_t^T{\big[\left< Q(s) X(s) ,X(s) \right> +2\left< S(s) X(s) ,u(s) \right> +\left< R(s) u(s) ,u(s) \right>\big] \mathrm{d}s}\bigg],
\end{equation}
where the random variable $G: \Omega \rightarrow \mathbb{S}^{n}$ and the processes $(Q, R, S): [0, T]\times \Omega   \rightarrow\left(\mathbb{S}^{n}, \mathbb{S}^{m}, \mathbb{R}^{m \times n}\right)$ are called the weighting coefficients.  The optimal control problem for
\eqref{state} and \eqref{costL} can be stated as follows

\textbf{Problem (SLQ-J):} For any initial time $t\in[0,T]$ and any
square-integrable $\mathcal F_t$-measurable initial state $\xi$, find a control $u^*\in\mathcal U[t,T]$ such that
\begin{equation}\label{Jmin}
  J(t,\xi;u^*)
  =
  \inf_{u\in\mathcal U[t,T]} J(t,\xi;u)
  \triangleq V(t,\xi).
\end{equation}

% Under \ref{SLQH1}--\ref{SLQH2}, for any $(t,\xi)\in\mathcal D$ and $u\in\mathcal U[t,T]$, the random variable $L(t,\xi;u)$ is integrable. We define
%   \begin{align*}
%   &J(t, \xi ; u)=\mathbb{E}[L(t, \xi ; u)],\qquad \widehat J(t,\xi;u)=\mathbb E[L(t,\xi;u)\mid\mathcal F_t], \quad(t, \xi) \in \mathcal{D}, u \in \mathcal{U}[t, T]. 
%   \end{align*}
%   These two functionals measure the performance of the control $u \in \mathcal{U}[t, T]$ and are called the \emph{expected cost functional} and the \emph{conditional cost functional}, respectively.

\subsection{Background and motivating examples}

The study of stochastic linear-quadratic (SLQ) optimal control problems was initiated by Wonham~\cite{wonham1968matrix} in 1968. 
Under the classical \emph{standard condition}
($Q\ge0$, $G\ge0$, $S=0$, $R\ge\delta I_m$),
the associated Riccati equation admits a unique solution
and the optimal control has a linear state feedback representation;
see~\cite{bismut1976linear,davis1977linear,bensoussan1982lectures,
yong1999stochastic}.
It is worth noting that for the deterministic LQ problem,
Molinari~\cite{molinari1977time} showed that the conditions $Q\ge0$
and $G\ge0$ are not necessary for open-loop solvability, though
$R\ge0$ remains necessary.
In contrast, for SLQ problems driven by Brownian motion,
Chen, Li, and Zhou~\cite{chen1998stochastic} discovered that even
$R\ge0$ is not necessary for solvability.
This phenomenon, now referred to as the \emph{indefinite} SLQ problem,
has attracted extensive subsequent research; see,
e.g.,~\cite{lim1999stochastic,chen2000stochastic,%
ait2002indefinite,hu2003indefinite,mei2021optimal}
and the references therein.
For SLQ problems with deterministic coefficients,
Sun, Li, and Yong~\cite{sun2016open} gave a complete
treatment. They proved that the strongly regular solvability of the Riccati equation is equivalent to the uniform convexity of the cost functional, and obtained a closed-loop representation of the open-loop optimal control.

% establishing the equivalence between the strongly regular
% solvability of the Riccati equation and the uniform convexity of the
% cost functional.

For SLQ problems with \emph{random coefficients}, the Riccati equation
becomes a highly nonlinear backward SDE, usually called the
\emph{stochastic Riccati equation} (SRE).
Under the standard condition, Tang~\cite{tang2003general} first proved the
unique solvability of the matrix-valued SRE by connecting it to the stochastic
Hamilton system, which settled a long-standing problem raised by Bismut~\cite{bismut1976linear} and Peng~\cite{peng1999open}.
A key step in his approach is to show that the matrix-valued optimal state
$\mathbf X$ is globally invertible, so that the Riccati process can
be recovered via $P=\mathbf Y\mathbf X^{-1}$, where $\mathbf Y$ is
the matrix-valued adjoint process.
Tang~\cite{tang2015dynamic} later gave an alternative proof based on
the dynamic programming principle and the Doob--Meyer decomposition.
Under the weaker uniform convexity condition,
Sun, Xiong, and Yong~\cite{sun2021indefinite} extended the theory to
the \emph{indefinite} case with random coefficients.
Their proof also relies on the representation $P=\mathbf Y\mathbf X^{-1}$,
but requires a new and delicate technique to establish the invertibility
of $\mathbf X$ when the weighting matrices are no longer assumed positive.
Their proof strategy can be summarized as follows:
\begin{equation}\label{sunchain}
  \text{Uniform convexity}
  \;\Rightarrow\; P~\text{bounded and continuous}
  \;\Rightarrow\; \mathbf X~\text{invertible}
  \;\Rightarrow\; P=\mathbf Y\mathbf X^{-1}
  \;\Rightarrow\; \text{SRE solvable}.
\end{equation}
For SLQ problems driven only by Brownian motion, the continuity of $P$ is crucial to the stopping-time method for proving the invertibility of~$\mathbf X$.
This invertibility is then used to establish the solvability of
the SRE and to derive a closed-loop representation of the unique
open-loop optimal control.

  % This invertibility yields unique open-loop solvability for
  % every initial pair and the closed-loop representation via the SRE.

On the other hand, the SLQ problem with Poisson jumps has attracted considerable attention in finance, insurance, and engineering, where jump terms naturally model abrupt changes in the state dynamics; see, e.g.,~\cite{li2018indefinite,Zhang2021,wen2022stochastic,meng2023optimal} and the references therein. The Riccati equation associated with Problem (SLQ-J) is the following matrix-valued backward stochastic differential equation with jumps (see Zhang, Dong, and Meng~\cite{zhang2020backward}), which we call the \emph{stochastic Riccati equation with jumps} (SRE-J, for short):
% Zhang, Dong, and Meng \cite{zhang2020backward} introduced the corresponding \emph{stochastic Riccati equation with jumps} (SRE-J):
\begin{equation}\label{SREJ}
\begin{cases}
  \mathrm{d} P(s)=-\Big(A(s)^\top P(s-)+P(s-)A(s)+Q(s)
  +C(s)^\top P(s-)C(s)+C(s)^\top\Lambda(s)+\Lambda(s)C(s)\\
  \qquad\qquad\quad
  +\displaystyle\int_Z\Big[\zeta(s,z)E(s,z)+E(s,z)^\top\zeta(s,z)
  +E(s,z)^\top\big(\zeta(s,z)+P(s-)\big)E(s,z)\Big]
  \lambda(\mathrm{d} z)\\
  \qquad\qquad\quad
  -\widehat S(s)^\top\widehat R(s)^{-1}\widehat S(s)\Big)\mathrm{d} s
  +\Lambda(s)\mathrm{d} W(s)
  +\displaystyle\int_Z\zeta(s,z)\widetilde N(\mathrm{d} s,\mathrm{d} z),\\
  P(T)=G,
\end{cases}
\end{equation}
where \(\widehat S\) and \(\widehat R\) are defined by
\begin{equation}\label{RShat}
\begin{aligned}
  \widehat S(s)&:=B(s)^\top P(s-)+D(s)^\top\Lambda(s)
  +D(s)^\top P(s-)C(s)+S(s)\\
  &\quad+\int_Z\!\Big[F(s,z)^\top\zeta(s,z)
  +F(s,z)^\top\big(\zeta(s,z)+P(s-)\big)E(s,z)\Big]\lambda(\mathrm{d}z),\\
  \widehat R(s)&:=R(s)+D(s)^\top P(s-)D(s)
  +\int_Z F(s,z)^\top\big(\zeta(s,z)+P(s-)\big)F(s,z)\lambda(\mathrm{d}z).
\end{aligned}
\end{equation}
% Here, \(P_-(s):=P(s-)\) denotes the left limit of the c\`adl\`ag process \(P\).
A solution of the SRE-J \eqref{SREJ} is a triplet of \(\mathbb S^n\)-valued processes
\((P,\Lambda,\zeta)\).
Thus, \eqref{SREJ} is a nonlinear matrix-valued backward stochastic
differential equation with jumps.
 Compared with the SRE in the
diffusion case, its generator contains an additional coupling through
the jump martingale component \(\zeta\). Indeed, \(\widehat S\) depends
on all three unknown components \((P,\Lambda,\zeta)\), whereas
\(\widehat R\) depends on \((P,\zeta)\). Consequently, \(\zeta\) enters
both \(\widehat S\) and the inverse matrix \(\widehat R^{-1}\) in the
nonlinear term
\(\widehat S^\top\widehat R^{-1}\widehat S\).
When \(E=F=0\), the terms involving \(\zeta\) disappear from the
generator, so that
\(
  \widehat R(s)
  =R(s)+D(s)^\top P(s-)D(s)
\)
depends only on \(P\), and the SRE-J formally reduces to the SRE
studied in
\cite{tang2003general,tang2015dynamic,sun2021indefinite}.
% Its generator is more strongly coupled than that of
% the SRE in the diffusion case. Indeed, \(\widehat S\) depends on all
% three unknown components \((P,\Lambda,\zeta)\), whereas \(\widehat R\)
% depends on \((P,\zeta)\). Consequently, the jump martingale component
% \(\zeta\) enters both \(\widehat S\) and the inverse matrix
% \(\widehat R^{-1}\) in the nonlinear term
% \(\widehat S^\top\widehat R^{-1}\widehat S\). By contrast, when
% \(E=F=0\), the SRE-J reduces to the SRE studied
% in~\cite{tang2003general,tang2015dynamic,sun2021indefinite}, 
% the jump component vanishes and
% \(
%   \widehat R=R+D^\top P D
% \)
% depends only on \(P\). This additional coupling through \(\zeta\) is a
% characteristic difficulty of stochastic Riccati equations with jump.
% When \(E=F=0\), the SRE-J reduces to the SRE studied
% in~\cite{tang2003general,tang2015dynamic,sun2021indefinite}.
Under the standard condition,
Zhang, Dong, and Meng~\cite{zhang2020backward} proved the unique solvability
of the SRE-J and constructed the optimal feedback control by combining
dynamic programming, the Doob--Meyer decomposition, and an inverse flow
technique.
For the indefinite case with jumps, Li, Wu, and
Yu~\cite{li2018indefinite} obtained solvability results by a relaxed compensator method, which transforms the problem into an auxiliary positive definite one.
In a broader context, Meng et al.~\cite{meng2023optimal} studied stochastic HJB equations with jumps
for non-Markovian control problems, but their framework does not cover the LQ problem due to
the unbounded control region.
These works motivate the following two questions.

\smallskip
\textbf{(Q1)}\; \emph{Is the standard condition imposed
in~\cite{zhang2020backward} necessary for the unique open-loop
solvability of the SLQ problem and the solvability of the SRE-J \eqref{SREJ},
or does the uniform convexity condition suffice?}
\smallskip

The following example shows that the standard condition is not
necessary.

\begin{exam}[SRE-J solvable with $G<0$]\label{exam:Gneg}
Let $n=m=1$, $\lambda(Z)=1$, and take
$A=B=C=E=S=0$, $D=F=1$, $G=-\tfrac34$, $Q=0$, $R(s)=s^2+2$.
Then the state equation is 
\[
  \mathrm dX(s)=u(s)\,\mathrm dW(s)
  +\!\int_Z u(s)\,\widetilde N(\mathrm ds,\mathrm dz),
  \qquad X(t)=\xi,
\]
and the cost functional becomes
\(
  J(t,\xi;u)
  =\mathbb E\big[-\tfrac34\,X(T)^2
  +\!\int_t^T(s^2+2)\,|u(s)|^2\,\mathrm ds\big].
\)
Applying It\^o's formula to $s\mapsto-\tfrac34|X(s)|^2$ and adding the running cost gives
\begin{equation}\label{eq:Gneg_square}
  J(t,\xi;u)=-\frac34\,\mathbb E[\xi^2]
  +\mathbb E\bigg[\int_t^T\Big(s^2+\frac12\Big)|u(s)|^2\,\mathrm ds\bigg],
\end{equation}
which implies that the optimal control is $u^* \equiv 0$. Moreover, 
the SRE-J \eqref{SREJ} for this problem reduces to
\begin{equation}\label{eq:SRE-ex1}
\mathrm dP(s)=
\frac{\big(\Lambda(s)+\int_Z\zeta(s,z)\lambda(\mathrm dz)\big)^2}
{s^2+2+2P(s-)+\int_Z\zeta(s,z)\lambda(\mathrm dz)}\,\mathrm ds
+\Lambda(s)\mathrm dW(s)
+\int_Z\zeta(s,z)\widetilde N(\mathrm ds,\mathrm dz),
\quad
P(T)=-\frac34.
\end{equation}
It is easy to verify that $(P,\Lambda,\zeta)\equiv(-\tfrac34,0,0)$ solves \eqref{eq:SRE-ex1}.
\end{exam}

This example shows that the SRE-J is solvable and the SLQ problem
is uniquely open-loop solvable even though the standard condition
fails. Moreover, setting $\xi=0$ 
in~\eqref{eq:Gneg_square} gives
$J(t,0;u)\ge\tfrac12\,\mathbb E[\int_t^T\!|u|^2\mathrm ds]$,
so the cost functional is uniformly convex although $G<0$.
This motivates proving the solvability of the SRE-J under the
uniform convexity condition.
A natural strategy is to extend the method~\eqref{sunchain} to SLQ problems with Poisson jumps, which leads to the second question.
 
\smallskip
\textbf{(Q2)}\; \emph{Can the global fundamental matrix representation
$P=\mathbf Y\mathbf X^{-1}$, which is the cornerstone of the SRE
solvability proof in~\cite{sun2021indefinite}, be extended to SLQ problems with Poisson jumps?}
\smallskip

In fact, substituting the feedback control $u(s)=\Theta(s)X(s-)$ into~\eqref{state}
shows that at each jump time $T_j$ with mark $Z_j$,
\[
  X(T_j)=\bigl[I_n+E(T_j,Z_j)+F(T_j,Z_j)\Theta(T_j)\bigr]\,X(T_j-).
\]
We call $I_n+E(s,z)+F(s,z)\Theta(s)$ the \emph{closed-loop jump matrix}.
The following example shows that this matrix can be singular under the
optimal control, so the answer is negative.

\begin{exam}[Failure of invertibility of $\mathbf X$]
\label{exam:intro_failure}\label{exam:counterexample2}
Let $n=m=1$, $Z=\{z_0\}$, $\lambda(Z)=\tfrac12$, and take
$A=C=D=S=0$, $B=1$, $E=0$, $F=1$, $R\equiv1$, $Q\equiv G\equiv2$.
Then the state equation is 
\[
  \mathrm dX(s)=u(s)\,\mathrm ds
  + u(s)\,\widetilde N(\mathrm ds,\{z_0\}),
  \quad X(0)=\xi,
\]
and the cost functional is
\(
  J(0,\xi;u)
  =\mathbb E\Big[2X(T)^2
  +\!\int_0^T\!\big(2X(s)^2+|u(s)|^2\big)\mathrm ds\Big].
\)
Since $Q=G=2$ and $R=1$, the standard condition holds.
Applying It\^o's formula to $s\mapsto 2|X(s)|^2$ and adding the
running cost, we obtain
\begin{equation}
  J(0,\xi;u)
  =2 \mathbb E[\xi^2]
  +2\,\mathbb E\bigg[\int_0^T\!\big(u(s)+X(s-)\big)^2\mathrm ds\bigg].
\end{equation}
Hence the unique optimal control is $u^*(s)=-X^*(s-)$, and the
corresponding feedback coefficient is $\Theta=-1$. Then the
closed-loop jump matrix is \(1+E+F\Theta=0\), and consequently
\(X^*(T_j)=0\) at every Poisson jump time \(T_j\). Since
\(\mathbb P(T_1\le T)>0\), the optimal state fundamental matrix is
singular with positive probability.
Nevertheless, the corresponding SRE-J \eqref{SREJ} for this problem is (with $\zeta(s):=\zeta(s,z_0)$)
\begin{equation}\label{eq:SRE-ex2}
\mathrm dP(s)=\bigg(
\frac{\big(P(s-)+\tfrac12\zeta(s)\big)^2}
{1+\tfrac12 P(s-)+\tfrac12\zeta(s)}-2
\bigg)\mathrm ds
+\Lambda(s)\,\mathrm dW(s)
+\zeta(s)\,\widetilde N(\mathrm ds,\{z_0\}),
\quad P(T)=2.
\end{equation}
It is easy to verify that $(P,\Lambda,\zeta)\equiv(2,0,0)$ solves the SRE-J \eqref{eq:SRE-ex2}.
\end{exam}

In \exaref{exam:intro_failure}, the matrix $I_n+E$ is invertible, whereas the closed-loop jump matrix is singular.
Conversely, the following example shows that $I_n+E$ may be singular while the cost functional is uniformly convex and the closed-loop jump matrix is invertible.

\begin{exam}[SRE-J solvable with $I_n+E=0$]\label{exam:IEsingular}
Let $n=m=1$, $\lambda(Z)=1$, and take
\(
  A=C=S=Q=0, B=2, D=F=1, E=-1, G=1, R=-1.
\)
Then $I_n+E=0$, and the state equation is 
\[
  \mathrm dX(s)=2u(s)\,\mathrm ds+u(s)\,\mathrm dW(s)
  +\int_Z\big[u(s)-X(s-)\big]\widetilde N(\mathrm ds,\mathrm dz),
  \quad X(t)=\xi.
\]
The cost functional is $J(t,\xi;u)=\mathbb E\Big[X(T)^2-\int_t^T|u(s)|^2\,\mathrm ds\Big]$.
Applying It\^o's formula to $s\mapsto|X(s)|^2$, we obtain
\begin{equation}\label{eq:IEsingular_square}
  J(t,\xi;u)=\mathbb E|\xi|^2
  +\mathbb E\int_t^T\big|u(s)+X(s-)\big|^2\,\mathrm ds .
\end{equation}
Hence the unique open-loop optimal control admits the closed-loop representation $u^*(s)=\Theta X^*(s-)$ with $\Theta=-1$, and the closed-loop jump matrix $1+E+F\Theta=-1$ is invertible.
It is straightforward to verify that $(P,\Lambda,\zeta)\equiv(1,0,0)$ solves the SRE-J \eqref{SREJ}, with $\widehat R=1$ and $\widehat S=1$.

To verify uniform convexity, let $t=0$, $\xi=0$, and $v:=u+X(\cdot-)$. By It\^o's formula, $y(s):=\mathbb E|X(s)|^2$ satisfies
$y'(s)=\mathbb E[X(s)^2-2X(s)v(s)+2v(s)^2]\le2y(s)+3\,\mathbb E|v(s)|^2$.
It follows from Gronwall's inequality that
$\int_0^Ty(s)\,\mathrm ds\le\frac32(e^{2T}-1)\,\mathbb E\int_0^T|v(s)|^2\,\mathrm ds$.
Since $|u|^2\le2|v|^2+2|X_-|^2$, we obtain
\[
  \mathbb E\int_0^T|u(s)|^2\,\mathrm ds
  \le(3e^{2T}-1)\,\mathbb E\int_0^T|v(s)|^2\,\mathrm ds
  =(3e^{2T}-1)\,J(0,0;u).
\]
Thus the uniform convexity condition \ref{N_0uniformalconvex2-I} in \secref{sec:formulation} holds with $\varepsilon=(3e^{2T}-1)^{-1}$.
Since $R<0$ and $I_n+E=0$, this example satisfies neither the standard condition of~\cite{zhang2020backward} nor the nondegeneracy condition $|\det(I_n+E)|\ge\delta_0>0$ of~\cite{ma2026indefinite}. Nevertheless, it is covered by \thmref{thm:SREsolve} below.
\end{exam}

% The above example and the structure of the SRE-J reveal two 
% main obstacles to the existence of a strongly regular
% solution to the SRE-J. 

The above examples and the structure of the SRE-J show that the
method summarized in \eqref{sunchain} cannot, in general, be directly
extended to SLQ problems with Poisson jumps.
Two difficulties must be resolved to establish strongly regular
solvability under the uniform convexity condition.

\begin{itemize}
  \item In the presence of Poisson jumps, the optimal state fundamental matrix
may lose invertibility at a jump time. Indeed, the additional
inverse-flow condition on \(I_n+E\) discussed in
\cite[Remark~3.1]{zhang2020backward} is satisfied in
\exaref{exam:intro_failure}, since \(I_n+E=I_n\), whereas the optimal
closed-loop jump matrix \(I_n+E+F\Theta\) is singular. The reason is
that its nonsingularity depends on the unknown feedback coefficient
\(\Theta\). Consequently, the global representation
\(P=\mathbf Y\mathbf X^{-1}\) used in the diffusion case is not
available in general. Moreover, by \exaref{exam:IEsingular}, a condition on $I_n+E$ is not natural either, since the problem may be well posed even if $I_n+E=0$.

\item The second difficulty is to establish the uniform positive definiteness of
$\widehat R$. This property is required for strong regularity and
ensures that both the nonlinear term
\(\widehat S^\top\widehat R^{-1}\widehat S\) and the feedback
coefficient \(\Theta=-\widehat R^{-1}\widehat S\) are well defined.
Under the standard condition,
Zhang, Dong, and Meng~\cite{zhang2020backward} obtain this property from
\(R\ge\delta I_m\) and \(P\ge0\).
However, for the present
indefinite problem, $Q$, $G$, and $R$ need not be positive
semidefinite, $S$ may be nonzero, and \(P\) need not be nonnegative.
Moreover, $\widehat R$ depends on the unknown jump martingale component~$\zeta$; see \eqref{RShat}. Its uniform positive definiteness therefore cannot be inferred from
the weighting coefficients and must instead be derived from the
uniform convexity of the cost functional.
\end{itemize}

\subsection{Main results}

To address these difficulties, we construct the Riccati process \(P\)
directly from the stochastic value flow without using the global
representation \(P=\mathbf Y\mathbf X^{-1}\). For each stopping time,
we represent the conditional value by a quadratic form and aggregate
the corresponding family of matrices into a bounded c\`adl\`ag special
semimartingale \(P\). We then apply the uniform convexity condition to
spike controls supported on short time intervals and obtain
\(\widehat R\ge\varepsilon I_m\). Using the completion-of-squares
identity and the invertibility of the fundamental matrix on each
stochastic interval between successive Poisson jump times, we identify
the predictable finite-variation part of \(P\) with the drift term of
the SRE-J. Finally, estimates for the martingale components establish the
global square integrability of \((\Lambda,\zeta)\).

% Our proof proceeds in three steps.
% First, we establish a Hilbert space formulation that yields the
% quadratic representation of the cost functional and the
% characterization of open-loop optimality
% (\secref{sec:formulation}--\secref{sec:5open-loop}).
% Second, we construct the stochastic value flow and aggregate it into a
% bounded c\`adl\`ag special semimartingale via a stopping-time extension
% and the $\mathcal T$-system framework
% of~\cite{el1981aspects,zhang2020backward} (\secref{sec:Value Flow}).
% Third, we derive $\widehat R\ge\delta I_m$ by a small-interval localization
% argument and identify the Riccati drift on inter-jump intervals
% (\secref{sec:Riccati}).

% The main results and contributions are summarized as follows.
Our main results and contributions are summarized as follows.

\begin{enumerate}[label=(\roman*)]
  \item
To answer \textbf{(Q1)}, we prove that the uniform convexity of the
cost functional is sufficient for the SRE-J~\eqref{SREJ} to admit a
strongly regular solution \((P,\Lambda,\zeta)\) satisfying
\(\widehat R\ge\varepsilon I_m\).
Moreover, the constructed solution is maximal in the following sense:
if \((P',\Lambda',\zeta')\) is any strongly regular solution of
\eqref{SREJ}, then \(P'(t)\le P(t)\) for all \(t\in[0,T]\), a.s. Moreover, it is unique among the strongly regular solutions whose feedback coefficients generate admissible closed-loop controls.
% Consequently, it is the unique maximal strongly regular solution.
We further show that Problem (SLQ-J) is uniquely open-loop solvable and
that its unique open-loop optimal control admits the closed-loop
representation
\(
  u^*(s)
  =-\widehat R(s)^{-1}\widehat S(s)X^*(s-).
\)
See \thmref{Opensolve} and \thmref{thm:SREsolve}.
For \textbf{(Q2)}, \exaref{exam:intro_failure} gives a negative
answer: the representation \(P=\mathbf Y\mathbf X^{-1}\) may fail in
the presence of Poisson jumps. The value-flow construction described
above avoids this obstruction and requires neither global
invertibility of the optimal state fundamental matrix nor any
additional inverse-flow condition on \(I_n+E\).

\item 
To address the second difficulty, 
we derive the uniform positive definiteness of \(\widehat R\) by
applying the uniform convexity condition to spike controls supported on
short time intervals. 
Since a Poisson jump occurs with probability \(O(h)\), the fourth
moment of the state process \(X^i\) corresponding to such a spike
control is only \(O(h)\), rather than \(O(h^2)\) in the diffusion case.
We therefore decompose \(X^i=J^i+Y^i\), where \(J^i\) is the Poisson
integral of \(F\eta\) over the spike interval and \(Y^i\) is the
remainder. Combining the \(O(h^2)\) fourth-moment estimate for \(Y^i\)
with a separate estimate for \(J^i\), we obtain
\(\widehat R\ge\varepsilon I_m\); see
\lemref{lem:localized_error} and \lemref{lem:Rhat_unifpos}.

% Its key technical ingredient is a short-interval estimate for spike
% controls. Although the fourth-moment estimate for the spike-induced
% state \(X^i\) is only of order \(O(h)\), we write
% \(X^i=J^i+Y^i\), where \(J^i\) is the Poisson integral generated
% directly by the spike control, and establish an \(O(h^2)\)
% fourth-moment estimate for \(Y^i\).
% This estimate controls the error terms in the localization argument
% used to prove the uniform positive definiteness of \(\widehat R\);
% see \lemref{lem:localized_error}.

\item
We obtain verifiable sufficient conditions for uniform convexity and
present further examples beyond the standard condition. In the first
example of \secref{sec:example}, the terminal weight $G$ is indefinite and the control weight
 $R<0$. In the second example, we use the
Clark--Ocone formula and Malliavin derivatives on the Wiener--Poisson
space to construct an explicit random solution $(P,\Lambda,\zeta)$ of
the SRE-J, where $P$ is not positive definite and
$\Lambda\not\equiv0$, $\zeta\not\equiv0$.
\end{enumerate}

\subsection{Recent related work}
In parallel with the present work (the first version was posted as~\cite{ding2026indefinite}),
Ma and Meng~\cite{ma2026indefinite} independently studied the indefinite SLQ
problem with jumps under the same uniform convexity condition. Their
approach relies on the global inverse flow of the uncontrolled state equation under the additional
assumption $|\det(I_n+E)|\ge\delta_0>0$ and uses a single-time spike
variation with Lebesgue differentiation. In a follow-up, Ma, Liu, and
Meng~\cite{ma2026recursive} extended this framework to recursive cost
functionals and proved that the SRE-J admits a maximal strongly regular solution, which is unique among the strongly regular solutions whose feedback coefficients generate admissible closed-loop controls; this is consistent with our \thmref{thm:SREsolve}(ii).
In the applied direction, Fu et al.~\cite{fu2026darkpools} studied a multi-asset dark-pool liquidation
problem via a matrix BSDE with jumps, and proved that, for their model, this equation admits a unique solution in the class of nonnegative solutions; see also \remref{rem:why_maximal}.

Our approach differs from~\cite{ma2026indefinite,ma2026recursive} as
follows. First, we impose no invertibility condition on
$I_n+E$. 
In particular, \exaref{exam:IEsingular}, in which $I_n+E=0$ and $R<0$, is not covered by~\cite{ma2026indefinite}.
Our construction also allows the optimal closed-loop jump
matrix to be singular, as in~\exaref{exam:counterexample2}.
Second, to derive $\widehat R\ge\varepsilon I_m$, we apply spike
variations on stopped partitions and use the Doob--Meyer decomposition
to control the accumulated error.  This differs from the single-time
spike argument with Lebesgue differentiation
in~\cite{ma2026indefinite} and the auxiliary cost comparison
in~\cite{ma2026recursive}. 
Third, we require no a priori regularity of $P$ beyond its semimartingale structure.
In~\cite[Theorem~3.2]{ma2026indefinite}, the absolute continuity of the finite-variation part of $P$ and the square integrability of the martingale integrands are obtained in advance from the inverse-flow representation, and hence rely on the condition $|\det(I_n+E)|\ge\delta_0>0$.
In our approach, $\widehat R\ge\varepsilon I_m$ is derived when the finite-variation part $\mathscr A_P$ is only known to be predictable and $(\Lambda,\zeta)$ is only known to be locally square-integrable. Then, the absolute continuity of $\mathscr A_P$ and the global square integrability of $(\Lambda,\zeta)$ are deduced from the SRE-J itself (see \lemref{lem:SREidentify} and \lemref{lem:BMO}).

\subsection{Organization} 
The rest of the paper is organized as follows.
\secref{sec:formulation} collects notation, function spaces, and
preliminary estimates for FSDEs and BSDEs with jumps.
\secref{sec:5open-loop} derives a quadratic representation of the cost functional and establishes unique open-loop solvability under the uniform convexity condition. 
\secref{sec:Value Flow} constructs and aggregates the stochastic value
flow.
\secref{sec:Riccati} identifies the SRE-J and proves the closed-loop
representation.
\secref{sec:sufficient} gives sufficient conditions for uniform
convexity, \secref{sec:example} presents illustrative examples, and
\secref{sec:conclusion} concludes the paper with some open problems.

\section{Notation and Preliminaries}\label{sec:formulation}

Let $\mathcal{P}$ be the $\mathbb{F}$-predictable $\sigma$-field on $[0, T] \times \Omega$ and denote
$ \tilde{\mathcal{P}}:=\mathcal{B}(Z) \otimes \mathcal{P}.$ 
We use the following notation throughout:
\begin{itemize}
  \item $M^{\top}:$ the transpose of any matrix or vector $M$;
  % \item $|a|:=\sqrt{\sum_{i} a_{i}^{2}}$ for any vector $a=\left(a_{i}\right)$;
  \item $|\cdot|:$ the Euclidean norm for vectors and the Frobenius norm
  for matrices;
  % \item $|M|:=\sqrt{\sum_{i, j} m_{i j}^{2}}$ for any matrix or vector $M=\left(m_{i j}\right)$;
  % \item $\mathbb{R}^{m}: m$ dimensional real Euclidean space;
  \item $\mathbb{E}_t[\eta]:=\mathbb{E}[\eta\mid \mathcal{F}_t]$ for any integrable random variable $\eta$;
  \item $Y_-(s):=Y(s-)$: the left limit of a c\`{a}dl\`{a}g process $Y$ at $s\in(0,T]$;
  \item $\mathcal{T}[a,b]:$ the set of all $\mathbb{F}$-stopping times $\tau$ taking values in $[a,b]$;
  % \item $\mathcal{S}[a,b):$ the set of all $\mathbb{F}$-stopping times $\sigma$ taking values in $[a,b)$.
\end{itemize}
Next, for any $t \in[0, T)$ and Euclidean space $\mathbb{H}=\mathbb{R}^{n}, \mathbb{R}^{n \times m}$ or $\mathbb{S}^{n}$, we set
\begin{itemize}
  \item $L_{\mathcal F_t}^{\infty}(\Omega;\mathbb H)$ is the space of
  $\mathcal F_t$-measurable, essentially bounded $\mathbb H$-valued random variables;
  \item $L_{\mathcal F_t}^{2}(\Omega;\mathbb H)$ is the space of
  $\mathcal F_t$-measurable $\mathbb H$-valued random variables $\xi$ such that
  $\mathbb E[|\xi|^2]<\infty$;
  \item $S_{\mathbb F}^{2}(t,T;\mathbb H)$ is the space of
  $\mathbb F$-adapted c\`adl\`ag $\mathbb H$-valued processes $\phi$ such that
  \[
    \mathbb E\Big[\sup_{t\le s\le T}|\phi(s)|^2\Big]<\infty;
  \]
  \item $S_{\mathbb F}^{\infty}(t,T;\mathbb H)$ is the space of
  essentially bounded $\mathbb F$-adapted c\`adl\`ag $\mathbb H$-valued processes;
  \item $L_{\mathbb F}^{2}(t,T;\mathbb H)$ is the space of
  $\mathbb F$-predictable $\mathbb H$-valued processes $\phi$ such that
  \[
    \mathbb E\Big[\int_t^T |\phi(s)|^2\,\mathrm ds\Big]<\infty;
  \]
  \item $L_{\mathbb F}^{\infty}(t,T;\mathbb H)$ is the space of
  essentially bounded $\mathbb F$-predictable $\mathbb H$-valued processes;
  \item $G_{\mathbb F}^{2}(t,T,\lambda;\mathbb H)$ is the space of
  $\tilde{\mathcal P}$-measurable $\mathbb H$-valued processes $r$ such that
  \[
    \mathbb E\Big[\int_t^T\int_Z |r(s,z)|^2\lambda(\mathrm dz)\,\mathrm ds\Big]<\infty;
  \]
  \item $G_{\mathbb F}^{\infty}(t,T,\lambda;\mathbb H)$ is the space of
  essentially bounded $\tilde{\mathcal P}$-measurable $\mathbb H$-valued processes.
\end{itemize}
Moreover, for elements of $L_{\mathcal F_t}^{\infty}(\Omega;\mathbb H)$, $S_{\mathbb F}^{\infty}(t,T;\mathbb H)$ and $L_{\mathbb F}^{\infty}(t,T;\mathbb H)$, we use the 
notation $\|\cdot\|_\infty$ for the corresponding essential supremum norm if no confusion occurs.

For later use, the set of admissible initial pairs is denoted by
\(
  \mathcal D:=\bigl\{(t,\xi)\mid t\in[0,T],\;
  \xi\in L^2_{\mathcal F_t}(\Omega;\mathbb R^n)\bigr\}.
\)
We impose the following standing assumptions throughout the paper.
\begin{enumerate}[label=\bfseries (H\arabic*)]
  \item \label{SLQH1} The coefficients of the state equation \eqref{state} are all bounded processes, i.e.,
  \begin{equation*}
  \begin{aligned}
    &A,C\in L_{\mathbb F}^{\infty}(0,T;\mathbb R^{n\times n}),~~
    B,D\in L_{\mathbb F}^{\infty}(0,T;\mathbb R^{n\times m}),~~
    E\in G_{\mathbb F}^{\infty}(0,T,\lambda;\mathbb R^{n\times n}),~~
    F\in G_{\mathbb F}^{\infty}(0,T,\lambda;\mathbb R^{n\times m}).
  \end{aligned}
  \end{equation*}
  \item \label{SLQH2} The weighting coefficients of the cost functional \eqref{costL} are bounded, 
  \begin{equation*}
    G\in L_{\mathcal{F}_{T}}^{\infty}\left( \Omega ;\mathbb{S}^n \right) ,\quad Q\in L_{\mathbb{F}}^{\infty}\left( 0,T;\mathbb{S}^n \right) ,\quad S\in L_{\mathbb{F}}^{\infty}\left( 0,T;\mathbb{R}^{m\times n} \right) ,\quad R\in L_{\mathbb{F}}^{\infty}\left( 0,T;\mathbb{S}^m \right).
    \end{equation*}
\end{enumerate}
  Note that \ref{SLQH2} does not impose any definiteness or nonnegativity conditions on $Q$, $R$, or $G$.  Under \ref{SLQH1}--\ref{SLQH2}, for any $(t,\xi)\in\mathcal D$ and 
  $u\in\mathcal U[t,T]$, the random cost
\[
  L(t,\xi;u):=\langle GX(T),X(T)\rangle
  +\int_t^T\big[\langle QX,X\rangle
  +2\langle SX,u\rangle+\langle Ru,u\rangle\big]\,\mathrm ds
\]
is integrable, and $J(t,\xi;u)=\mathbb E[L(t,\xi;u)]$.
To handle random initial times (needed for the dynamic programming in \secref{sec:Value Flow}), we also consider the conditional
cost functional
  \[
    \widehat J(t,\xi;u) := \mathbb{E}[L(t,\xi;u)\mid\mathcal{F}_t],
    \qquad (t,\xi)\in\mathcal{D},\; u\in\mathcal{U}[t,T].
  \]
The corresponding control problem is formulated as follows.

 \textbf{Problem} $\widehat{(\textbf{SLQ-J})}:$ For any given initial
pair $(t,\xi)\in\mathcal D$, find a control
$u^*\in\mathcal U[t,T]$ such that
\begin{equation}\label{Jhatmin}
  \widehat J(t,\xi;u^*)
  =
  \operatorname*{ess\,inf}_{u\in\mathcal U[t,T]}
  \widehat J(t,\xi;u)
  \equiv \widehat V(t,\xi).
\end{equation}
The equivalence between Problems (SLQ-J) and 
$\widehat{(\text{SLQ-J})}$ will be established 
in \thmref{thm:SLQhat_equiv}.
Since the space $L_{\mathcal{F}_{t}}^{2}\left(\Omega ; \mathbb{R}^{n}\right)$ grows with the initial time $t$, we call $(t, \xi) \mapsto V(t, \xi)$ the \emph{value flow} of Problem (SLQ-J) and $(t, \xi) \mapsto \widehat{V}(t, \xi)$ the \emph{stochastic value flow} of Problem ($\widehat{\text{SLQ-J}}$).

Next, for $\Theta\in\boldsymbol{\Theta}[t, T]=L_{\mathbb{F}}^{2}\left(t, T ; \mathbb{R}^{m \times n}\right)$, 
the feedback control $ u(s)=\Theta(s)X(s-)$ leads to the closed-loop state equation
\begin{equation}\label{closestate}
\left\{
\begin{aligned}
  \dd X(s)&=(A+B\Theta)X(s-)\dd s+(C+D\Theta)X(s-)\dd W(s)+\int_{Z}(E+F\Theta)X(s-)\,\widetilde{N}(\dd s,\dd z),\\
  X(t)&=\xi.
\end{aligned}
\right.
\end{equation}

\begin{defn}[Open-loop optimal control and closed-loop representation]\label{def:open}
  \begin{enumerate}[label=(\roman*)]
    \item For a given initial pair $(t,\xi)\in\mathcal D$, a control
      $u^*(\cdot) \equiv u^*(\cdot; t, \xi) \in \mathcal{U}[t, T]$ is called an open-loop optimal control of Problem (SLQ-J) (respectively, {Problem} $(\widehat{SLQ-J})$) if
    \[
    J(t, \xi ; u^*) \le J(t, \xi ; u), \quad \forall u \in \mathcal{U}[t, T]\quad (\text {respectively},~ \widehat{J}\left(t, \xi ;u^*\right) \leq \widehat{J}(t, \xi ; u), \text { a.s. }~\forall u \in \mathcal{U}[t, T] )
    \]
    \item An open-loop optimal control $u^*$ is said to admit a closed-loop
        representation if there exists
        $\Theta^*\in\boldsymbol{\Theta}[t,T]$ such that
        \[
          u^*(s)=\Theta^*(s) X^*(s-),
          \qquad \text{a.e. }s\in[t,T],\ \text{a.s.},
        \]
        where $X^*$ is the solution to the closed-loop system \eqref{closestate}  with $\Theta=\Theta^*$.
  \end{enumerate}
\end{defn}

  We further impose the following uniform convexity condition.
  \begin{enumerate}[label=\bfseries (H3)]
  \item \label{N_0uniformalconvex2-I} There exists $\varepsilon>0$ such that
  \begin{equation*}
    J(0,0 ; u) \ge \varepsilon \mathbb{E} \int_{0}^{T}|u(s)|^{2} \mathrm{d} s, 
    \quad \forall u \in \mathcal{U}[0, T].
  \end{equation*}
  \end{enumerate}
  This condition is strictly weaker than the standard 
  condition; see \secref{sec:sufficient} for sufficient conditions and 
  \exaref{exam:Gneg}, \exaref{exam:IEsingular}, and \secref{sec:example} for examples.

We first recall standard well-posedness estimates for linear forward SDEs (FSDEs) and backward SDEs (BSDEs) with random coefficients. Consider the linear FSDE
\begin{equation}\label{FSDE}
\begin{cases}
	\mathrm{d}X(s)=[A(s)X(s-)+b(s)]\mathrm{d}s+[C(s)X(s-)+\sigma (s)]\mathrm{d}W(s)\\
  \hspace{4.2em}+\int_Z{\left[ E\left( s,z \right) X\left( s- \right) +l\left( s,z \right) \right] \widetilde{N}\left( \mathrm{d}s,\mathrm{d}z \right)},\quad s\in [t,T],\\
	X(t)=\xi,
\end{cases}
\end{equation}
and the linear BSDE
\begin{equation}\label{BSDE}
\begin{cases}
\mathrm{d}Y(s)=-\left[
  A(s)^\top Y(s)+C(s)^\top Z(s)
  +\int_Z E(s,z)^\top r(s,z)\lambda(\mathrm dz)+\varphi(s)
  \right]\mathrm ds\\
	\hspace{4.2em}+Z(s) \mathrm{d}W(s) +\int_Z{r\left( s,z \right) \widetilde{N}\left( \mathrm{d}s,\mathrm{d}z \right)},\quad s\in [t,T],\\
	Y\left( T \right) =\eta.
\end{cases}
\end{equation}
We have the following result.
\begin{lem} \label{existunique}
  Let  $A,C\in L_{\mathbb{F}}^{\infty}(t, T ; \mathbb{R}^{n\times n})$ and $E\in G_{\mathbb{F}}^{\infty}\left( t,T,\lambda ;\mathbb{R}^{n\times n} \right) $. Then the following assertions hold:
 \begin{enumerate}[label=(\roman*)]
   \item For every initial pair $(t, \xi) \in \mathcal{D}$ and $b, \sigma \in L_{\mathbb{F}}^{2}\left(t, T ; \mathbb{R}^{n}\right), l\in G_{\mathbb{F}}^{2}\left( t,T,\lambda ;\mathbb{R}^{n} \right)  $, FSDE \eqref{FSDE} admits a unique adapted solution $X \in S_{\mathbb{F}}^{2}(t,T;\mathbb{R}^{n})$. In addition, there is a positive constant $K$, depending only on $A, C,E$ and $T$, such that
   \begin{equation}\label{FSDEestimate}
   \mathbb{E}\left[ \mathop {\mathrm{sup}} \limits_{t\le s\le T}|X(s)|^2 \right] \le K\mathbb{E}\left[ |\xi |^2+\int_t^T{|}b(s)|^2\mathrm{d}s+\int_t^T{|}\sigma (s)|^2\mathrm{d}s+\int_t^T{\int_Z{\left| l\left( s,z \right) \right|^2\lambda \left( \mathrm{d}z \right)}}\mathrm{d}s \right]. 
   \end{equation}

   \item For every terminal state $\eta \in L_{\mathcal{F}_{T}}^{2}\left(\Omega ; \mathbb{R}^{n}\right)$ and  $\varphi \in L_{\mathbb{F}}^{2}\left(t, T ; \mathbb{R}^{n}\right)$, BSDE \eqref{BSDE} admits a unique adapted solution $(Y,Z,r)\in S_{\mathbb{F}}^{2}(t,T;\mathbb{R}^{n})\times L_{\mathbb{F}}^{2}(t,T;\mathbb{R}^{n})\times G_{\mathbb{F}}^{2}\left( t,T,\lambda ;\mathbb{R}^{n} \right)$. In addition, there is a positive constant $K$, depending only on $A, C,E$ and $T$, such that
   \begin{equation}\label{BSDEestimate}
   \mathbb{E}\left[ \mathop {\mathrm{sup}} \limits_{t\le s\le T}|Y(s)|^2+\int_t^T{|}Z(s)|^2\mathrm{d}s+\int_t^T{\int_Z{|r\left( s,z \right) |^2}\lambda \left( \mathrm{d}z \right)}\mathrm{d}s \right] \le K\mathbb{E}\left[ |\eta |^2+\left( \int_t^T{|}\varphi (s)|\mathrm{d}s \right) ^2 \right]. 
  \end{equation}  
 \end{enumerate}
\end{lem}

\begin{proof}
  The proof of the existence and uniqueness can be found in \cite{li2009stochastic,tang1994necessary}. 
  The estimates \eqref{FSDEestimate} and \eqref{BSDEestimate} follow from
  It\^o's formula, the Burkholder--Davis--Gundy inequality, and Gronwall's inequality.
\end{proof}

For later use, we introduce the following Lyapunov BSDE for
$\mathbb S^n$-valued processes on $[0,T]$:
\begin{equation}\label{Lyapunov}
\begin{cases}
  \mathrm dM(s)=-\Big[
    M(s)A(s)+A(s)^\top M(s)+C(s)^\top M(s)C(s)
    +\Sigma(s)C(s)+C(s)^\top\Sigma(s)+Q(s) \\
  \hspace{5em}
    +\int_Z[
      \Gamma(s,z)E(s,z)+E(s,z)^\top\Gamma(s,z)
      +E(s,z)^\top\big(\Gamma(s,z)+M(s)\big)E(s,z)
    ]\lambda(\mathrm dz)
  \Big]\mathrm ds\\
  \hspace{4em}
  +\Sigma(s)\mathrm dW(s)
  +\int_Z\Gamma(s,z)\widetilde N(\mathrm ds,\mathrm dz),\\
  M(T)=G.
\end{cases}
\end{equation}
By \lemref{existunique}, BSDE
\eqref{Lyapunov} admits a unique solution
\(
  (M,\Sigma,\Gamma)\in
  S_{\mathbb F}^{2}(0,T;\mathbb S^n)
  \times L_{\mathbb F}^{2}(0,T;\mathbb S^n)
  \times G_{\mathbb F}^{2}(0,T,\lambda;\mathbb S^n)
\) under \ref{SLQH1}--\ref{SLQH2}.
The following proposition shows that the first component $M$ is in fact
bounded. Throughout the paper, $K$ denotes a positive constant which may vary
from line to line.

% By \lemref{existunique}, BSDE \eqref{Lyapunov} admits a unique solution $(M,\Sigma,\Gamma)\in S_{\mathbb{F}}^{2}(t,T;\mathbb{R}^{n})\times L_{\mathbb{F}}^{2}(t,T;\mathbb{R}^{n})\times G_{\mathbb{F}}^{2}\left( t,T,\lambda ;\mathbb{R}^{n} \right)$ under \ref{SLQH1}-\ref{SLQH2}. Furthermore, $M$ is uniformly bounded, as shown by the following proposition. Throughout the paper, $K$ denotes a positive constant that may vary from line to line.

\begin{prop}\label{prop:Mbouned}
  Under \ref{SLQH1}-\ref{SLQH2}, the first component $M$ of the adapted solution $(M,\Sigma,\Gamma)$ to the BSDE \eqref{Lyapunov} is bounded.
\end{prop}

% \begin{proof}
%   The proof uses It\^{o}'s formula applied to
%   $e^{\beta s}|M(s)|^2$ together with Gronwall's inequality.
%   % it follows the same lines as \cite[Proposition~5.1]{sun2021indefinite}
%   % with the additional jump terms handled by the finite-activity
%   % assumption $\lambda(Z)<\infty$.
%   The complete argument is given in the arXiv version
%   of this paper.
% \end{proof}

\begin{proof}
Denote the drift of \eqref{Lyapunov} as $-\Pi(s)\mathrm{d}s$. Since
$A,C,E,Q$ are bounded, there exists $K>0$ such that
\(
|\Pi (s) |\le 
K\left( 1+|M(s) |+|\Sigma(s) |+\int_Z{|\Gamma(s,z) |\lambda \left( \mathrm{d}z \right)} \right). 
\)
By the Cauchy-Schwarz inequality, we have
\[
  2\langle M(s),\Pi(s)\rangle
  \le K(1+|M(s)|^2)
     +|\Sigma(s)|^2
     +\int_Z|\Gamma(s,z)|^2\lambda(\mathrm dz).
\]
Fix $t\in[0,T]$ and let $\beta>0$. Applying It\^o's formula to
$e^{\beta s}|M(s)|^2$ on $[t,T]$ gives
\[
\begin{aligned}
	e^{\beta t}|M(t)|^2=&e^{\beta T}|G|^2+\int_t^T{e^{\beta s}\Big[ -\beta |M(s)|^2+2\langle M(s),\Pi (s)\rangle -|\Sigma (s)|^2-\int_Z{|\Gamma \left( s,z \right) |^2\lambda \left( \mathrm{d}z \right)} \Big] \mathrm{d}s}\\
	&-2\int_t^T{e^{\beta s}\langle M(s),\Sigma (s)\rangle \mathrm{d}W(s)}-\int_t^T{e^{\beta s}\int_Z{[2\langle M(s-),\Gamma (s,z)\rangle +|\Gamma \left( s,z \right) |^2]\widetilde{N}\left( \mathrm{d}s,\mathrm{d}z \right)}}\\
\end{aligned}
\]
Choosing $\beta>K$ and taking the conditional expectations with respect to $\mathcal{F}_{t}$, we obtain
\[
  |M(t)|^2\le e^{\beta t}|M(t)|^2
  \le
  \mathbb E\Big[
    e^{\beta T}|G|^2+K\int_t^T e^{\beta s}\,\mathrm ds
    \,\Big|\,\mathcal F_t
  \Big].
\]
Since $G\in L^\infty_{\mathcal F_T}(\Omega;\mathbb S^n)$, the
right-hand side is bounded uniformly in $t$. Hence $M$ is bounded.
\end{proof}

We conclude this section with a brief account of the $\mathcal T$-system aggregation framework, which will be used in \secref{sec:Value Flow}. The following definitions and results are taken from El~Karoui~\cite{el1981aspects} and Zhang, Dong, and Meng~\cite[Proposition~2.3]{zhang2020backward}.

% [$\mathcal T$-system and RCE; {\cite{el1981aspects,zhang2020backward}}]
\begin{defn}\label{def:Tsystem}
A $\mathcal T$-system is a family $\{\mu(\tau)\}_{\tau\in\mathcal T[0,T]}$ such that
$\mu(\tau)$ is $\mathcal F_\tau$-measurable for every $\tau\in\mathcal T[0,T]$, and
\[
  \mu(\tau)=\mu(\sigma)\quad\text{a.s.\ on }\{\tau=\sigma\},
  \qquad \forall\,\tau,\sigma\in\mathcal T[0,T].
\]
We say that the $\mathcal T$-system is of class D if the family $\{\mu(\tau):\tau\in\mathcal T[0,T]\}$ is uniformly integrable.
The $\mathcal T$-system is called a \emph{submartingale system} if $\mu(\sigma)\le\mathbb E[\mu(\tau)\mid\mathcal F_\sigma]$ a.s.\ for all $\sigma\le\tau$ in $\mathcal T[0,T]$, and it is said to satisfy \emph{right-continuity in expectation} (RCE) if, for every sequence of stopping times $\tau_n\downarrow\tau$ a.s., one has $\mathbb{E}[\mu(\tau_n)]\to\mathbb{E}[\mu(\tau)]$.
\end{defn}

% A \emph{$\mathcal T$-system} is a family $\{\mu(\tau)\}_{\tau\in\mathcal T[0,T]}$ of integrable random variables satisfying $\mu(\tau)=\mu(\sigma)$ a.s.\ on $\{\tau=\sigma\}$ for all $\tau,\sigma\in\mathcal T[0,T]$. The $\mathcal T$-system is said to satisfy \emph{right-continuity in expectation} (RCE) if, for every sequence of stopping times $\tau_n\downarrow\tau$ a.s., one has $\mathbb E[\mu(\tau_n)]\to \mathbb E[\mu(\tau)]$. The following aggregation result is fundamental.

\begin{prop}\label{prop:RCE}
  Let $\{\mu(\tau)\}_{\tau\in\mathcal T[0,T]}$ be a submartingale $\mathcal T$-system of class $D$ satisfying the RCE condition. Then there exists a unique (up to indistinguishability) c\`{a}dl\`{a}g adapted process $\bar\mu$ such that $\bar\mu(\tau)=\mu(\tau)$ a.s.\ for every $\tau\in\mathcal T[0,T]$. Moreover, $\bar\mu$ is a c\`{a}dl\`{a}g submartingale of class D and admits the Doob--Meyer decomposition.
\end{prop}

% In \secref{sec:Value Flow} we will show that the zero-control value system associated with the SLQ-J problem forms a class $D$ submartingale $\mathcal T$-system satisfying the RCE. \proref{prop:RCE} then produces the bounded c\`{a}dl\`{a}g aggregation of the value process~$P$; see \proref{prop:Psemimartingale} for the precise statement.

The same statement holds for $\mathcal T[\rho,T]$-systems with $\rho\in\mathcal T[0,T]$.
We shall also use the following elementary fact.

\begin{lem}\label{lem:FVsubmart}
Let $\mathcal{V}$ be a predictable c\`adl\`ag process with $\mathcal{V}(0)=0$ whose total variation on $[0,T]$ is integrable. If $\mathcal{V}$ is a submartingale, then $\mathcal{V}$ is nondecreasing on $[0,T]$, a.s.
\end{lem}

\begin{proof}
Let $\mathcal{V}=\mathcal{M}+\mathcal{A}$ be its Doob--Meyer decomposition, where $\mathcal{M}$ is a
martingale and $\mathcal{A}$ is predictable and nondecreasing, with
$\mathcal{M}(0)=\mathcal{A}(0)=0$.
Then $\mathcal{M}=\mathcal{V}-\mathcal{A}$ is a predictable martingale of finite variation,
so $\mathcal{M}\equiv0$ by
\cite[Corollary~I.3.16]{jacod2003limit}.
Hence $\mathcal{V}=\mathcal{A}$ is nondecreasing.
\end{proof}

\section{Open-Loop Solvability of Problem (SLQ-J) and \texorpdfstring{$\widehat{\textbf{(SLQ-J)}}$}{(SLQ-J)hat}}\label{sec:5open-loop}

\subsection{A Hilbert Space Point of View}\label{sec:hilbert}

This subsection derives a quadratic representation of the cost functional.
The representation identifies the operator that determines convexity in the
control variable and will be used later in the construction of the stochastic
value flow. For $u,v\in\mathcal U[t,T]$, set
\[
  \llbracket u,v\rrbracket
  :=\mathbb E\int_t^T\langle u(s),v(s)\rangle\,\mathrm ds.
\]

For $(t,\xi)\in\mathcal D$ and $u\in\mathcal U[t,T]$, let $(X,Y,Z,r)$ be the
adapted solution of the following decoupled linear FBSDE:
\begin{equation}\label{FBSDE1}
  \begin{cases}
    \mathrm{d}X(s) =\left[ A(s) X\left( s- \right) +B(s) u(s) \right] \mathrm{d}s+\left[ C(s) X\left( s- \right) +D(s) u(s) \right] \mathrm{d}W(s)\\
    \hspace{4em}+\int_Z{\left[ E\left( s,z \right) X\left( s- \right) +F\left( s,z \right) u(s) \right] \widetilde{N}\left( \mathrm{d}s,\mathrm{d}z \right)},\\[0.8em]
    \mathrm{d}Y(s) =-\left[ A(s) ^{\top}Y(s) +C(s) ^{\top}Z(s) +Q(s) X(s) +S(s)^\top u(s) +\int_Z{E\left( s,z \right) ^{\top}r\left( s,z \right) \lambda \left( \mathrm{d}z \right)} \right] \mathrm{d}s\\
    \hspace{4em}+Z(s) \mathrm{d}W(s) +\int_Z{r\left( s,z \right) \widetilde{N}\left( \mathrm{d}s,\mathrm{d}z \right)},\quad s\in[t,T].\\[0.8em]
    X\left( t \right) =\xi ,\quad  Y\left( T \right) =GX\left( T \right).\\
  \end{cases}
\end{equation}
% Note that the backward part of \eqref{FBSDE1} is precisely the adjoint BSDE introduced in \secref{sec:3preli}.
By \lemref{existunique}, \eqref{FBSDE1} admits a unique solution in
$
S_{\mathbb F}^2(t,T;\mathbb R^n)\times S_{\mathbb F}^2(t,T;\mathbb R^n)\times L_{\mathbb F}^2(t,T;\mathbb R^n)\times G_{\mathbb F}^2(t,T,\lambda;\mathbb R^n).
$
Furthermore, applying It\^o's formula to
$s\mapsto\langle Y(s),X(s)\rangle$ and taking conditional expectation with
respect to $\mathcal F_t$, we obtain the following representation of
$\widehat J(t,\xi;u)$.
\begin{equation}\label{rep1}
\begin{aligned}
\widehat J(t,\xi;u)
=&\,\langle Y(t),\xi\rangle
+\mathbb E\bigg[\int_t^T
\big\langle B^\top Y+D^\top Z+SX+Ru
+\int_ZF^\top r\lambda(\mathrm dz),\,u\big\rangle
\mathrm ds\,\bigg|\,\mathcal F_t\bigg].
\end{aligned}
\end{equation}

By the linearity of \eqref{FBSDE1}, we write its solution as
\(
  (X,Y,Z,r)
  =
  (\widetilde X,\widetilde Y,\widetilde Z,\widetilde r)
  +(\bar X,\bar Y,\bar Z,\bar r),
\)
where $(\widetilde X,\widetilde Y,\widetilde Z,\widetilde r)$ solves
\begin{equation}\label{Xtilde}
 \begin{cases}
  \mathrm{d}\widetilde{X}(s)=(A\widetilde{X}+Bu)\mathrm{d} s+(C\widetilde{X}+Du)\mathrm{d} W(s)+\int_Z{(E\widetilde{X}_{-}+Fu)}\widetilde{N}\left( \mathrm{d}s,\mathrm{d}z \right),
  \\
  \mathrm{d}\widetilde{Y}(s)=-\left( A^{\top}\widetilde{Y}+C^{\top}\widetilde{Z}+Q\widetilde{X}+S^{\top}u+\int_Z{E^\top \widetilde{r}\lambda \left( \mathrm{d}z \right)} \right) \mathrm{d} s+\widetilde{Z}\mathrm{d} W(s)+\int_Z\widetilde{r}\widetilde{N}\left( \mathrm{d}s,\mathrm{d}z \right),
  \\
  \widetilde{X}(t)=0,\quad \widetilde{Y}(T)=G\widetilde{X}(T),
\end{cases} 
\end{equation}  
and $(\bar X,\bar Y,\bar Z,\bar r)$ solves
\begin{equation}\label{Xbar}
\begin{cases}
  \mathrm{d}\bar{X}(s)=A\bar{X}\mathrm{d}s+C\bar{X}\mathrm{d}W(s)+\int_Z{E\bar{X}_{-}}\widetilde{N}\left( \mathrm{d}s,\mathrm{d}z \right), 
  \\
  \mathrm{d}\bar{Y}(s)=-\left( A^{\top}\bar{Y}+C^{\top}\bar{Z}+Q\bar{X}+\int_Z{E^\top \bar{r}\lambda \left( \mathrm{d}z \right)} \right) \mathrm{d}s+\bar{Z}\mathrm{d}W(s)+\int_Z{\bar{r}}\widetilde{N}\left( \mathrm{d}s,\mathrm{d}z \right),
  \\
  \bar{X}(t)=\xi,\quad \bar{Y}(T)=G\bar{X}(T).
\end{cases} 
\end{equation} 

Next, we introduce the following three operators:
\[
\mathcal{N}_t:\mathcal{U}[t,T]\rightarrow \mathcal{U}[t,T],\quad \mathcal{L}_t:L_{\mathcal{F}_t}^{2}(\Omega ;\mathbb{R}^n)\rightarrow \mathcal{U}[t,T],\quad M\left( t \right) :L_{\mathcal{F}_t}^{2}(\Omega ;\mathbb{R}^n)\rightarrow L_{\mathcal{F}_t}^{2}(\Omega ;\mathbb{R}^n).
\]
Specifically, for any $u\in \mathcal{U}[t, T]$,
\begin{equation}\label{Nt}
  \left[ \mathcal{N}_tu \right] (s):=B(s)^{\top}\widetilde{Y}(s)+D(s)^{\top}\widetilde{Z}(s)+S(s)\widetilde{X}(s)+\int_Z{F \left( s,z \right)^\top \widetilde{r}\left( s,z \right) \lambda (\mathrm{d}z)}+R(s)u(s),\quad s\in [t,T]
\end{equation}
for any $\xi\in L_{\mathcal{F}_t}^{2}(\Omega ;\mathbb{R}^n)$,
\begin{equation}\label{Lt}
\left[ \mathcal{L}_t\xi \right] (s):=B(s)^{\top}\bar{Y}(s)+D(s)^{\top}\bar{Z}(s)+S(s)\bar{X}(s)+\int_Z{F\left( s,z \right)^\top  \bar{r}\left( s,z \right) \lambda (\mathrm{d}z)},\quad s\in [t,T]
\end{equation}
In addition, the operator $M(t)$ is induced by the first component of the
matrix-valued Lyapunov BSDE \eqref{Lyapunov}. Indeed, applying the jump
It\^o's formula to $M(s)\bar X(s)$, where $M$ satisfies \eqref{Lyapunov}, we obtain
\[
  \bar Y(s)=M(s)\bar X(s),\qquad s\in[t,T].
\]
Therefore, since $\bar X(t)=\xi$, we define the operator $M(t)$ by
\begin{equation}\label{Mt}
  M(t)\xi:=M(t)\bar X(t)=\bar Y(t),
  \qquad \xi\in L_{\mathcal F_t}^{2}(\Omega;\mathbb R^n).
\end{equation}

For the operators $\mathcal N_t$, $\mathcal L_t$, and $M(t)$ introduced above,
we have the following result. 
% It is the jump-diffusion counterpart of
% \cite[Proposition~3.2]{sun2021indefinite}.
The self-adjointness of
$\mathcal N_t$ is proved by the duality method as in
\cite[Proposition~3.2]{sun2021indefinite}.
The boundedness of
$\mathcal N_t$, $\mathcal L_t$, and $M(t)$ follows from the estimates
for the linear FSDE and BSDE with jumps in \lemref{existunique}. 
% Hence we omit
% the details here and refer the interested reader to the arXiv version \cite{ding2026indefinite} of this paper for a complete proof.

% The subsequent theorem gives the quadratic representation of the cost functional.
% \begin{itemize}
% \item $\mathcal{N}_{t}$ is a bounded self-adjoint operator from $\mathcal{U}[t, T]$ into itself;
% \item $\mathcal{L}_{t}$ is bounded operator from $L_{\mathcal{F}_t}^{2}(\Omega ;\mathbb{R}^n)$ into $\mathcal{U}[t, T]$;
% \item $M(t)$ is an $n \times n$ symmetric matrix, i.e, $M(t) \in \mathbb{S}^{n}$
% \end{itemize}
% Finally,  the cost functional can be represented as 
% \[
% J(t, \xi ; u)=\llbracket \mathcal{N}_{t} u, u  \rrbracket +2\llbracket\mathcal{L}_{t} \xi, u  \rrbracket + \mathbb{E}\left< M(t) \xi, \xi \right> \quad \forall(t, \xi)\in\mathcal{D} 
% \]
% We now state and prove these properties precisely.

\begin{prop}\label{prop:NtLtbounded}
Under \ref{SLQH1}-\ref{SLQH2}, the following results hold:
\begin{enumerate}[label=(\roman*)]
  \item $\mathcal{N}_{t}$ defined by \eqref{Nt} is a bounded self-adjoint operator on $\mathcal{U}[t, T]$.
  \item  $\mathcal{L}_{t}$ defined by \eqref{Lt} is a bounded operator from  $L_{\mathcal{F}_t}^{2}(\Omega ;\mathbb{R}^n)$ into $\mathcal{U}[t, T]$. Moreover, there exists some positive constant $K$, independent of $(t, \xi)$, such that
  \[
  \llbracket\mathcal{L}_{t} \xi, \mathcal{L}_{t} \xi\rrbracket \le K \mathbb{E}|\xi|^{2}, \quad \forall \xi \in L_{\mathcal{F}_t}^{2}(\Omega ;\mathbb{R}^n) .
  \]
  \item $M(t)$ defined by \eqref{Mt} is a bounded operator on $L_{\mathcal F_t}^2(\Omega;\mathbb R^n)$.
\end{enumerate}
\end{prop}

% \begin{proof}
%   The self-adjointness of $\mathcal N_t$ follows from applying
%   It\^{o}'s formula to
%   $\langle\widetilde Y_2(s),\widetilde X_1(s)\rangle$
%   and using the symmetry of $G$ and $Q$;
%   the boundedness of $\mathcal L_t$ and $M(t)$ follows from the BSDE
%   estimates in \lemref{existunique}. 
%   See the arXiv version \cite{ding2026indefinite} for details.
%   % Both arguments are identical to
%   % \cite[Proposition~4.2]{sun2021indefinite} with the additional
%   % jump integrals
% \end{proof}

% %%%完整证明
\begin{proof}
  (i) The boundedness of $\mathcal N_t$ follows from \lemref{existunique} and the boundedness of the coefficients. For $u_1,u_2\in\mathcal U[t,T]$, let $(\widetilde X_j,\widetilde Y_j,\widetilde Z_j,\widetilde r_j)$ be the adapted solution to \eqref{Xtilde} with $u=u_j$, $j=1,2$. Applying It\^o's formula to $s\mapsto\langle\widetilde Y_1(s),\widetilde X_2(s)\rangle$ and taking expectations, we obtain
\[
\begin{aligned}
  \llbracket\mathcal N_tu_1,u_2\rrbracket
  =\mathbb E\bigg[&\langle G\widetilde X_1(T),\widetilde X_2(T)\rangle+\int_t^T\Big(\langle Q\widetilde X_1,\widetilde X_2\rangle
  +\langle S\widetilde X_1,u_2\rangle+\langle S\widetilde X_2,u_1\rangle
  +\langle Ru_1,u_2\rangle\Big)\mathrm ds\bigg].
\end{aligned}
\]
Since $G$, $Q$, and $R$ are symmetric, the right-hand side is symmetric in $(u_1,u_2)$. Hence $\mathcal N_t$ is self-adjoint.

  (ii) Choose a constant $\alpha_1$ such that 
\begin{equation}\label{coeffbound}
  |G|^{2},|B(s)|^{2},|D(s)|^{2},|S(s)|^{2},|Q(s)|^{2},|F(s,z)|^2 \le \alpha_1 \quad \text { a.e. } s \in[0, T], \text { a.e. } z\in Z, \text { a.s. }
\end{equation}
  Therefore,
  \begin{equation}\label{Lestimate1}
  \begin{aligned}
    \llbracket\mathcal{L}_{t} \xi, \mathcal{L}_{t} \xi\rrbracket =&~\mathbb{E}\left[ \int_t^T{\left| B(s)^{\top}\bar{Y}(s)+D(s)^{\top}\bar{Z}(s)+S(s)\bar{X}(s)+\int_Z{F\left( s,z \right) ^{\top}\bar{r}\left( s,z \right) \lambda (\mathrm{d}z)} \right|^2\mathrm{d}s} \right] 
  \\
  \le&~ 4\alpha _1\mathbb{E}\left[ \int_t^T{\left| \bar{Y}(s) \right|^2+\left| \bar{Z}(s) \right|^2+\left| \bar{X}(s) \right|^2+\int_Z{\left| \bar{r}\left( s,z \right) \right|^2\lambda (\mathrm{d}z)}\mathrm{d}s} \right] 
  \end{aligned}
 \end{equation}
  By \lemref{existunique} again, there exists a constant $\alpha_2>0$, independent of $t$ and $\xi$ , such that
  \begin{align}
    \mathbb{E}\left[ \int_t^T{\left| \bar{Y}(s) \right|^2+\left| \bar{Z}(s) \right|^2+\int_Z{\left| \bar{r}\left( s,z \right) \right|^2\lambda (\mathrm{d}z)}\mathrm{d}s} \right] &\le \alpha _2\mathbb{E}\left[ |G\bar{X}(T)|^2+\int_t^T{|}Q(s)\bar{X}(s)|^2\mathrm{d}s \right],\label{Lestimate2} \\
    \mathbb{E}|\bar{X}(T)|^2+\mathbb{E}\int_t^T{|}\bar{X}(s)|^2\mathrm{d}s&\le \alpha _2\mathbb{E}|\xi |^2.\label{Lestimate3}
  \end{align}
  Substituting \eqref{Lestimate2} and \eqref{Lestimate3} into \eqref{Lestimate1} and making use of \eqref{coeffbound}, we have
  \[
  \mathbb{E}\left[ \int_t^T{\left| \bar{Y}(s) \right|^2+\left| \bar{Z}(s) \right|^2+\int_Z{\left| \bar{r}\left( s,z \right) \right|^2\lambda (\mathrm{d}z)}\mathrm{d}s} \right] \le \alpha _1\alpha _{2}^{2}\mathbb{E}|\xi |^2
  \]
  which implies that $\llbracket\mathcal{L}_{t} \xi, \mathcal{L}_{t} \xi\rrbracket \le 4\alpha _1(\alpha _1\alpha _{2}^{2}+\alpha _2)\mathbb{E}|\xi |^2$ for all $\xi\in L_{\mathcal{F}_t}^{2}(\Omega ;\mathbb{R}^n)$.

  (iii) By \eqref{Mt}, the estimates \eqref{FSDEestimate} and \eqref{BSDEestimate}, and the boundedness of $G$ and $Q$, we have
  \[
    \mathbb E[|M(t)\xi|^2]
    =\mathbb E[|\bar Y(t)|^2]
    \le C\mathbb E\bigg[|G\bar X(T)|^2
      +\bigg(\int_t^T|Q(s)\bar X(s)|\,\mathrm ds\bigg)^2\bigg]
    \le C\mathbb E|\xi|^2.
  \]
  Hence, $M(t)$ is a bounded operator on $L_{\mathcal F_t}^2(\Omega;\mathbb R^n)$.
\end{proof}

We are now ready to present the functional representation of the cost
functional.

\begin{thm}\label{thm:Jquadratic}
  Under \ref{SLQH1}-\ref{SLQH2}, for any initial pair $(t, \xi)\in\mathcal{D}$  and control $u \in$ $\mathcal{U}[t, T]$, we have the following representation of the cost functional, 
  \begin{equation}\label{costrep}
    J(t, \xi ; u)=\llbracket \mathcal{N}_{t} u, u  \rrbracket +2\llbracket\mathcal{L}_{t} \xi, u  \rrbracket + \mathbb E\langle M(t)\xi,\xi\rangle, \quad \forall(t, \xi)\in\mathcal{D}, 
    \end{equation}
  where $\mathcal N_t$, $\mathcal L_t$, and $M(t)$ are defined by
  \eqref{Nt}, \eqref{Lt}, and \eqref{Mt}, respectively.
\end{thm}

\begin{proof}
  
By the representation \eqref{rep1}, together with the operator $\mathcal{N}_{t}$, $\mathcal{L}_{t}$ and $M(t)$, we have
\begin{equation}\label{identity}
\begin{aligned}
	J(t,\xi;u)=&\mathbb{E}\Big[ \langle \widetilde{Y}(t),\xi\rangle +\langle \bar{Y}(t),\xi\rangle +\int_t^T{\langle \left[ \mathcal{N}_tu \right] (s)+\left[ \mathcal{L}_t\xi \right] (s),u(s)}\rangle \mathrm{d}s \Big]\\
	=&\mathbb{E}[ \langle \widetilde{Y}(t),\xi\rangle ]+  \mathbb{E}\langle M(t) \xi, \xi  \rangle + \llbracket \mathcal{N}_{t} u, u  \rrbracket +\llbracket\mathcal{L}_{t} \xi, u  \rrbracket.
\end{aligned}
\end{equation}
It remains to identify the first term $\mathbb{E}\big[ \langle \widetilde{Y}(t),\xi\rangle \big] $. Applying It\^{o}'s formula to $s \mapsto\langle\widetilde{Y}(s), \bar{X}(s)\rangle$ gives
\begin{equation}\label{tilbar}
  \mathbb{E}\big[ \langle \widetilde{Y}(t),\xi\rangle \big] =\mathbb{E}\Big[ \left< G\widetilde{X}(T),\bar{X}(T) \right> +\int_t^T{\big( \langle S\bar{X},u\rangle +\langle Q\widetilde{X},\bar{X}\rangle \big) \mathrm{d}s} \Big], 
\end{equation}
and applying It\^{o}'s formula to $s \mapsto\langle\bar{Y}(s), \widetilde{X}(s)\rangle$ gives
\begin{equation}\label{bartil}
  0=\mathbb{E}\left[ \langle G\bar{X}(T),\widetilde{X}(T)\rangle \right] -\mathbb{E}\Big[\int_t^T{\Big(\langle B^{\top}\bar{Y}+D^{\top}\bar{Z}+\int_Z{F^\top \bar{r}\lambda (\mathrm{d}z)},u  \rangle -\langle Q\widetilde{X},\bar{X}\rangle \Big)}\mathrm{d}s\Big]
\end{equation}
Since $G$ is symmetric, it follows from \eqref{tilbar} and \eqref{bartil} that
\[
\mathbb{E}\big[ \langle \widetilde{Y}(t),\xi\rangle \big] =\mathbb{E}\Big[\int_t^T{\langle B^{\top}\bar{Y}+D^{\top}\bar{Z}+S\bar{X}+\int_Z{F^\top \bar{r}\lambda (\mathrm{d}z)},u \rangle }\mathrm{d}s\Big]=\mathbb{E}\Big[\int_t^T{\langle \left[ \mathcal{L}_t\xi \right] (s),u(s)}\rangle \mathrm{d}s\Big]
\]
Substituting this into \eqref{identity} proves \eqref{costrep}.
\end{proof}

The representation \eqref{costrep} shows that the convexity of the cost
functional with respect to the control variable is determined by
the operator $\mathcal N_t$. In particular, taking $\xi=0$ in
\eqref{costrep}, we have
\(
  J(t,0;u)=\llbracket \mathcal N_tu,u\rrbracket, u\in\mathcal U[t,T].
\)
Therefore, the uniform convexity of $u\mapsto J(t,0;u)$ is equivalently
written as
\begin{equation}\label{H3t}
  J(t,0;u)
  =
  \llbracket \mathcal N_tu,u\rrbracket
  \ge
  \varepsilon\,\mathbb E\bigg[\int_t^T |u(s)|^2\,\mathrm ds\bigg],
  \quad \forall u\in\mathcal U[t,T],\qquad \text {for some } \varepsilon>0.
\end{equation}

For later comparison, we recall two standard sufficient conditions for
\eqref{H3t}. The first condition is that there exists some $\varepsilon_R>0$, such that
\begin{equation}\label{convex1}
  G \ge 0, \quad Q(\cdot) \ge 0, \quad S(\cdot)=0, \quad R(\cdot) \ge \varepsilon_R I_{m},
\end{equation}
The second condition is that there exists some $\varepsilon_G>0$ and  $\varepsilon_0>0$, such that
\begin{equation}\label{convex2}
    Q(\cdot)\ge0,\quad R(\cdot)\ge0,\quad S(\cdot)=0,\quad
    G\ge\varepsilon_G I_n, \quad D(s)^\top D(s)+\int_Z F(s,z)^\top F(s,z)\lambda(\mathrm dz)
  \ge \varepsilon_0 I_m,
\end{equation}
Either \eqref{convex1} or \eqref{convex2} implies \eqref{H3t};
see \proref{prop:basic_unifconv} below.
The sufficient conditions in
\secref{sec:sufficient} show that \ref{N_0uniformalconvex2-I} is strictly
more general than both \eqref{convex1} and \eqref{convex2}.

\subsection{Open-Loop Solvability}\label{subsec:5open-loop}

In this subsection, we characterize open-loop optimal controls by the quadratic
representation obtained in \secref{sec:hilbert}. The key point is that the
convexity of the cost functional is necessary for optimality, while its
uniform convexity is sufficient for unique open-loop solvability.

 \begin{thm}\label{Opensolve}
  Let \ref{SLQH1}-\ref{SLQH2} hold, and let the initial pair $(t, \xi) \in\mathcal{D}$ be given. A control $u^*\in\mathcal U[t,T]$ is open-loop optimal for Problem
 (SLQ-J) at $(t,\xi)$ if and only if
  \begin{enumerate}[label=(\roman*)]
    \item $u\mapsto J(t,0;u)$ is convex, equivalently,
    \[
      J(t,0;v)\ge0,\qquad \forall v\in\mathcal U[t,T];
    \]
    \item the adapted solution $(X^*,Y^*,Z^*,r^*)$ of \eqref{FBSDE1}
    corresponding to $u^*$ satisfies
    \[
      B^\top Y^*+D^\top Z^*+SX^*
      +\int_Z F^\top r^*\,\lambda(\mathrm dz)
      +Ru^*=0,
      \qquad \text{a.e. on }[t,T],\ \text{a.s.}
    \]
  \end{enumerate}
If, in addition, \eqref{H3t} holds, then Problem (SLQ-J) is uniquely
open-loop solvable at $(t,\xi)$, and its unique optimal control is
\(
  u^*=-\mathcal N_t^{-1}\mathcal L_t\xi.
\)
 \end{thm}
 
 \begin{proof}
By the quadratic representation \eqref{costrep}, for any
$u,v\in\mathcal U[t,T]$ and $\beta\in\mathbb R$,
\[
\begin{aligned}
  J(t,\xi;u+\beta v)-J(t,\xi;u)
  &=
  \beta^2\llbracket \mathcal N_tv,v\rrbracket
  +2\beta\llbracket \mathcal N_tu+\mathcal L_t\xi,v\rrbracket  
  =
  \beta^2 J(t,0;v)
  +2\beta\llbracket \mathcal N_tu+\mathcal L_t\xi,v\rrbracket .
\end{aligned}
\]
Therefore $u^*$ is optimal if and only if the above quadratic polynomial in
$\beta$ is nonnegative for every $v\in\mathcal U[t,T]$ and every
$\beta\in\mathbb R$. This is equivalent to
\[
  J(t,0;v)\ge0,\qquad \forall v\in\mathcal U[t,T],
  \qquad
  \mathcal N_tu^*+\mathcal L_t\xi=0.
\]
By the definitions of $\mathcal N_t$ and $\mathcal L_t$, the second identity
is precisely the stationarity condition in (ii). This proves the
characterization of optimal controls.

Furthermore, if \eqref{H3t} holds, then
$\mathcal N_t\ge\varepsilon I$ on the Hilbert space $\mathcal U[t,T]$.
Since $\mathcal N_t$ is bounded and self-adjoint by
\proref{prop:NtLtbounded}, it is boundedly invertible. Hence 
$\mathcal N_tu^*+\mathcal L_t\xi=0$ admits the unique solution
$u^*=-\mathcal N_t^{-1}\mathcal L_t\xi$, which is therefore the unique
open-loop optimal control.
\end{proof}
%  Suppose $u$ is open-loop optimal for $(t, \xi)$. Then it satisfies 
%  \[
%    J(t, \xi  ; u+\lambda v)-J(t, \xi ; u) \ge 0, \quad \forall \lambda \in \mathbb{R}, \forall v \in \mathcal{U}[t, T] 
%  \]
%  which is equivalent to 
%  \[
%  \lambda ^2J(t,0;v)+2\lambda \llbracket \mathcal{N}_tu+\mathcal{L}_t \xi,v \rrbracket =\lambda ^2J(t,0;v)+\lambda \llbracket \mathcal{D}_uJ(t,\xi;u),v \rrbracket \ge 0,\quad \forall \lambda \in \mathbb{R}, \forall v \in \mathcal{U}[t, T] .
%  \]
%  Consequently, $u$ is open-loop optimal if and only if
%  \[
%  J(t,0;v)\ge 0 \quad \forall v\in \mathcal{U}[t, T]; \quad \mathcal{D}_uJ(t,\xi;u)=0.
%  \]

%  \emph{Proof of (ii).}\;
%  By the quadratic representation \eqref{costrep} and \proref{unifoper}, the uniform convexity is equivalent to the coercivity $\llbracket \mathcal{N}_t u,u\rrbracket\ge\delta\|u\|_{\mathcal{U}[t,T]}^2$ for all $u\in\mathcal{U}[t,T]$. Since $\mathcal{N}_t$ is a bounded self-adjoint operator (\proref{prop:NtLtbounded}(i)), it is therefore boundedly invertible on $\mathcal{U}[t,T]$. The quadratic functional $u\mapsto J(t,\xi;u)$ thus admits the unique minimizer $u^*=-\mathcal{N}_t^{-1}\mathcal{L}_t\xi$.

Consequently, if $u^*$ is an open-loop optimal control for Problem
(SLQ-J) at $(t,\xi)$, then the corresponding state-adjoint quadruple
$(X^*,Y^*,Z^*,r^*)$ satisfies the optimality system
\begin{equation}\label{optimal}
\begin{cases}
  \mathrm dX^*(s)
  =(AX^*+Bu^*)\mathrm ds+(CX^*+Du^*)\mathrm dW(s)
  +\int_Z(EX^*+Fu^*)\widetilde N(\mathrm ds,\mathrm dz),\\[0.4em]
  \mathrm dY^*(s)
  =-\left(A^\top Y^*+C^\top Z^*+QX^*+S^\top u^*
  +\int_ZE^\top r^*\,\lambda(\mathrm dz)\right)\mathrm ds
  +Z^*\mathrm dW(s)
  +\int_Z r^*\,\widetilde N(\mathrm ds,\mathrm dz),
  \\[0.4em]
  B^\top Y^*+D^\top Z^*+SX^*
  +\int_ZF^\top r^*\,\lambda(\mathrm dz)+Ru^*=0,
  ~\text{a.e. on }[t,T],\ \text{a.s.}\\[0.4em]
  X^*(t)=\xi,~Y^*(T)=GX^*(T).
\end{cases}
\end{equation}
 We call \eqref{optimal} the optimality system of Problem (SLQ-J).  The \emph{stationarity condition} couples the forward and backward equations
in \eqref{optimal}.

%  Thus, to find an open-loop optimal control of Problem (SLQ-J), one actually needs to solve a \emph{coupled FBSDE}.

 The next theorem establishes the equivalence between Problems (SLQ-J) and ($\widehat{\text{SLQ-J}}$).

\begin{thm}\label{thm:SLQhat_equiv}
Let \ref{SLQH1}--\ref{SLQH2} hold. For any $(t,\xi)\in\mathcal D$, a
control $u^*\in\mathcal U[t,T]$ is optimal for Problem (SLQ-J) if and only
if it is optimal for Problem $\widehat{\mathrm{(SLQ\text{-}J)}}$.
\end{thm}

\begin{proof}
The proof is similar to those of Wen et al.~\cite[Theorem~4.2]{wen2022stochastic}
and Sun, Xiong, and Yong~\cite[Theorem~4.2]{sun2021indefinite}. We
therefore omit the details and refer the interested reader to these works.
\end{proof}

To conclude this section, we present the following useful consequence.

\begin{coro}\label{cor:value_adjoint}
Let \ref{SLQH1}--\ref{SLQH2} hold. Suppose that $(X^*,u^*)$ is an optimal
pair at $(t,\xi)$, and let $(Y^*,Z^*,r^*)$ be the adjoint BSDE solution
associated with $(X^*,u^*)$. Then
\[
  \widehat V(t,\xi)
  =
  \widehat J(t,\xi;u^*)
  =
  \langle Y^*(t),\xi\rangle .
\]
\end{coro}

\begin{proof}
By \thmref{Opensolve}, the stationarity condition in \eqref{optimal}
holds for the optimal pair. Substituting this condition into
\eqref{rep1} gives
\(
  \widehat J(t,\xi;u^*)=\langle Y^*(t),\xi\rangle .
\)
The equality with $\widehat V(t,\xi)$ follows from
\thmref{thm:SLQhat_equiv}.
\end{proof}

\begin{remark}\label{rem:stochastic_interval}
So far we have established the results of this section on a deterministic interval $[t,T]$.
We may also consider Problems (SLQ-J) and (\SLQh) on a stochastic interval $[\sigma,\tau]$, where $\sigma$ and $\tau$ are $\mathbb F$-stopping times with $0\le\sigma\le\tau\le T$, and the terminal weight $G$ is replaced by a bounded $\mathcal F_\tau$-measurable $\mathbb S^n$-valued random variable.
With $t$ and $T$ respectively replaced by $\sigma$ and $\tau$, all the results of this section remain valid and can be proved using the same arguments.
See \cite{sun2021indefinite,tang2015dynamic} for a similar consideration.
\end{remark}

% By virtue of the optimality of $\left(X^{*}, u^{*}\right)$ and \thmref{Opensolve}, we obtain the following stationarity condition
% \[
% B(s)^{\top}Y^*(s)+D(s)^{\top}Z^*(s)+S(s)X^*(s)+\int_Z{F^{\top}\left( s,z \right)}r^*\left( s,z \right) \lambda (\mathrm{d}z)+R(s)u^*(s)=0,\quad \mathrm{a}.\mathrm{e}.s\in [t,T],\mathrm{a}.\mathrm{s}.
% \]
% Therefore, the result follows immediately from \eqref{rep1}. 

% The following corollary states the time-consistency (principle of optimality) for Problem (SLQ-J).

% \begin{coro}[Principle of Optimality]\label{coro:time-consistency}
%   Let \ref{SLQH1}-\ref{SLQH2} hold. Suppose that $u^{*}(\cdot) \in \mathcal{U}[t, T]$ is an optimal control with respect to the initial pair $(t, \xi) \in \mathcal{D}$, and let $X^{*}=\left\{X^{*}(s) ; s \in[t, T]\right\}$ be the corresponding optimal state process. Then for any stopping time $\tau$ with $\tau \in[t, T]$, the restriction
%   \[
%   \left.u^{*}(\cdot)\right|_{[\tau, T]} \triangleq\left\{u^{*}(s) ; s \in[\tau, T]\right\}
%   \]
%   of $u^{*}(\cdot)$ over the time horizon $[\tau, T]$ remains optimal with respect to the initial pair $\left(\tau, X^{*}(\tau)\right)$.
% \end{coro}

 \section{Stochastic Value Flow}\label{sec:Value Flow}

This section constructs the stochastic value flow. Under the uniform
convexity condition, we first prove that the conditional value is quadratic
at every stopping time:
\(
  \widehat V(\tau,\xi)=\langle P(\tau)\xi,\xi\rangle .
\)
We then show that the family $\{P(\tau)\}_{\tau\in\mathcal T[0,T]}$
admits a bounded c\`adl\`ag special-semimartingale aggregation. This
aggregated process will be identified in \secref{sec:Riccati} with the
solution of the SRE-J.

% \begingroup\color{blue}
% \subsection[Why the global Y X-1 method fails]{Why the global
% $\mathbf Y\mathbf X^{-1}$ method fails}\label{sec:inverseflow_fail}

% \exaref{exam:counterexample2} shows the obstruction that is specific to jump
% systems: the closed-loop jump map is governed by $I_n+E+F\Theta$, not by the
% open-loop multiplier $I_n+E$. Hence $I_n+E$ may be nonsingular while the
% optimal state fundamental matrix loses invertibility at the first jump time.
% This prevents a global recovery of the Riccati process through
% $P=\mathbf Y\mathbf X^{-1}$.

% The construction below avoids this difficulty. On each inter-jump interval the
% Poisson random measure has no actual jump and the state equation reduces to a
% Brownian SDE with the compensator in the drift; the corresponding fundamental
% matrix is locally invertible there. We use this only locally, and then glue the
% resulting semimartingales by the $\mathcal T$-system aggregation procedure.
% Thus the argument needs inter-jump inverse flows but never a global inverse flow
% across jump times.
% \endgroup

\subsection{Quadratic representation of the stochastic value flow}
\label{sec:dettime}

Throughout Sections~\ref{sec:Value Flow}--\ref{sec:Riccati}, we assume
\ref{N_0uniformalconvex2-I}, i.e., there exists $\varepsilon>0$ such that
\begin{equation}\label{N_0uniformalconvex}
  J(0,0;u)=\llbracket \mathcal N_0u,u\rrbracket
  \ge \varepsilon\,\mathbb E\int_0^T |u(s)|^2\,\mathrm ds,
  \qquad u\in\mathcal U[0,T].
\end{equation}
We shall use the control problem with a random initial time.  For notational
simplicity, write
\[
  \ell(s,x,u):=\langle Q(s)x,x\rangle
  +2\langle S(s)x,u\rangle+\langle R(s)u,u\rangle .
\]
For a stopping time $\sigma\in\mathcal T[0,T]$, let
$\mathcal U[\sigma,T]$ be the set of $\mathbb F$-predictable controls
satisfying
\(
  \mathbb E\int_\sigma^T|u(s)|^2\,\mathrm ds<\infty .
\)
For $\xi\in L^2_{\mathcal F_\sigma}(\Omega;\mathbb R^n)$ and
$u\in\mathcal U[\sigma,T]$, denote by $X^{\sigma,\xi,u}$ the corresponding
state process and set
\begin{equation}\label{eq:random_initial_cost}
  \widehat J(\sigma,\xi;u)
  :=\mathbb E_\sigma\left[
    \langle GX^{\sigma,\xi,u}(T),X^{\sigma,\xi,u}(T)\rangle
    +\int_\sigma^T
      \ell(s,X^{\sigma,\xi,u}(s),u(s))\,\mathrm ds
  \right].
\end{equation}
and define
\[
  \widehat V(\sigma,\xi)
  :=
  \operatorname*{ess\,inf}_{u\in\mathcal U[\sigma,T]}
  \widehat J(\sigma,\xi;u),
  \qquad
  J(\sigma,\xi;u):=\mathbb E\,\widehat J(\sigma,\xi;u).
\]

We first propagate the uniform convexity \eqref{N_0uniformalconvex} from the
time $0$ to an arbitrary stopping time $\tau$. 
% this coercivity is the
% basis for the quadratic representation at $\tau$ established in
% \thmref{thm:Prandomtime}.
% [Propagation of uniform convexity]
\begin{lem}\label{lem:H3propagation}
  Let \ref{SLQH1}--\ref{SLQH2} and \eqref{N_0uniformalconvex} hold.  Then, for every
$\tau\in\mathcal T[0,T]$,
\[
  \widehat J(\tau,0;v)
  \ge
  \varepsilon\,\mathbb E_\tau\int_\tau^T |v(s)|^2\,\mathrm ds,
  \qquad
  v\in\mathcal U[\tau,T],
  \quad \hbox{a.s.}
\]
\end{lem}

\begin{proof}
Let $A\in\mathcal F_\tau$ and $v\in\mathcal U[\tau,T]$ be fixed. Define
\[
  \tilde v(s):=\mathbf 1_A\mathbf 1_{(\tau,T]}(s)v(s),
  \qquad s\in[0,T].
\]
Then $\tilde v\in\mathcal U[0,T]$. 
Let $X^{\tilde v}$ be the state starting from $0$ under $\tilde v$. Then
$X^{\tilde v}=0$ on $[0,\tau]$. On $A$, its restriction to $[\tau,T]$
coincides with the state starting from $0$ at time $\tau$ under $v$. On
$A^c$, it remains zero. Hence
\(
  J(0,0;\tilde v)
  =
  \mathbb E\bigl[
    \mathbf 1_A\,\widehat J(\tau,0;v)
  \bigr].
\)
By \eqref{N_0uniformalconvex},
\[
  \mathbb E\bigl[
    \mathbf 1_A\,\widehat J(\tau,0;v)
  \bigr]
  =
  J(0,0;\tilde v)
  \ge
  \varepsilon\,\mathbb E\int_0^T|\tilde v(s)|^2\,\mathrm ds
  =
  \varepsilon\,\mathbb E\biggl[
    \mathbf 1_A\int_\tau^T |v(s)|^2\,\mathrm ds
  \biggr].
\]
Since $A\in\mathcal F_\tau$ is arbitrary, the conditional inequality follows.
% Taking expectations gives the second assertion.
\end{proof}

Now, we show that $\widehat V(\tau,\xi)$ is quadratic in $\xi$ at every
$\tau\in\mathcal T[0,T]$.
% [Quadratic representation at stopping times]
\begin{thm}\label{thm:Prandomtime}
Let \ref{SLQH1}--\ref{SLQH2} and \eqref{N_0uniformalconvex} hold. Then, for every
$\tau\in\mathcal T[0,T]$, there exists a unique
$P(\tau)\in L^\infty_{\mathcal F_\tau}(\Omega;\mathbb S^n)$ such that
\[
  \widehat V(\tau,\xi)=\langle P(\tau)\xi,\xi\rangle,
  \qquad
  \xi\in L^2_{\mathcal F_\tau}(\Omega;\mathbb R^n).
\]
Moreover, there exists a constant $K>0$, independent of $\tau$, such that
$|P(\tau)|\le K$ a.s.
\end{thm}

\begin{proof}
% {\color{red}
% The proof strategy follows \cite[Theorem~5.1]{sun2021indefinite}
% and \cite[Theorem~3.4]{zhang2020backward}; the key difference is that the
% lower bound for the value (Step~2) relies on the propagated uniform convexity
% (\lemref{lem:H3propagation}) rather than the standard condition.}

Fix $\tau\in\mathcal T[0,T]$. By \lemref{lem:H3propagation}, the cost
functional is uniformly convex on $[\tau,T]$.
Thus, by \thmref{Opensolve} and \remref{rem:stochastic_interval}, the problem is uniquely open-loop solvable at $(\tau,\xi)$ for every $\xi\in L^2_{\mathcal F_\tau}(\Omega;\mathbb R^n)$.

\emph{Step 1: quadratic representation for bounded $\xi$.} Let $(X_i,U_i)$ be the optimal pair corresponding to $(\tau,e_i)$, $i=1,\ldots,n$, and set
$\mathbf X=(X_1,\ldots,X_n)$, $\mathbf U=(U_1,\ldots,U_n)$.
By the linearity of the state equation and the uniqueness of the optimal
control, $(\mathbf X\xi,\mathbf U\xi)$ is the optimal pair corresponding to
$(\tau,\xi)$ for every
$\xi\in L^\infty_{\mathcal F_\tau}(\Omega;\mathbb R^n)$. Therefore,
\begin{align*}
  \widehat V(\tau,\xi)
  &=
  \mathbb E_\tau\Big[
    \langle G\mathbf X(T)\xi,\mathbf X(T)\xi\rangle  
    +\int_\tau^T
    \Big(
      \langle Q\mathbf X\xi,\mathbf X\xi\rangle
      +2\langle S\mathbf X\xi,\mathbf U\xi\rangle
      +\langle R\mathbf U\xi,\mathbf U\xi\rangle
    \Big)\,\mathrm ds
  \Big] =:\langle P(\tau)\xi,\xi\rangle,
\end{align*}
where
\[
  P(\tau):=
  \mathbb E_\tau\Big[
    \mathbf X(T)^\top G\mathbf X(T)
    +\int_\tau^T
    \Big(
      \mathbf X^\top Q\mathbf X
      +\mathbf X^\top S^\top\mathbf U
      +\mathbf U^\top S\mathbf X
      +\mathbf U^\top R\mathbf U
    \Big)\,\mathrm ds
  \Big],
\]
is $\mathcal F_\tau$-measurable and symmetric.

\emph{Step 2: uniform boundedness of $P(\tau)$.}
% It remains to prove the uniform boundedness of $P(\tau)$.  
Since the problem on $[\tau,T]$ is uniquely solvable,
$\mathbb E[\widehat V(\tau,\xi)]=\inf_{u\in\mathcal U[\tau,T]}J(\tau,\xi;u)$.
Taking $u=0$ and using \lemref{existunique}, we obtain
\[
  \mathbb E[\widehat V(\tau,\xi)]\le J(\tau,\xi;0)\le K\mathbb E[|\xi|^2],
\]
where $K>0$ is independent of $\tau$ and $\xi$. On the other hand, by the
decomposition $X^{\tau,\xi,u}=X^{\tau,\xi,0}+X^{\tau,0,u}$, the boundedness of
the weighting coefficients, and \lemref{existunique},
\[
  J(\tau,\xi;u)
  \ge J(\tau,0;u)-K\,\mathbb E[|\xi|^2]
  -K\big(\mathbb E[|\xi|^2]\big)^{1/2}
     \Big(\mathbb E\bigg[\int_\tau^T|u(s)|^2\,\mathrm ds\bigg]\Big)^{1/2}.
\]
Taking expectations in \lemref{lem:H3propagation} and applying Young's
inequality, we get
\[
  J(\tau,\xi;u)
  \ge \frac{\varepsilon}{2}\,\mathbb E\bigg[\int_\tau^T|u(s)|^2\,\mathrm ds\bigg]
  -K\,\mathbb E[|\xi|^2]
  \ge -K\,\mathbb E[|\xi|^2] .
\]
Combining the two bounds gives
\[
  -K\,\mathbb E[|\xi|^2]
  \le \mathbb E\,\widehat V(\tau,\xi)
  \le K\,\mathbb E[|\xi|^2],
  \qquad \xi\in L^\infty_{\mathcal F_\tau}(\Omega;\mathbb R^n).
\]
For $x\in\mathbb R^n$ and $A\in\mathcal F_\tau$, taking $\xi=\mathbf 1_Ax$
implies
\(
  -K\,\mathbb P(A)|x|^2
  \le \mathbb E\big[\mathbf 1_A\langle P(\tau)x,x\rangle\big]
  \le K\,\mathbb P(A)|x|^2 .
\)
Since $A$ is arbitrary,
$|\langle P(\tau)x,x\rangle|\le K|x|^2$ a.s. for every $x\in\mathbb R^n$.
Hence, $|P(\tau)|\le K$ a.s.

\emph{Step 3: extension from $L^\infty_{\mathcal F_\tau}$ to $L^2_{\mathcal F_\tau}$.}
For $\xi\in L^2_{\mathcal F_\tau}(\Omega;\mathbb R^n)$, set
$\xi^N=\xi\mathbf 1_{\{|\xi|\le N\}}$, so that $\xi^N\to\xi$ in $L^2$.
By \thmref{Opensolve} and \lemref{existunique}, the optimal control
and state depend continuously on the initial state, 
then $\widehat V(\tau,\xi^N)\to
\widehat V(\tau,\xi)$ in $L^1$. Since $P(\tau)$ is bounded,
$\langle P(\tau)\xi^N,\xi^N\rangle\to\langle P(\tau)\xi,\xi\rangle$ in $L^1$,
and passing to the limit proves the representation on $L^2_{\mathcal F_\tau}$.
Finally, if $\widetilde P(\tau)\in L^\infty_{\mathcal F_\tau}(\Omega;\mathbb S^n)$ also represents $\widehat V(\tau,\cdot)$, then taking $\xi=\mathbf 1_Ax$ with $A\in\mathcal F_\tau$ and $x\in\mathbb R^n$ gives $\langle(P(\tau)-\widetilde P(\tau))x,x\rangle=0$ a.s.\ for all $x\in\mathbb Q^n$. Hence $\widetilde P(\tau)=P(\tau)$ a.s.
\end{proof}

By the uniqueness in \thmref{thm:Prandomtime}, the family
$\{P(\tau)\}_{\tau\in\mathcal T[0,T]}$ is a bounded $\mathcal T$-system.
Indeed, if $\sigma,\tau\in\mathcal T[0,T]$ and $A=\{\sigma=\tau\}$, then
$\widehat V(\sigma,\mathbf 1_Ax)=\widehat V(\tau,\mathbf 1_Ax)$ on $A$,
so $\mathbf 1_A P(\sigma)=\mathbf 1_A P(\tau)$ a.s.

For stopping times $\sigma\le\tau$, we write $\mathcal U[\sigma,\tau]$
for the set of $\mathbb F$-predictable processes $u$ on
$(\sigma,\tau]$ with
$\mathbb E\int_\sigma^\tau|u(s)|^2\mathrm ds<\infty$.
The following Bellman principle follows from the standard dynamic
programming method for random initial times; see, for example,
Tang~\cite[Theorem~4.6]{tang2015dynamic} and Zhang, Dong, and
Meng~\cite[Proposition~3.3]{zhang2020backward}.

\begin{prop}\label{prop:DPPgeneral}
Let \ref{SLQH1}--\ref{SLQH2} and \eqref{N_0uniformalconvex} hold, and let
$\sigma,\tau\in\mathcal T[0,T]$ with $\sigma\le\tau$.
Then, for every
$\xi\in L^2_{\mathcal F_\sigma}(\Omega;\mathbb R^n)$,
\begin{equation}\label{eq:DPP_general}
  \widehat V(\sigma,\xi)
  =
  \operatorname*{ess\,inf}_{u\in\mathcal U[\sigma,\tau]}
  \mathbb E_\sigma\!\left[
    \int_\sigma^\tau \ell(s,X(s),u(s))\,\mathrm ds
    +\langle P(\tau)X(\tau),X(\tau)\rangle
  \right],
\end{equation}
where $X$ is the state process on $[\sigma,\tau]$ driven by $u$ with
$X(\sigma)=\xi$.
\end{prop}

\subsection{Aggregation into a c\`{a}dl\`{a}g semimartingale}
\label{sec:aggregation}

We now verify the RCE property required by the $\mathcal T$-system
aggregation (see \proref{prop:RCE}) and then construct the c\`adl\`ag
semimartingale~$P$.

\begin{prop}\label{prop:stopcont}
Let \ref{SLQH1}--\ref{SLQH2} and \eqref{N_0uniformalconvex} hold.
\begin{enumerate}[label=\upshape(\roman*)]
  \item For every $x\in\mathbb R^n$, the $\mathcal T$-system
  $\{\widehat V(\tau,x)\}_{\tau\in\mathcal T[0,T]}$ is of class $D$ and
  satisfies the RCE condition.

  \item Let $\rho\in\mathcal T[0,T]$, and let $X^{\rho,x,0}$ be the state on
  $[\rho,T]$ starting from $x\in\mathbb R^n$ under the zero control. Then
  \[
    Y^{\rho,x}(\tau)
    :=
    \widehat V\bigl(\tau,X^{\rho,x,0}(\tau)\bigr)
    +\int_\rho^\tau
      \ell\bigl(s,X^{\rho,x,0}(s),0\bigr)\,\mathrm ds,
    \qquad \tau\in\mathcal T[\rho,T],
  \]
  is a class $D$ submartingale $\mathcal T[\rho,T]$-system satisfying the RCE condition.
\end{enumerate}
\end{prop}

\begin{proof}
By \thmref{thm:Prandomtime}, $|P(\sigma)|\le K$ a.s., uniformly in $\sigma\in\mathcal T[0,T]$. Hence $|\widehat V(\sigma,x)|\le K|x|^2$, and, by \eqref{FSDEestimate},
\(
  |Y^{\rho,x}(\tau)|\le K\big(1+\sup_{s\in[\rho,T]}|X^{\rho,x,0}(s)|^2\big)\in L^1(\Omega)
\)
for $\tau\in\mathcal T[\rho,T]$. Thus the systems in (i) and (ii) are of class $D$. Moreover, for $\sigma\in\mathcal T[0,T]$ and $\eta_1,\eta_2\in L^2_{\mathcal F_\sigma}(\Omega;\mathbb R^n)$,
\begin{equation}\label{eq:RCE_perturbation}
  \mathbb E\big|\widehat V(\sigma,\eta_1)-\widehat V(\sigma,\eta_2)\big|
  \le K\,\mathbb E\big[|\eta_1-\eta_2|\,(|\eta_1|+|\eta_2|)\big].
\end{equation}

Fix $\tau_k\downarrow\tau$ in $\mathcal T[0,T]$. We first claim that, for every
$\xi\in L^2_{\mathcal F_\tau}(\Omega;\mathbb R^n)$,
\begin{equation}\label{eq:RCE_random_initial_state}
  \mathbb E\widehat V(\tau_k,\xi)\rightarrow\mathbb E\widehat V(\tau,\xi).
\end{equation}
Let $(X^*,u^*)$ be the optimal pair for $(\tau,\xi)$. By the time consistency consequence of
\proref{prop:DPPgeneral}, $u^*|_{[\tau_k,T]}$ is optimal for $(\tau_k,X^*(\tau_k))$, so that
\[
  \mathbb E\widehat V(\tau,\xi)
  =\mathbb E\int_\tau^{\tau_k}\ell(s,X^*(s),u^*(s))\,\mathrm ds
  +\mathbb E\widehat V\bigl(\tau_k,X^*(\tau_k)\bigr).
\]
The integral term tends to zero, and \eqref{eq:RCE_perturbation} with $\eta_1=X^*(\tau_k)$ and $\eta_2=\xi$ gives \eqref{eq:RCE_random_initial_state}. Assertion (i) follows by taking $\xi=x$.

For (ii), the $\mathcal T[\rho,T]$-system property follows
from the pathwise uniqueness of the state equation and the $\mathcal T$-system
property of $P$. If $\eta,\theta\in\mathcal T[\rho,T]$ and $\eta\le\theta$,
then applying \proref{prop:DPPgeneral} on $[\eta,\theta]$ and testing with the
zero control gives
\[
  \widehat V\bigl(\eta,X^{\rho,x,0}(\eta)\bigr)
  \le
  \mathbb E_\eta\Big[
    \int_\eta^\theta
      \ell\bigl(s,X^{\rho,x,0}(s),0\bigr)\,\mathrm ds
    +\widehat V\bigl(\theta,X^{\rho,x,0}(\theta)\bigr)
  \Big].
\]
Hence $Y^{\rho,x}$ is a submartingale system.
Finally, let $\tau_k\downarrow\tau$ in $\mathcal T[\rho,T]$, and set $\Xi_k:=X^{\rho,x,0}(\tau_k)$ and $\Xi:=X^{\rho,x,0}(\tau)$, so that $\Xi_k\to\Xi$ in $L^2$. Writing
\[
\begin{aligned}
  \mathbb E Y^{\rho,x}(\tau_k)-\mathbb E Y^{\rho,x}(\tau)
  &=\mathbb E\big[\widehat V(\tau_k,\Xi_k)-\widehat V(\tau_k,\Xi)\big]
  +\big[\mathbb E\widehat V(\tau_k,\Xi)-\mathbb E\widehat V(\tau,\Xi)\big]+\mathbb E\int_\tau^{\tau_k}\ell\bigl(s,X^{\rho,x,0}(s),0\bigr)\,\mathrm ds,
\end{aligned}
\]
we see that the three terms tend to zero by \eqref{eq:RCE_perturbation}, \eqref{eq:RCE_random_initial_state}, and the dominated convergence theorem, respectively. Thus $Y^{\rho,x}$ satisfies RCE.
\end{proof}

\begin{prop}\label{prop:Psemimartingale}
  The $\mathcal T$-system $\{P(\tau)\}_{\tau\in\mathcal T[0,T]}$ admits a
  bounded c\`{a}dl\`{a}g adapted aggregation, still denoted by $P$. Moreover,
  $P$ is a bounded special semimartingale with canonical
  decomposition
  \begin{equation}\label{Psemimartingale_canonical}
    P(t)=P(0)-\mathscr A_P(t)+\mathcal M(t),\qquad t\in[0,T],
  \end{equation}
  where $\mathscr A_P$ is a predictable finite-variation
  $\mathbb S^n$-valued process with $\mathscr A_P(0)=0$, and
  $\mathcal M$ is a c\`{a}dl\`{a}g local martingale admitting the
  representation
  \begin{equation}\label{Psemimartingale_PRP}
    \mathcal M(t)
    =
    \int_0^t\Lambda(s)\,\mathrm dW(s)
    +\int_0^t\int_Z\zeta(s,z)\,\widetilde N(\mathrm ds,\mathrm dz),
  \end{equation}
  for some $\mathbb S^n$-valued predictable integrands $\Lambda$ and
  $\zeta$ that are locally square-integrable.
\end{prop}

\begin{proof}
The proof follows the aggregation argument of Zhang, Dong, and Meng~\cite{zhang2020backward} (see also Tang~\cite{tang2015dynamic} for the case without jumps). We only sketch the argument.
Let $(T_k)_{k\ge1}$ be the successive jump times of $N$, and set $T_0=0$. For $k\ge0$, let $\Phi_k$ be the solution to the following matrix-valued SDE without jumps:
\[
  \mathrm d\Phi_k=\Big(A-\int_ZE\,\lambda(\mathrm dz)\Big)\Phi_k\,\mathrm ds+C\Phi_k\,\mathrm dW(s),
  \quad s\in[T_k\wedge T,T],\qquad \Phi_k(T_k\wedge T)=I_n .
\]
Then $\Phi_k$ is continuous and invertible on $[T_k\wedge T,T]$. Since $N$ has no jump in $(T_k,T_{k+1})$, the $i$-th column of $\Phi_k$ coincides with $X^{T_k\wedge T,e_i,0}$ on $[T_k\wedge T,T_{k+1}\wedge T)$.

By \proref{prop:stopcont}(ii) with $\rho=T_k\wedge T$ and \proref{prop:RCE}, the system $Y^{T_k\wedge T,x}$ admits a c\`adl\`ag submartingale aggregation for every $x\in\mathbb R^n$. By polarization, the $\mathbb S^n$-valued c\`adl\`ag semimartingale
\(
  \Gamma_k:=\frac14\big(Y^{T_k\wedge T,e_i+e_j}-Y^{T_k\wedge T,e_i-e_j}\big)_{1\le i,j\le n}
\)
satisfies, by \thmref{thm:Prandomtime} and the linearity of the state equation,
\[
  \Gamma_k(\tau)
  =\Phi_k(\tau)^\top P(\tau)\Phi_k(\tau)
  +\int_{T_k\wedge T}^{\tau}\Phi_k^\top Q\Phi_k\,\mathrm ds
  \quad\text{on }\{\tau<T_{k+1}\},\qquad \tau\in\mathcal T[T_k\wedge T,T].
\]
Hence the c\`adl\`ag semimartingale
\(
  P_k:=(\Phi_k^{-1})^\top\big(\Gamma_k-\int_{T_k\wedge T}^{\cdot}\Phi_k^\top Q\Phi_k\,\mathrm ds\big)\Phi_k^{-1}
\)
aggregates $\{P(\tau)\}$ on $\{T_k\le\tau<T_{k+1}\}$. Since only finitely many $T_k$ lie in $[0,T]$, a.s., the process
\[
  P(t):=\sum_{k=0}^{\infty}P_k(t)\mathbf 1_{\{T_k\le t<T_{k+1}\}},\quad t\in[0,T),\qquad P(T):=G,
\]
is a c\`adl\`ag semimartingale aggregating $\{P(\tau)\}_{\tau\in\mathcal T[0,T]}$; see, e.g., \cite[Chapter~II]{protter2005stochastic}.

By \thmref{thm:Prandomtime} and the optional section theorem, $P$ is bounded. Hence $P$ is a special semimartingale, and \eqref{Psemimartingale_canonical} holds. Since $P$ has bounded jumps, so does $\mathcal M$, and thus $\mathcal M$ is locally square-integrable. 
Finally, by the martingale representation theorem for the filtration
generated by $W$ and $N$ (see Jacod and
Shiryaev~\cite[Theorem~III.4.34]{jacod2003limit}), we obtain
\eqref{Psemimartingale_PRP}.
\end{proof}

\begin{remark}\label{rem:cadlag}
\proref{prop:Psemimartingale} gives only a local martingale representation.
At this stage, we do not assert a global
$L^2_{\mathbb F}(0,T;\mathbb S^n)\times
G^2_{\mathbb F}(0,T,\lambda;\mathbb S^n)$ estimate for $(\Lambda,\zeta)$,
nor the absolute continuity of $\mathscr A_P$. These properties will be
proved after the Riccati drift is identified.
\end{remark}

 \section{Stochastic Riccati Equation and Closed-Loop Representation}\label{sec:Riccati}

In this section, we identify the semimartingale obtained in \secref{sec:Value Flow} with the solution of the stochastic Riccati equation with jumps, and then derive the closed-loop representation of the open-loop optimal control. 
The main difference from the standard condition case of
Zhang, Dong, and Meng~\cite{zhang2020backward} is that the positive definiteness of
$\widehat R$ no longer follows from positivity of the
coefficients. Instead, it is extracted from the uniform convexity of the cost functional
by a small-interval method.

% The stochastic Riccati equation with jumps associated with Problem (SLQ-J) is
% \begin{equation}\label{SREJ}
% \begin{cases}
%   \mathrm{d} P(s)=-\Big(A(s)^\top P(s-)+P(s-)A(s)+Q(s)+C(s)^\top P(s-)C(s)
%   +C(s)^\top\Lambda(s)+\Lambda(s)C(s) \\
%   \qquad\quad +\displaystyle\int_Z\Big[\zeta(s,z)E(s,z)+E(s,z)^\top\zeta(s,z)+E(s,z)^\top\big(\zeta(s,z)+P(s-)\big)E(s,z)\Big]\lambda(\mathrm{d} z) \\
%   \qquad\quad -\widehat S(s)^\top\widehat R(s)^{-1}\widehat S(s)\Big)\mathrm{d} s
%   +\Lambda(s)\mathrm{d} W(s)+\displaystyle\int_Z\zeta(s,z)\widetilde N(\mathrm{d} s,\mathrm{d} z),\\
%   P(T)=G,
% \end{cases}
% \end{equation}
% where
% \begin{equation}\label{RShat}
% \begin{cases}
%   \widehat S=B^\top P_-+D^\top\Lambda +D^\top P_-C
%   +\displaystyle\int_Z\Big[F^\top\zeta+F^\top(\zeta+P_-)E\Big]\lambda(\mathrm{d} z)+S,\\[4pt]
%   \widehat R=R+D^\top P_-D+\displaystyle\int_Z F^\top(\zeta+P_-)F\,\lambda(\mathrm{d} z).
% \end{cases}
% \end{equation}

\begin{defn}\label{def:strongregular}
A triple $(P,\Lambda,\zeta)\in S^\infty_{\mathbb F}(0,T;\mathbb S^n)\times L^2_{\mathbb F}(0,T;\mathbb S^n)\times G^2_{\mathbb F}(0,T,\lambda;\mathbb S^n)$ is called a \emph{strongly regular solution} of the SRE-J \eqref{SREJ} if
$(P,\Lambda,\zeta)$ satisfies \eqref{SREJ}, and there exists a constant $\delta>0$ such that
\begin{equation}\label{Rhatposdef}
  \widehat R(s)=R(s)+D(s)^\top P(s-)D(s)+\int_Z F(s,z)^\top\big(\zeta(s,z)+P(s-)\big)F(s,z)\lambda(\mathrm{d} z)
  \ge \delta I_m,
\end{equation}
for a.e.\ $s\in[0,T]$, a.s.
\end{defn}

% [Strong regularity and closed-loop representation]
\begin{thm}\label{thm:SREsolve}
Let \ref{SLQH1}--\ref{SLQH2} hold, and suppose that
\eqref{N_0uniformalconvex} holds. 
Let \(P\) be the c\`adl\`ag semimartingale aggregation constructed in
\proref{prop:Psemimartingale}, with martingale components
\((\Lambda,\zeta)\) given by \eqref{Psemimartingale_PRP}.
Then

\begin{enumerate}[label=(\roman*)]
  \item For every stopping time $\tau\in\mathcal T[0,T]$ and every $\xi\in L^2_{\mathcal F_\tau}(\Omega;\mathbb R^n)$, the random-time problem is uniquely open-loop solvable and
  \begin{equation}\label{randomtime_value_in_SRE_section}
    \widehat V(\tau,\xi)=\langle P(\tau)\xi,\xi\rangle.
  \end{equation}

  \item The triple $(P,\Lambda,\zeta)$ is a strongly regular
  solution of the SRE-J \eqref{SREJ}, and \eqref{Rhatposdef} holds with $\delta=\varepsilon$. Moreover,
\begin{itemize}
  \item $(P,\Lambda,\zeta)$ is maximal in the sense that any strongly regular
  solution $(P',\Lambda',\zeta')$ of \eqref{SREJ} satisfies
  \(
    P'(t)\le P(t),\ t\in[0,T],\ \text{a.s.}
  \)

  \item For any such solution, set
  $\Theta'=-(\widehat R')^{-1}\widehat S'$.
  If, for every $(t,\xi)\in\mathcal D$, the closed-loop equation
  \eqref{closestate} with $\Theta=\Theta'$ admits a solution $X'$
  satisfying $\Theta'X'_-\in\mathcal U[t,T]$, then
  \(
    (P',\Lambda',\zeta')=(P,\Lambda,\zeta).
  \)
\end{itemize}
  % Moreover, $P$ is the unique \emph{maximal} strongly
  % regular solution: if $(P',\Lambda',\zeta')$ is any strongly regular
  % solution of \eqref{SREJ}, then
  % \(
  %   P'(t)\le P(t),\ t\in[0,T],\ \text{a.s.}
  % \)
  % Moreover, let $(P',\Lambda',\zeta')$ be a strongly regular solution of \eqref{SREJ}, and set $\Theta'=-(\widehat R')^{-1}\widehat S'$. If, for every $(t,\xi)\in\mathcal D$, the closed-loop system \eqref{closestate} with $\Theta=\Theta'$ admits a solution $X'$ such that $\Theta'X'_-\in\mathcal U[t,T]$, then $(P',\Lambda',\zeta')=(P,\Lambda,\zeta)$.

  \item For every $(t,\xi)\in\mathcal D$, the unique open-loop optimal control
  admits the closed-loop representation
  \begin{equation}\label{feedback_rep}
    u^*_{t,\xi}(s)=\Theta(s)X^*(s-),
    \qquad
    \Theta(s):=-\widehat R(s)^{-1}\widehat S(s),
  \end{equation}
  where $X^*$ is the corresponding optimal state. In particular, $X^*$ satisfies the closed-loop equation \eqref{closestate} with $\Theta$ defined above,  and there exists a
  constant \(K\), independent of \((t,\xi)\), such that
  \[
    \mathbb E\left[
    \sup_{s\in[t,T]}|X^*(s)|^2+
    \int_t^T|\Theta(s)X^*(s-)|^2\,\mathrm ds
    \right]
    \le K\mathbb E[|\xi|^2] .
  \]
\end{enumerate}
\end{thm}

The proof of \thmref{thm:SREsolve} is given at the end of this section.
We first collect the objects that are available immediately after the
aggregation step of \secref{sec:Value Flow}.
Recall the canonical
decomposition \eqref{Psemimartingale_canonical}--\eqref{Psemimartingale_PRP},
\[
  P(t)=P(0)-\mathscr A_P(t)+\mathcal M(t),
  \qquad
  \mathcal M(t)=\int_0^t\Lambda(s)\,\mathrm dW(s)
  +\int_0^t\int_Z\zeta(s,z)\,\widetilde N(\mathrm ds,\mathrm dz),
\]
in which $\mathscr A_P$ is the predictable finite-variation part of $P$,
and $(\Lambda,\zeta)$
are locally square-integrable. Inspired by the SRE-J \eqref{SREJ}, we introduce the
\emph{Lyapunov drift}
\begin{equation}\label{Gammacirc_def}
\begin{aligned}
  \Gamma^\circ
  :=&\;A^\top P_-+P_-A+Q+C^\top P_-C+C^\top\Lambda+\Lambda C +\int_Z\Big[\zeta E+E^\top\zeta+E^\top(\zeta+P_-)E\Big]
  \lambda(\mathrm{d} z),
\end{aligned}
\end{equation}
which is exactly the part of the generator of
\eqref{SREJ} that is linear in the solution.
For use in the It\^o identity, we introduce the auxiliary finite-variation
measure
\begin{equation}\label{rawresidualmeasure}
  \mathrm d\mathscr N_P(s):=\Gamma^\circ(s)\,\mathrm ds
  -\mathrm d\mathscr A_P(s).
\end{equation}
Once $\widehat R$ is known to be invertible, the
\emph{completed residual measure} is defined by
\begin{equation}\label{completedresidual}
  \mathrm d\mathscr R_P(s):=
  \Big(\Gamma^\circ(s)-\widehat S(s)^\top\widehat R(s)^{-1}\widehat S(s)\Big)
  \mathrm ds-\mathrm d\mathscr A_P(s).
\end{equation}
Thus proving that $(P,\Lambda,\zeta)$ solves the SRE-J is equivalent to proving
$\mathscr R_P\equiv0$.

Under the standard condition, $\widehat R>0$ follows directly from the
positivity assumptions. For the present indefinite problem, however, the
positive definiteness of $\widehat R$ must be derived from the uniform
convexity condition. The proof is organized as follows:

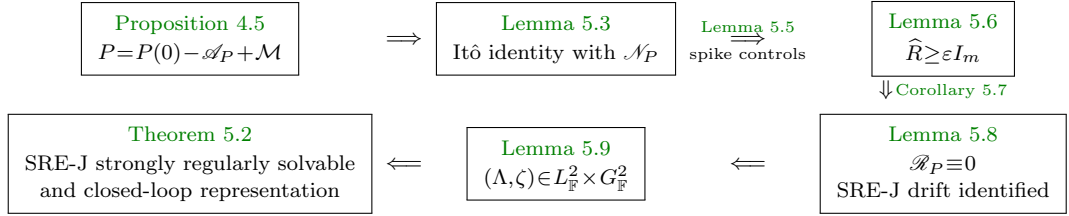
\begingroup
\footnotesize
\setlength{\arraycolsep}{3pt}
\setlength{\fboxsep}{3pt}
\renewcommand{\arraystretch}{1.05}
\[
\begin{array}{@{}c@{\hspace{0.6em}}c@{\hspace{0.6em}}c@{\hspace{0.6em}}c@{\hspace{0.6em}}c@{\hspace{0.6em}}c@{\hspace{0.6em}}c@{}}
\fbox{$\begin{array}{c}
\text{\proref{prop:Psemimartingale}}\\[1pt]
P=P(0)-\mathscr A_P+\mathcal M
\end{array}$}
&\Longrightarrow&
\fbox{$\begin{array}{c}
\text{\lemref{lem:jumpito}}\\[1pt]
\text{It\^o identity with }\mathscr N_P
\end{array}$}
&\overset{\text{\lemref{lem:localized_error}}}{\underset{\text{spike controls}}{\Longrightarrow}}
&
\fbox{$\begin{array}{c}
\text{\lemref{lem:Rhat_unifpos}}\\[1pt]
\widehat R\ge\varepsilon I_m
\end{array}$}
\\[6pt]
&&&&\Downarrow\;{\scriptstyle\text{\corref{cor:complsq_identity}}}\\[4pt]
\fbox{$\begin{array}{c}
\text{\thmref{thm:SREsolve}}\\[1pt]
\text{SRE-J strongly regularly solvable}\\
\text{and closed-loop representation}
\end{array}$}
% &\Longleftarrow&
% \fbox{$\begin{array}{c}
% \text{\lemref{lem:prime_admissibility}}\\[1pt]
% \text{verification identity}\\
% \text{and optimal feedback}
% \end{array}$}
% &\overset{\text{\corref{cor:complsq_identity}}}{\Longleftarrow}&
&\Longleftarrow&
\fbox{$\begin{array}{c}
\text{\lemref{lem:BMO}}\\[1pt]
(\Lambda,\zeta)\in L^2_{\mathbb F}\times G^2_{\mathbb F}
% \text{feedback admissibility}
\end{array}$}
&\Longleftarrow&
\fbox{$\begin{array}{c}
\text{\lemref{lem:SREidentify}}\\[1pt]
\mathscr R_P\equiv0\\
\text{SRE-J drift identified}
\end{array}$}
\end{array}
\]
\endgroup

\subsection[It\^o identity with jumps and positivity of R hat]
{It\^o identity with jumps and positivity of $\widehat R$}\label{sec:Rhatpos}

Throughout this subsection we use one localization. By the local
square-integrability of $(\Lambda,\zeta)$ in \eqref{Psemimartingale_PRP} and the
definition of $\mathscr N_P$, the increasing process
\[
  \mathcal K(t):=
  \int_0^t\Big(|\Lambda(s)|^2+
  \int_Z|\zeta(s,z)|^2\lambda(\mathrm dz)\Big)\mathrm ds
  +|\mathscr N_P|((0,t])
\]
is finite on $[0,T]$ pathwise. We choose stopping times $\tau_k\uparrow T$
such that
\begin{equation}\label{localization_sequence}
  \int_0^{\tau_k}\!\Big(|\Lambda(s)|^2
  +\int_Z|\zeta(s,z)|^2\lambda(\mathrm{d} z)\Big)\mathrm{d} s
  +|\mathscr N_P|\big((0,\tau_k]\big)\le k,
  \qquad\text{a.s.}
\end{equation}
Since $\mathcal K(T)<\infty$ a.s., we have $\mathbb P(\tau_k=T)\to1$. By the definition \eqref{RShat}, the boundedness of $P$ and of the coefficients, and $\lambda(Z)<\infty$, we have
\begin{equation}\label{RS_pointwise_bound}
  |\widehat R(s)|+|\widehat S(s)|
  \le K\bigg(1+|\Lambda(s)|
  +\Big(\int_Z|\zeta(s,z)|^2\lambda(\mathrm dz)\Big)^{1/2}\bigg),
\end{equation}
so that \eqref{localization_sequence} and the Cauchy--Schwarz inequality give
the pathwise bounds
\begin{equation}\label{RS_local_bound}
  \int_0^{\tau_k}\Big(|\widehat R(s)|+|\widehat S(s)|^2\Big)\,\mathrm ds
  \le K(1+k),
  \qquad\text{a.s.},\quad k\ge1 .
\end{equation}

% [Raw jump-It\^o identity]
Throughout the remainder of this section, $P$ denotes the bounded
c\`adl\`ag semimartingale in \proref{prop:Psemimartingale},
and $\Gamma^\circ$ and $\mathscr N_P$ are defined by
\eqref{Gammacirc_def}--\eqref{rawresidualmeasure}.

 \begin{lem}\label{lem:jumpito}
Let \ref{SLQH1}--\ref{SLQH2} and \eqref{N_0uniformalconvex} hold.
Fix $\sigma_0\in\mathcal T[0,T]$,
$\xi\in L^2_{\mathcal F_{\sigma_0}}(\Omega;\mathbb R^n)$,
and $u\in\mathcal U[\sigma_0,T]$, and let
$X=X^{\sigma_0,\xi;u}$ be the corresponding state.
Let $\sigma,\theta\in\mathcal T[\sigma_0,T]$ satisfy
$\sigma\le\theta\le\sigma\vee\tau_k$ for some $k\ge1$.
Suppose that
\begin{equation}\label{jumpito_integrability}
\mathbb E\bigg[\int_\sigma^\theta
\Big(
|\langle\widehat R(s)u(s),u(s)\rangle|
+
|\langle\widehat S(s)X(s-),u(s)\rangle|
\Big)\mathrm ds\bigg]<\infty .
\end{equation}
Then
\begin{equation}\label{jumpitoid_stopped}
\begin{aligned}
&\mathbb E\langle P(\theta)X(\theta),X(\theta)\rangle
-\mathbb E\langle P(\sigma)X(\sigma),X(\sigma)\rangle
+\mathbb E\bigg[\int_\sigma^\theta \ell(s,X(s),u(s))\,\mathrm ds\bigg]\\
=&\;
\mathbb E\bigg[\int_\sigma^\theta
\langle\widehat R(s)u(s),u(s)\rangle
+2\langle\widehat S(s)X(s-),u(s)\rangle
\mathrm ds\bigg]
+\mathbb E\bigg[\int_{(\sigma,\theta]}
X(s-)^\top\mathrm d\mathscr N_P(s)X(s-)\bigg].
\end{aligned}
\end{equation}
Moreover, \eqref{jumpito_integrability} holds whenever $u$ is bounded on
$(\sigma,\theta]$.
\end{lem}

\begin{proof}
Set
\(
  \beta:=AX_-+Bu,
  \gamma:=CX_-+Du,
  \delta(s,z):=E(s,z)X(s-)+F(s,z)u(s).
\)
Then
\(
  \mathrm dX(s)=\beta(s)\mathrm ds+\gamma(s)\mathrm dW(s)
  +\int_Z\delta(s,z)\widetilde N(\mathrm ds,\mathrm dz).
\)
Applying It\^o's formula for c\`{a}dl\`{a}g semimartingales
(see, e.g., Jacod and
Shiryaev~\cite[Theorem~I.4.57]{jacod2003limit}) to
$\langle P(s)X(s),X(s)\rangle$ on $(\sigma,\theta]$ and adding
$\ell(s,X(s),u(s))\,\mathrm ds$, we obtain
\begin{equation}\label{P-complete}
\begin{aligned}
  &\mathrm d\langle P(s)X(s),X(s)\rangle
  +\ell(s,X(s),u(s))\,\mathrm ds \\
  =\;&X(s-)^\top\mathrm d\mathscr N_P(s)\,X(s-)
  +2\langle\widehat S(s)X(s-),u(s)\rangle\,\mathrm ds
  +\langle\widehat R(s)u(s),u(s)\rangle\,\mathrm ds
  +\mathrm d\mathfrak m^u(s),
\end{aligned}
\end{equation}
where $\mathfrak m^u$ is the local martingale
\begin{equation*}\label{eq:mu_martingale}
\begin{aligned}
  \mathfrak m^u(s)
  =&\int_\sigma^s
  \big(X_-^\top\Lambda\,X_-+2\,X_-^\top P_-\gamma\big)\,\mathrm dW(r) +\int_\sigma^s\!\int_Z
  \Big[(X_-+\delta)^\top(P_-+\zeta)(X_-+\delta)
  -X_-^\top P_-\,X_-\Big]\widetilde N(\mathrm dr,\mathrm dz).
\end{aligned}
\end{equation*}

Let $(\nu_j)_{j\ge1}$ be a sequence of stopping times with $\nu_j\ge\sigma$, $\nu_j\uparrow\infty$ a.s., such that
$\mathfrak m^u(\cdot\wedge\nu_j)$ is a martingale on $[\sigma,T]$.
Set $\theta_j:=\theta\wedge\nu_j$.
Integrating \eqref{P-complete} over $(\sigma,\theta_j]$ and taking
expectations gives
\[
\begin{aligned}
&\mathbb E\langle P(\theta_j)X(\theta_j),
X(\theta_j)\rangle
-\mathbb E\langle P(\sigma)X(\sigma),X(\sigma)\rangle
+\mathbb E\bigg[\int_\sigma^{\theta_j}
\ell(s,X(s),u(s))\,\mathrm ds\bigg]\\
=&\;
\mathbb E\bigg[\int_\sigma^{\theta_j}
\Big[
\langle\widehat R(s)u(s),u(s)\rangle
+2\langle\widehat S(s)X(s-),u(s)\rangle
\Big]\mathrm ds\bigg]+
\mathbb E\bigg[\int_{(\sigma,\theta_j]}
X(s-)^\top\mathrm d\mathscr N_P(s)X(s-)\bigg].
\end{aligned}
\]
Letting $j\to\infty$ and applying the dominated convergence theorem gives \eqref{jumpitoid_stopped}. Indeed, the terminal
terms are dominated by
$\|P\|_\infty\sup_{r\in[\sigma_0,\theta]}|X(r)|^2$, the running cost is integrable by
\ref{SLQH1}--\ref{SLQH2}, $u\in\mathcal U[\sigma_0,T]$, and \lemref{existunique}.
Noting that $(\sigma,\theta]\subset(0,\tau_k]$ on $\{\sigma<\theta\}$, we see that
\[
  \mathbb E\bigg[\int_{(\sigma,\theta]}|X(s-)|^2\,|\mathrm d\mathscr N_P|(s)\bigg]
  \le
  k\mathbb E \Big[\sup_{r\in[\sigma_0,\theta]}|X(r)|^2\Big]<\infty .
\]
In addition, the two terms involving $\widehat R$ and $\widehat S$ are integrable by
\eqref{jumpito_integrability}. 

Finally, suppose $|u|\le M$ on $(\sigma,\theta]$. By \eqref{RS_local_bound}
and \lemref{existunique},
\[
  \mathbb E\bigg[\int_\sigma^\theta
  \Big(|\langle\widehat R(s)u(s),u(s)\rangle|
  +|\langle\widehat S(s)X(s-),u(s)\rangle|\Big)\mathrm ds\bigg]
  \le K(1+k)\Big(M^2+M\,\mathbb E\Big[\sup_{r\in[\sigma_0,T]}|X(r)|\Big]\Big)<\infty,
\]
which verifies the condition in \eqref{jumpito_integrability}.
\end{proof}
 
% The small-interval argument below tests the uniform convexity
% \eqref{N_0uniformalconvex} on stochastic intervals
% $[\sigma,\tau]$ with terminal weight $P(\tau)$. 

Now, we introduce the following stopped problem.
% $\mathbf{(SLQ\text{-}J)}_{\sigma}^\tau$. 
Let $\sigma,\tau\in\mathcal T[0,T]$ with $\sigma\le\tau$. For
$\xi\in L^2_{\mathcal F_\sigma}(\Omega;\mathbb R^n)$ and
$u\in\mathcal U[\sigma,\tau]$, let
$X^{\sigma,\xi;u}$ be the solution of the state equation \eqref{state} on
$[\sigma,\tau]$ with initial condition $X^{\sigma,\xi;u}(\sigma)=\xi$. Define
\begin{equation}\label{Jtaumin}
\begin{aligned}
  \widehat J^\tau(\sigma,\xi;u)
  &:=
  \mathbb E_\sigma\left[
    \int_\sigma^\tau
    \ell\big(s,X^{\sigma,\xi;u}(s),u(s)\big)\,\mathrm ds
    +\big\langle P(\tau)X^{\sigma,\xi;u}(\tau),
    X^{\sigma,\xi;u}(\tau)\big\rangle
  \right],\\
  J^\tau(\sigma,\xi;u)
  &:=
  \mathbb E\,\widehat J^\tau(\sigma,\xi;u),
  \qquad
  V^\tau(\sigma,\xi)
  :=
  \inf_{u\in\mathcal U[\sigma,\tau]}J^\tau(\sigma,\xi;u).
\end{aligned}
\end{equation}
% We call this problem $\mathbf{(SLQ\text{-}J)}_{\sigma}^\tau$.

% For $\sigma,\tau\in\mathcal T[0,T]$ with
% $\sigma\le\tau$, $\xi\in L^2_{\mathcal F_\sigma}(\Omega;\mathbb R^n)$, and
% $u\in\mathcal U[\sigma,\tau]$, define
% \begin{equation}\label{Jtaumin}
% \begin{aligned}
%   \widehat J^\tau(\sigma,\xi;u)
%   &:=
%   \mathbb E_\sigma\left[
%     \int_\sigma^\tau \ell(s,X(s),u(s))\,\mathrm ds
%     +\langle P(\tau)X(\tau),X(\tau)\rangle
%   \right],\\
%   J^\tau(\sigma,\xi;u)
%   &:=
%   \mathbb E\,\widehat J^\tau(\sigma,\xi;u),
%   \qquad
%   V^\tau(\sigma,\xi)
%   :=
%   \inf_{u\in\mathcal U[\sigma,\tau]}J^\tau(\sigma,\xi;u),
% \end{aligned}
% \end{equation}
% where $X$ is the state on $[\sigma,\tau]$ with $X(\sigma)=\xi$.

\begin{prop}\label{prop:stoppedunifconv}
Let \ref{SLQH1}--\ref{SLQH2} hold and suppose
\eqref{N_0uniformalconvex} holds. Then, for all stopping times
$\sigma\le\tau$,
\begin{equation}\label{eq:stopped_unifconv}
  J^\tau(\sigma,0;u)
  \ge
  \varepsilon\,\mathbb E\bigg[\int_\sigma^\tau |u(s)|^2\,\mathrm ds\bigg],
  \qquad u\in\mathcal U[\sigma,\tau].
\end{equation}
Consequently, the stopped problem is uniquely open-loop solvable, and
\(
  V^\tau(\sigma,\xi)
  =
  \mathbb E\langle P(\sigma)\xi,\xi\rangle .
\)
\end{prop}

\begin{proof}
Fix $u\in\mathcal U[\sigma,\tau]$, and let $X$ be the state on
$[\sigma,\tau]$ starting from $0$. By \lemref{lem:H3propagation} and the
Hilbert-space method in \thmref{thm:Prandomtime}, the problem starting
from $(\tau,X(\tau))$ admits a unique optimal continuation
$v^*\in\mathcal U[\tau,T]$. Set
\[
  \bar u:=u\mathbf 1_{[\sigma,\tau]}+v^*\mathbf 1_{(\tau,T]} .
\]
Then $\bar u\in\mathcal U[\sigma,T]$. Since $v^*$ attains the infimum at
$(\tau,X(\tau))$, the optimality of $v^*$ and
\thmref{thm:Prandomtime} give
$\langle P(\tau)X(\tau),X(\tau)\rangle
=\widehat V(\tau,X(\tau))
=\widehat J(\tau,X(\tau);v^*)$, and therefore
$J^\tau(\sigma,0;u)=J(\sigma,0;\bar u)$.
Applying \lemref{lem:H3propagation} at $\sigma$ yields
\[
  J^\tau(\sigma,0;u)
  =J(\sigma,0;\bar u)
  \ge\varepsilon\,\mathbb E\bigg[\int_\sigma^T|\bar u(s)|^2\,\mathrm ds\bigg]
  \ge\varepsilon\,\mathbb E\bigg[\int_\sigma^\tau|u(s)|^2\,\mathrm ds\bigg] .
\]
This proves \eqref{eq:stopped_unifconv}. Unique open-loop solvability of the
stopped problem then follows from \thmref{Opensolve} and \remref{rem:stochastic_interval}.
Finally, \proref{prop:DPPgeneral} gives
$\operatorname*{ess\,inf}_{u\in\mathcal U[\sigma,\tau]}
\widehat J^\tau(\sigma,\xi;u)
=\widehat V(\sigma,\xi)
=\langle P(\sigma)\xi,\xi\rangle$.
Taking expectations yields $V^\tau(\sigma,\xi)=\mathbb E\langle P(\sigma)\xi,\xi\rangle$.
\end{proof}

We next derive the positivity of $\widehat R$ from the uniform convexity.
Fix $k\ge1$, $\eta\in\mathbb R^m$, $0\le a<b\le T$, and $h\in(0,1)$.
Set $M_h:=\lceil(b-a)/h\rceil$ and
\begin{equation}\label{partition_def}
  t_i:=(a+ih)\wedge b,
  \qquad
  \rho_i:=t_i\vee\big(t_{i+1}\wedge\tau_k\big),
  \qquad i=0,1,\ldots,M_h-1.
\end{equation}
Let $X^i$ be the state on $[t_i,\rho_i]$ with initial condition $X^i(t_i)=0$ 
under the spike control $u(s)=\eta\mathbf 1_{(t_i,\rho_i]}(s)$. The intervals $(t_i,\rho_i]$
are pairwise disjoint, and on $\{a<\tau_k\}$ their union is
$(a,b\wedge\tau_k]$. On $\{\tau_k\le t_i\}$ one has $\rho_i=t_i$, and the
corresponding interval and all integrals over it are empty.
The next lemma shows that the \(\widehat S\) and \(\mathscr N_P\) terms
generated by the spike controls vanish after summing over the above partition
as \(h\downarrow0\).

% [Summed short-interval error estimate]
\begin{lem}\label{lem:localized_error}
Let \ref{SLQH1}--\ref{SLQH2} and \eqref{N_0uniformalconvex} hold, and fix
$k\ge1$ and $\eta\in\mathbb R^m$. There exists a constant $K>0$, independent
of $a$, $b$, $h$, and $i$, such that
\begin{equation}\label{shortinterval_cond2}
  \mathbb E\Big[\sup_{r\in[t_i,\rho_i]}|X^i(r)|^2\Big]\le Kh|\eta|^2,
  \qquad
  \mathbb E\Big[\sup_{r\in[t_i,\rho_i]}|X^i(r)|^4\Big]\le Kh|\eta|^4 .
\end{equation}
Moreover,
\begin{equation}\label{summed_error}
  \lim_{h\downarrow0}\,
  \sum_{i=0}^{M_h-1}\mathbb E\bigg[
  \int_{t_i}^{\rho_i}
  \big|\big\langle\widehat S(s)X^i(s-),\eta\big\rangle\big|\,\mathrm ds
  +\int_{(t_i,\rho_i]}|X^i(s-)|^2\,|\mathrm d\mathscr N_P|(s)
  \bigg]=0 .
\end{equation}
\end{lem}

\begin{proof}
\emph{Step 1: moment estimates.}
The second-moment bound in \eqref{shortinterval_cond2} is standard;
see, e.g., Zhang, Dong, and
Meng~\cite[Lemma~2.4]{zhang2020backward}.
For the fourth moment, set
$\delta^i(s,z):=E(s,z)X^i(s-)+F(s,z)\eta$. By the Burkholder--Davis--Gundy
inequality, we have
\[
  \mathbb E\bigg[\sup_{t\in[t_i,\rho_i]}
  \Big|\int_{t_i}^{t}\!\int_Z\delta^i(s,z)\,\widetilde N(\dd s,\dd z)\Big|^4\bigg]
  \le K\,\mathbb E\bigg[
  \Big(\int_{t_i}^{\rho_i}\!\int_Z|\delta^i(s,z)|^2\lambda(\dd z)\,\dd s\Big)^2
  +\int_{t_i}^{\rho_i}\!\int_Z|\delta^i(s,z)|^4\lambda(\dd z)\,\dd s
  \bigg].
\]
Note the terms containing $F\eta$ are bounded by $Kh|\eta|^4$, while the terms
containing $EX^i_-$ are bounded by
\(
  K\int_{t_i}^{\rho_i}
  \mathbb E\Big[\sup_{r\in[t_i,s]}|X^i(r)|^4\Big]\dd s .
\)
The drift and Brownian terms are estimated similarly, contributing only
state-dependent integral terms of the same form and an additional term bounded
by $Kh^2|\eta|^4$. Hence
\[
  \mathbb E\Big[\sup_{r\in[t_i,s]}|X^i(r)|^4\Big]
  \le Kh|\eta|^4
  +K\int_{t_i}^{s}
  \mathbb E\Big[\sup_{r\in[t_i,q]}|X^i(r)|^4\Big]\dd q,
  \qquad s\in[t_i,\rho_i].
\]
It follows from Gronwall's inequality that the fourth-moment bound in
\eqref{shortinterval_cond2} holds.

\emph{Step 2: the $\widehat S$-terms.}
Set $a_i:=\mathbb E\int_{t_i}^{\rho_i}|\widehat S(s)|^2\,\mathrm ds$. By
\eqref{RS_local_bound}, $\sum_{i=0}^{M_h-1}a_i\le K(1+k)$. The Cauchy--Schwarz inequality
and \eqref{shortinterval_cond2} give
\[
  \mathbb E\bigg[\int_{t_i}^{\rho_i}
  \big|\langle\widehat S(s)X^i(s-),\eta\rangle\big|\,\mathrm ds\bigg]
  \le|\eta|\,a_i^{1/2}
  \Big(\mathbb E\Big[\int_{t_i}^{\rho_i}|X^i(s-)|^2\,\mathrm ds\Big]\Big)^{1/2}
  \le K|\eta|^2h\,a_i^{1/2}.
\]
Summing over $i$ and applying the Cauchy--Schwarz inequality, we have (noting $M_hh\le T+1$)
\[
  \sum_{i=0}^{M_h-1}
  \mathbb E\bigg[\int_{t_i}^{\rho_i}
  \big|\langle\widehat S(s)X^i(s-),\eta\rangle\big|\,\mathrm ds\bigg]
  \le K|\eta|^2h\sqrt{M_h}\Big(\sum_{i=0}^{M_h-1}a_i\Big)^{1/2}
  % \le K(1+k)^{1/2}|\eta|^2\sqrt{h}\;\rightarrow\;0.
\]
The right-hand side is bounded by \(K(1+k)^{1/2}|\eta|^2\sqrt{h}\), which tends to zero as \(h\downarrow0\).
% where we used $Mh\le b-a+h\le T+1$.

\emph{Step 3: the $\mathscr N_P$-terms.}
The fourth-moment estimate for $X^i$ is only of order $h$ (rather than $h^2$),
because a single Poisson jump of size $|F\eta|$ occurs with probability
$O(h)$.  Such $O(h)$ jump contributions also arise in
spike variation estimates for stochastic control systems with jumps;
see~\cite{song2020maximum,zheng2023global}.
To overcome this, we decompose $X^i=J^i+Y^i$, where
\[
  J^i(s):=\int_{t_i}^s\!\int_ZF(r,z)\eta\,\mathbf 1_{(t_i,\rho_i]}(r)\,
  N(\mathrm dr,\mathrm dz).
\]
Then $Y^i:=X^i-J^i$ on $[t_i,\rho_i]$ satisfies
\begin{equation}\label{Yi_equation}
\begin{aligned}
  \mathrm dY^i(s)
  =&\;\Big[A(s)X^i(s-)+B(s)\eta
    -\int_ZF(s,z)\eta\,\lambda(\mathrm dz)\Big]\mathrm ds \\
  &\;+\big[C(s)X^i(s-)+D(s)\eta\big]\mathrm dW(s)
  +\int_ZE(s,z)X^i(s-)\,\widetilde N(\mathrm ds,\mathrm dz),
  \quad Y^i(t_i)=0.
\end{aligned}
\end{equation}
Applying the Burkholder--Davis--Gundy
inequality to \eqref{Yi_equation} and using
\eqref{shortinterval_cond2} for \(X^i\), we obtain
% Using the estimates in Step~1 and \eqref{shortinterval_cond2}, we obtain
\begin{equation}\label{Yi_fourth_estimate}
  \mathbb E\Big[\sup_{r\in[t_i,\rho_i]}|Y^i(r)|^4\Big]\le Kh^2|\eta|^4 .
\end{equation}
Let $v_i:=|\mathscr N_P|((t_i,\rho_i])$. By
\eqref{localization_sequence}, $\sum_{i=0}^{M_h-1} v_i\le k$ and $v_i\le k$ a.s. Hence
\[
  \sum_{i=0}^{M_h-1}(\mathbb E v_i^2)^{1/2}
  \le \sqrt{M_h}\bigg(\sum_{i=0}^{M_h-1}\mathbb E v_i^2\bigg)^{1/2}
  \le \sqrt{M_h}\bigg(
  \mathbb E\Big[\max_{0\le i\le M_h-1} v_i\sum_{i=0}^{M_h-1} v_i\Big]\bigg)^{1/2}
  \le k\sqrt{M_h} .
\]
Therefore, using the estimate \eqref{Yi_fourth_estimate} for \(Y^i\),
\begin{equation}\label{Ypart_bound}
\begin{aligned}
 \sum_{i=0}^{M_h-1}\mathbb E\bigg[\int_{(t_i,\rho_i]}|Y^i(s-)|^2
 |\mathrm d\mathscr N_P|(s)\bigg]
 &\le
 K h|\eta|^2\sum_{i=0}^{M_h-1}(\mathbb E v_i^2)^{1/2}     
 \le Kk|\eta|^2h\sqrt{M_h}
 \le Kk|\eta|^2\sqrt h\to0 .
\end{aligned}
\end{equation}

It remains to handle \(J^i\). Let
\(\mathsf{N}_i=N((t_i,t_{i+1}]\times Z)\), and set
\(
  \mathsf{N}_T:=N((0,T]\times Z).
\)
Denote the Poisson jump times on \((0,T]\) by
\(T_1,\ldots,T_{\mathsf{N}_T}\).
We estimate the contribution of $J^i$ by conditioning on
$\mathsf{N}_i$:
(i) On \(\{\mathsf{N}_i=0\}\), \(J^i=0\). 
(ii) On \(\{\mathsf{N}_i=1\}\), denoting the unique jump time by
\(T_{j(i)}\), we have \(J^i(s-)=0\) for \(s\le T_{j(i)}\) and
\(|J^i(s-)|\le K|\eta|\) afterwards. Set
\(
  \beta_j(h):=\mathbf 1_{\{T_j<\tau_k\}}\,
  |\mathscr N_P|((T_j,(T_j+h)\wedge\tau_k]).
\)
Then
\[
  \int_{(t_i,\rho_i]}|J^i(s-)|^2|\mathrm d\mathscr N_P|(s)
  \le K|\eta|^2\beta_{j(i)}(h)
  \quad\text{on }\{\mathsf{N}_i=1\}.
\]
(iii) On \(\{\mathsf{N}_i\ge2\}\), \(|J^i_-|\le K|\eta|\mathsf{N}_i\) and \(v_i\le k\). Therefore, it follows from (i)-(iii) that
\begin{equation}\label{Jpart_bound}
  \sum_{i=0}^{M_h-1}\mathbb E\bigg[\int_{(t_i,\rho_i]}|J^i(s-)|^2\,|\mathrm d\mathscr N_P|(s)\bigg]
  \le K|\eta|^2\,\mathbb E\bigg[\sum_{j:T_j\le T}\beta_j(h)\bigg]
  +K(1+k)|\eta|^2\sum_{i=0}^{M_h-1}\mathbb E\big[\mathsf{N}_i^2\mathbf 1_{\{\mathsf{N}_i\ge2\}}\big]=:\mathrm{I}+\mathrm{II}.
\end{equation}
For the first term, \(\beta_j(h)\downarrow0\) a.s., because
\((T_j,(T_j+h)\wedge\tau_k]\downarrow\varnothing\). Moreover,
\(
  \sum_{j:T_j\le T}\beta_j(h)\le k\mathsf{N}_T\in L^1(\Omega).
\)
Hence the term \((\mathrm{I})\to0\) by the dominated convergence
theorem. For the second term,
\(\mathsf N_i\) is Poisson with mean
\(\mu_i=\lambda(Z)(t_{i+1}-t_i)\le\lambda(Z)h\), and
\[
  \mathbb E\big[\mathsf N_i^2\mathbf 1_{\{\mathsf N_i\ge2\}}\big]
  =\mu_i^2+\mu_i(1-e^{-\mu_i})
  \le Kh^2 .
\]
Thus
\(
  \sum_{i=0}^{M_h-1}
  \mathbb E\big[\mathsf N_i^2\mathbf 1_{\{\mathsf N_i\ge2\}}\big]
  \le KM_hh^2\le K(T+1)h,
\)
and the term \((\mathrm{II})\to0\). Combining \eqref{Ypart_bound}, \eqref{Jpart_bound}
and
\(
  |X^i_-|^2\le2|Y^i_-|^2+2|J^i_-|^2
\)
proves the \(\mathscr N_P\)-term in \eqref{summed_error}.
\end{proof}

% [Uniform positive definiteness of $\widehat R$]
\begin{lem}\label{lem:Rhat_unifpos}
Let \ref{SLQH1}--\ref{SLQH2} and \eqref{N_0uniformalconvex} hold. 
Then the process \(\widehat R\) associated with the semimartingale \(P\) satisfies
\begin{equation}\label{Rhat_unifpos_statement}
  \widehat R(s)\ge\varepsilon I_m,
  \qquad
  \text{for a.e. } s\in[0,T],\ \text{a.s.}
\end{equation}
where $\varepsilon$ is the constant in \eqref{N_0uniformalconvex}. 
% In particular, \eqref{Rhatposdef} holds with $\delta=\varepsilon$.
\end{lem}

\begin{proof}
Fix \(k\ge1\), \(\eta\in\mathbb R^m\), \(0\le a<b\le T\), and
\(A_0\in\mathcal F_a\).

% localized spike controls.
\emph{Step 1:}
For the partition \eqref{partition_def}, define the spike controls
\[
  u_i(s):=\eta\,\mathbf 1_{A_0}\mathbf 1_{(t_i,\rho_i]}(s),
  \qquad i=0,\ldots,M_h-1 .
\]
Since \(A_0\in\mathcal F_a\subset\mathcal F_{t_i}\), the state driven by
\(u_i\) from \(0\) at \(t_i\) is \(\mathbf 1_{A_0}X^i\). 
Note that \(\rho_i\le t_i\vee\tau_k\). Applying \lemref{lem:jumpito} with $\sigma_0=\sigma=t_i$ and $\theta=\rho_i$, and using \proref{prop:stoppedunifconv}, we obtain
% Since \(A_0\in\mathcal F_a\subset\mathcal F_{t_i}\), the state driven by
% \(u_i\) from \(0\) at \(t_i\) is \(\mathbf 1_{A_0}X^i\). 
% If \(\rho_i=t_i\), the desired inequality below is trivial. We therefore
% assume \(\rho_i>t_i\), then \(\rho_i\le\tau_k\).  
% Applying \lemref{lem:jumpito} on
% \([t_i,\rho_i]\), and using \proref{prop:stoppedunifconv}, gives
\[
\begin{aligned}
J^{\rho_i}(t_i,0;u_i)=&\mathbb E\bigg[\int_{t_i}^{\rho_i}\mathbf 1_{A_0}
  \big(\langle\widehat R(s)\eta,\eta\rangle
  +2\langle\widehat S(s)X^i(s-),\eta\rangle\big)\,\mathrm ds\bigg]  
+\mathbb E\bigg[\mathbf 1_{A_0}\!\int_{(t_i,\rho_i]}
  \langle X^i(s-),\mathrm d\mathscr N_P(s)X^i(s-)\rangle\bigg]
\\
\ge&
\varepsilon|\eta|^2\,
\mathbb E\big[\mathbf 1_{A_0}(\rho_i-t_i)\big].
\end{aligned}
\]
Summing over \(i=0,\ldots,M_h-1\), using the disjoint union
\(
  \bigcup_{i=0}^{M_h-1}(t_i,\rho_i]=(a,b\wedge\tau_k],
\)
and noting that the two error terms vanish as \(h\downarrow0\) by
\lemref{lem:localized_error}, we obtain
\begin{equation}\label{rectangle_Rhat}
  \mathbb E\bigg[\mathbf 1_{A_0}\int_a^b\phi_{k,\eta}(s)\,\mathrm ds\bigg]\ge0,
  \qquad
  \phi_{k,\eta}(s):=\mathbf 1_{\{s\le\tau_k\}}
  \big(\langle\widehat R(s)\eta,\eta\rangle-\varepsilon|\eta|^2\big).
\end{equation}
The inequality \eqref{rectangle_Rhat} holds for every
\(0\le a<b\le T\) and every \(A_0\in\mathcal F_a\).

% finite-variation submartingale argument.
\emph{Step 2:}
By \eqref{RS_local_bound},
\(
  \int_0^T|\phi_{k,\eta}(s)|\,\mathrm ds
  \le K(1+k)(1+|\eta|^2),\text{a.s.}
\)
Hence
\(
  M_{k,\eta}(t):=\int_0^t\phi_{k,\eta}(s)\,\mathrm ds
\)
is a bounded continuous adapted process with paths of finite variation on $[0,T]$. Moreover,
\eqref{rectangle_Rhat} is exactly
\[
  \mathbb E\big[
    \mathbf 1_{A_0}\big(M_{k,\eta}(b)-M_{k,\eta}(a)\big)
  \big]\ge0,
  \qquad A_0\in\mathcal F_a,\quad 0\le a<b\le T .
\]
Thus \(M_{k,\eta}\) is a submartingale. Since \(M_{k,\eta}\) is continuous and adapted, it is predictable. Hence, by \lemref{lem:FVsubmart}, \(M_{k,\eta}\) is nondecreasing. Consequently,
% Thus \(M_{k,\eta}\) is a submartingale. Let \(M_{k,\eta}=N+\mathscr A\) be its Doob--Meyer decomposition, where \(N\) is a martingale and \(\mathscr A\) is predictable and nondecreasing, with
% \(N(0)=\mathscr A(0)=0\). Since \(M_{k,\eta}\) is continuous and adapted, it is
% predictable. Therefore
% \(
%   N=M_{k,\eta}-\mathscr A
% \)
% is a predictable martingale of finite variation.
% Hence, by Jacod and
% Shiryaev~\cite[Corollary~I.3.16]{jacod2003limit}, \(N\equiv0\).
% Therefore, \(M_{k,\eta}=\mathscr A\) is nondecreasing, and
$\phi_{k,\eta}(s)\ge0$ for a.e. $s\in[0,T]$, a.s.

\emph{Step 3:}
Taking the countable union of the exceptional sets in Step~2, we obtain that,
for every \((k,\eta)\in\mathbb N\times\mathbb Q^m\),
\[
  \mathbf 1_{\{s\le\tau_k\}}
  \big(\langle\widehat R(s)\eta,\eta\rangle
  -\varepsilon|\eta|^2\big)\ge0,
  \qquad
 \text{for a.e. }s\in[0,T],\ \text{a.s.}
\]
Since
\(\tau_k\uparrow T\) a.s.,
\(\mathbf 1_{\{s\le\tau_k\}}\uparrow1\) for a.e. $s\in[0,T]$, a.s.
Letting
\(k\to\infty\), we obtain
\[
  \langle\widehat R(s)\eta,\eta\rangle
  \ge\varepsilon|\eta|^2,
  \qquad
  \text{for } \eta\in\mathbb Q^m,\ \text{a.e. } s\in[0,T],\ \text{a.s.}
\]
By continuity of the quadratic form in \(\eta\), the inequality extends to all
\(\eta\in\mathbb R^m\). This proves \eqref{Rhat_unifpos_statement}.
\end{proof}

\subsection{Identification of the drift}\label{sec:driftid}

By \lemref{lem:Rhat_unifpos}, the feedback coefficient
$\Theta:=-\widehat R^{-1}\widehat S$ is well defined, and
$|\Theta|\le\varepsilon^{-1}|\widehat S|$, so \eqref{RS_local_bound} gives
\begin{equation}\label{Theta_local_bound}
  \int_0^{\tau_k}|\Theta(s)|^2\,\mathrm ds\le K(1+k),
  \qquad\text{a.s.},\quad k\ge1 .
\end{equation}
Similarly, since
$|\widehat S^\top\widehat R^{-1}\widehat S|\le\varepsilon^{-1}|\widehat S|^2$,
the definitions \eqref{rawresidualmeasure}--\eqref{completedresidual} and
\eqref{localization_sequence} yield
\begin{equation}\label{RP_local}
  |\mathscr R_P|\big((0,\tau_k]\big)
  \le|\mathscr N_P|\big((0,\tau_k]\big)
  +\varepsilon^{-1}\!\int_0^{\tau_k}|\widehat S(s)|^2\,\mathrm ds
  \le K(1+k),
  \qquad\text{a.s.}
\end{equation}

In view of the uniform positivity of \(\widehat R\) obtained in
\lemref{lem:Rhat_unifpos}, we apply completion of squares to the stopped
relation \eqref{jumpitoid_stopped} and obtain the following corollary.
The additional integrability condition \eqref{jumpito_integrability} in
\lemref{lem:jumpito} is removed below by a truncation technique together with
the convergence theorems.

% The truncation method below removes the additional integrability
% condition \eqref{jumpito_integrability} in \lemref{lem:jumpito}.

% [Completed-square identity]
\begin{coro}\label{cor:complsq_identity}
Let \ref{SLQH1}--\ref{SLQH2} and \eqref{N_0uniformalconvex} hold, and let
\(\mathscr R_P\) be defined by \eqref{completedresidual}.
\begin{enumerate}[label=(\roman*)]
\item
Let \(\sigma,\theta\in\mathcal T[0,T]\) satisfy
\(\sigma\le\theta\le\sigma\lor\tau_k\) for some \(k\ge1\). For
\(\xi\in L^2_{\mathcal F_\sigma}(\Omega;\mathbb R^n)\) and
\(u\in\mathcal U[\sigma,\theta]\), let \(X\) be the corresponding state with
\(X(\sigma)=\xi\). Then
\begin{equation}\label{complsq_local}
\begin{aligned}
  &\mathbb E\langle P(\theta)X(\theta),X(\theta)\rangle
  -\mathbb E\langle P(\sigma)\xi,\xi\rangle
  +\mathbb E\bigg[\int_\sigma^\theta\ell(s,X(s),u(s))\,\mathrm ds\bigg]  \\
  =&\;
  \mathbb E\bigg[\int_\sigma^\theta
  \big|\widehat R(s)^{1/2}\big(u(s)-\Theta(s)X(s-)\big)\big|^2\,\mathrm ds\bigg]  
  +\mathbb E\bigg[\int_{(\sigma,\theta]}
  \big\langle X(s-),\mathrm d\mathscr R_P(s)X(s-)\big\rangle\bigg] .
\end{aligned}
\end{equation}
% In particular, every term in \eqref{complsq_local} is finite.
\item
If, in addition, \(\mathscr R_P\equiv0\), then for every
\((t,\xi)\in\mathcal D\) and every \(u\in\mathcal U[t,T]\), with
\(X=X^{t,\xi;u}\),
\begin{equation}\label{complsq_global}
  J(t,\xi;u)
  =
  \mathbb E\langle P(t)\xi,\xi\rangle
  +\mathbb E\bigg[\int_t^T
  \big|\widehat R(s)^{1/2}\big(u(s)-\Theta(s)X(s-)\big)\big|^2\,\mathrm ds\bigg] .
\end{equation}
\end{enumerate}
\end{coro}

\begin{proof}
(i) Since \(\widehat R\ge\varepsilon I_m\), applying completion of squares with
\(\Theta=-\widehat R^{-1}\widehat S\) implies that
\begin{equation}\label{complsq_algebra}
  \langle\widehat Ru,u\rangle+2\langle\widehat SX_-,u\rangle
  =
  \big|\widehat R^{1/2}(u-\Theta X_-)\big|^2
  -\big\langle\widehat S^\top\widehat R^{-1}\widehat SX_-,X_-\big\rangle.
\end{equation}
For \(\ell\ge1\), define the stopping times
\(
  \alpha_\ell
  :=
  \theta\wedge\inf\big\{r\ge\sigma:
  \int_\sigma^r
  \big|\widehat R(s)^{1/2}\big(u(s)-\Theta(s)X(s-)\big)\big|^2\,\mathrm ds
  \ge\ell\big\}.
\)
On $(\sigma,\alpha_\ell]$, the integrability condition
\eqref{jumpito_integrability} holds. Indeed,
\[
  |\widehat R^{1/2}u|^2
  \le2\big|\widehat R^{1/2}(u-\Theta X_-)\big|^2
  +2\varepsilon^{-1}|\widehat S|^2|X_-|^2,
  \qquad
  |\langle\widehat SX_-,u\rangle|
  \le\tfrac12\varepsilon^{-1}|\widehat S|^2|X_-|^2
  +\tfrac12|\widehat R^{1/2}u|^2,
\]
while the definition of $\alpha_\ell$ and \eqref{RS_local_bound} give
\[
  \mathbb E\bigg[\int_\sigma^{\alpha_\ell}
  \big|\widehat R^{1/2}(u-\Theta X_-)\big|^2\mathrm ds\bigg]\le\ell,
  \qquad
  \mathbb E\bigg[\int_\sigma^{\alpha_\ell}|\widehat S|^2|X_-|^2\,\mathrm ds\bigg]
  \le K(1+k)\,\mathbb E\Big[\sup_{r\in[\sigma,\theta]}|X(r)|^2\Big]<\infty .
\]
Hence \lemref{lem:jumpito} applies on $(\sigma,\alpha_\ell]$. Substituting
\eqref{complsq_algebra} into \eqref{jumpitoid_stopped} 
and using
\(
  \mathrm d\mathscr N_P
  =
  \mathrm d\mathscr R_P
  +\widehat S^\top\widehat R^{-1}\widehat S\,\mathrm ds 
\)
gives
\eqref{complsq_local} with $\alpha_\ell$ in place of $\theta$.
It remains to let \(\ell\to\infty\). In view of \eqref{RP_local}, we have
\[
  \Big|
  \int_{(\sigma,\alpha_\ell]}
  \big\langle X(s-),\mathrm d\mathscr R_P(s)X(s-)\big\rangle
  \Big|
  \le
  |\mathscr R_P|((0,\tau_k])
  \sup_{r\in[\sigma,\theta]}|X(r)|^2
  \le
  K(1+k)\sup_{r\in[\sigma,\theta]}|X(r)|^2 .
\]
Therefore, the \(\mathscr R_P\)-term converges to the corresponding term on
\((\sigma,\theta]\) by the dominated convergence theorem. Similarly, the
term
\(\mathbb E\langle P(\alpha_\ell)X(\alpha_\ell),X(\alpha_\ell)\rangle\)
and the running-cost term
\(\mathbb E\int_\sigma^{\alpha_\ell}\ell(s,X(s),u(s))\,\mathrm ds\)
converge to their counterparts with \(\theta\) in place of \(\alpha_\ell\).
Furthermore, applying the monotone convergence theorem to
\(
  \mathbb E\int_\sigma^{\alpha_\ell}
  \big|\widehat R^{1/2}(u-\Theta X_-)\big|^2\,\mathrm ds
\)
gives \eqref{complsq_local}.

(ii) Fix \((t,\xi)\in\mathcal D\) and \(u\in\mathcal U[t,T]\). Set \(\theta_k:=t\vee\tau_k\). Applying part (i) with
\(\sigma=t\) and \(\theta=\theta_k\), and using \(\mathscr R_P\equiv0\), 
\[
  \mathbb E\langle P(\theta_k)X(\theta_k),X(\theta_k)\rangle
  -\mathbb E\langle P(t)\xi,\xi\rangle
  +\mathbb E\bigg[\int_t^{\theta_k}\ell(s,X(s),u(s))\,\mathrm ds\bigg]
  =
  \mathbb E\bigg[\int_t^{\theta_k}
  \big|\widehat R(s)^{1/2}(u(s)-\Theta(s) X(s-))\big|^2\,\mathrm ds\bigg] .
\]
Letting \(k\to\infty\), the boundary and running-cost terms converge by
dominated convergence, while the square term converges by monotone
convergence. Since \(P(T)=G\), we obtain \eqref{complsq_global}.
\end{proof}

% \begin{remark}\label{rem:complsq_portable}
% The proof uses only that \(P\) is a bounded c\`adl\`ag special semimartingale
% with locally square-integrable martingale part and that \(\widehat R\ge\delta
% I_m\) for some \(\delta>0\). Part (ii) also uses \(P(T)=G\). Hence the same
% completed-square identities hold for any triple
% \((\bar P,\bar\Lambda,\bar\zeta)\) with these properties, with the corresponding
% objects defined by the same formulas.
% \end{remark}

\begin{remark}\label{rem:complsq_portable}
We point out that, in the proof of \corref{cor:complsq_identity}, we have only used the following facts: \(P\) is a bounded c\`adl\`ag special semimartingale whose martingale integrands are locally square-integrable, and \(\widehat R\ge\delta I_m\) for some \(\delta>0\). In part~(ii), we have additionally used \(\mathscr R_P\equiv0\) and \(P(T)=G\). Consequently, \eqref{complsq_global} holds for any strongly regular solution of \eqref{SREJ}, with \(\widehat R\) and \(\Theta\) defined in terms of that solution.
\end{remark}

From \eqref{completedresidual}, it remains to prove
\(\mathscr R_P\equiv0\). This is the purpose of the next lemma.

% Thus, proving that $(P,\Lambda,\zeta)$ solves the SRE-J \eqref{SREJ} amounts
% to proving $\mathscr R_P\equiv0$ in \eqref{complsq_local}.

% [Identification of the drift]

\begin{lem}\label{lem:SREidentify}
Let \ref{SLQH1}--\ref{SLQH2} and \eqref{N_0uniformalconvex} hold. Then
\(\mathscr R_P\equiv0\). Consequently, the predictable finite-variation part
of \(P\) is absolutely continuous and
\begin{equation}\label{AP_identified}
  \mathrm d\mathscr A_P(s)
  =
  \Big(\Gamma^\circ(s)
  -\widehat S(s)^\top\widehat R(s)^{-1}\widehat S(s)\Big)\mathrm ds .
\end{equation}
Moreover, \((P,\Lambda,\zeta)\) satisfies the SRE-J \eqref{SREJ}, with
\((\Lambda,\zeta)\) locally square-integrable.
\end{lem}

\begin{proof}
Let \((T_j)_{j\ge1}\) be the jump times of the Poisson random
measure, and set \(T_0=0\). Since \(\lambda(Z)<\infty\), only finitely many
\(T_j\)'s lie in \([0,T]\), a.s. Fix \(j,k,q\ge1\), and set
\(\theta:=T_j\wedge T\wedge\tau_k\) and
\(\sigma:=(T_{j-1}+q^{-1})\wedge\theta\). Then
\(\sigma,\theta\in\mathcal T[0,T]\), \(\sigma\le\theta\le\tau_k\), and there
is no Poisson jump on \((\sigma,\theta)\).

Let \(\Phi\) be the fundamental matrix on \([\sigma,\theta]\) satisfying \(\Phi(\sigma)=I_n\) and
\[
  \mathrm d\Phi(s)=
  \Big(A+B\Theta-\int_Z(E+F\Theta)\lambda(\mathrm dz)\Big)(s)\Phi(s)\,\mathrm ds
  +(C+D\Theta)(s)\Phi(s)\,\mathrm dW(s),\quad \sigma< s\le \theta\le \tau_k.
\]
By \eqref{Theta_local_bound}, \(\Phi\) is well defined and invertible on \([\sigma,\theta]\).
% For the state corresponding to \(u=\Theta X_-\) and
% \(X(\sigma)=\xi\), we have \(X(s-)=\Phi(s-)\xi\) for \(s\in(\sigma,\theta]\).
Define the matrix measure on \((\sigma,\theta]\) by
\begin{equation}\label{pulledback_measure}
  \mathrm d\mathscr M(s)
  =
  \Phi(s-)^\top \mathrm d\mathscr R_P(s)\Phi(s-).
\end{equation}

We first prove that \(\mathrm d\mathscr M\succeq0\).\footnote{For an
\(\mathbb S^n\)-valued finite-variation measure \(\mu\), we write
\(\mathrm d\mu\succeq0\) if \(\langle y,\mathrm d\mu\,y\rangle\) is a
nonnegative scalar measure for every \(y\in\mathbb R^n\). We also write
\(\operatorname{tr}(\mathrm d\mu):=\sum_{i=1}^n
\langle e_i,\mathrm d\mu\,e_i\rangle\). If \(\mathrm d\mu\succeq0\) and
\(\operatorname{tr}(\mathrm d\mu)=0\), then \(\mathrm d\mu=0\).}
Fix \(y\in\mathbb R^n\), \(R\ge1\), and set
\(\theta_R:=\theta\wedge\inf\{r\ge\sigma:|\Phi(r)y|\ge R\}\). Let
\(\rho,\nu\) be stopping times such that \(\sigma\le\rho\le\nu\le\theta_R\),
and let \(A_0\in\mathcal F_\rho\).
Set \(\xi:=\mathbf 1_{A_0}\Phi(\rho)y\), and take \(u=\Theta X_-\) on
\([\rho,\nu]\). Then \(X(s-)=\mathbf 1_{A_0}\Phi(s-)y\) for
\(s\in(\rho,\nu]\). 
Since $|u|\le R|\Theta|$ on $(\rho,\nu]$, we have $u\in\mathcal U[\rho,\nu]$ by \eqref{Theta_local_bound}.
By \proref{prop:stoppedunifconv},
\(V^\nu(\rho,\xi)=\mathbb E[\langle P(\rho)\xi,\xi\rangle]\). Hence, applying
\corref{cor:complsq_identity}(i) on \([\rho,\nu]\), we obtain
\begin{equation}\label{dM-positive}
\begin{aligned}
0
\le
J^\nu(\rho,\xi;u)-\mathbb E[\langle P(\rho)\xi,\xi\rangle]  
=
\mathbb E\bigg[\int_{(\rho,\nu]}
\big\langle X(s-),\mathrm d\mathscr R_P(s)X(s-)\big\rangle  \bigg]
=
\mathbb E\bigg[
\mathbf 1_{A_0}\int_{(\rho,\nu]}
\langle y,\mathrm d\mathscr M(s)y\rangle
\bigg].
\end{aligned}
\end{equation}
Set
\(
  V_y^R(r):=
  \int_{(\sigma,(r\vee\sigma)\wedge\theta_R]}
  \langle y,\mathrm d\mathscr M(s)y\rangle, r\in[0,T].
\)
% By \eqref{RP_local}, \(V_y^R\) is integrable and of finite variation.
% Moreover, it follows from \eqref{dM-positive} that \(V_y^R\) is a submartingale.
% Since \(V_y^R\) is predictable and of
% finite variation, its Doob--Meyer decomposition gives \(V_y^R=N+A\), where
% \(A\) is predictable and increasing.
% Then \(N=V_y^R-A\) is a predictable martingale of finite variation.
% Hence, by Jacod and
% Shiryaev~\cite[Corollary~I.3.16]{jacod2003limit}, \(N\) is constant.
% Therefore \(V_y^R\) is
% nondecreasing
By \eqref{RP_local}, \(V_y^R\) is a predictable process of bounded variation.
Moreover, it follows from \eqref{dM-positive} that \(V_y^R\) is a submartingale.
By \lemref{lem:FVsubmart}, \(V_y^R\) is nondecreasing, and
\(\mathbf 1_{(\sigma,\theta_R]}\langle y,\mathrm d\mathscr M\,y\rangle\ge0\).
Letting \(R\to\infty\) and then \(y\) range over \(\mathbb Q^n\), we obtain
\(\mathrm d\mathscr M\succeq0\) on \((\sigma,\theta]\). Since \(\Phi\) is
invertible, this also gives \(\mathrm d\mathscr R_P\succeq0\) on \((\sigma,\theta]\).

Next, for \(i=1,\ldots,n\), let \((\bar X^i,\bar u^i)\) be the optimal pair of the
stopped problem on \([\sigma,\theta]\) with initial state \(e_i\). By
\proref{prop:stoppedunifconv} and \corref{cor:complsq_identity}(i),
\[
0=
\mathbb E\bigg[\int_\sigma^\theta
\big|\widehat R(s)^{1/2}(\bar u^i(s)-\Theta(s)\bar X^i(s-))\big|^2\,\mathrm ds\bigg]
+
\mathbb E\bigg[\int_{(\sigma,\theta]}
\langle \bar X^i(s-),\mathrm d\mathscr R_P(s)\bar X^i(s-)\rangle \bigg].
\]
Note that the first term is nonnegative, and the second is nonnegative since
\(\mathrm d\mathscr R_P\succeq0\) on \((\sigma,\theta]\). Hence
\(\bar u^i(s)=\Theta(s)\bar X^i(s-)\) for a.e. \(s\in[\sigma,\theta]\),
a.s.
By the uniqueness of solutions to the linear SDE satisfied by \(\Phi e_i\), we have \(\bar X^i(s-)=\Phi(s-)e_i\) for \(s\in(\sigma,\theta]\).
Substituting \(\bar X^i(s-)=\Phi(s-)e_i\)
into the second term and summing over \(i\), we obtain
% By uniqueness of the linear equation defining \(\Phi\), we have
% \(\bar X^i(s-)=\Phi(s-)e_i\) for \(s\in(\sigma,\theta]\). 
\[
  0
  =
  \sum_{i=1}^n
  \mathbb E\int_{(\sigma,\theta]}
  \big\langle \Phi(s-)e_i,
  \mathrm d\mathscr R_P(s)\Phi(s-)e_i\big\rangle
  =
  \mathbb E\int_{(\sigma,\theta]}
  \operatorname{tr}\big(\mathrm d\mathscr M(s)\big).
\]
Since \(\mathrm d\mathscr M\succeq0\), it follows that
\(\mathrm d\mathscr M=0\) on \((\sigma,\theta]\). The invertibility of
\(\Phi\) gives \(\mathscr R_P=0\) on \((\sigma,\theta]\).

Finally, letting \(q\to\infty\), we obtain
\(\mathscr R_P=0\) on \((T_{j-1},T_j\wedge T\wedge\tau_k]\). Taking the union
over \(j\) gives \(\mathscr R_P=0\) on \((0,T\wedge\tau_k]\). Letting
\(k\to\infty\) and using \(\tau_k=T\) for all large \(k\), a.s., we obtain
\(\mathscr R_P\equiv0\) on \((0,T]\). The relation \eqref{AP_identified}
follows from \eqref{completedresidual}. Combining \eqref{AP_identified} with
\eqref{Psemimartingale_canonical}--\eqref{Psemimartingale_PRP} and \(P(T)=G\),
we conclude that \((P,\Lambda,\zeta)\) satisfies the SRE-J \eqref{SREJ}. 
This completes the proof.
\end{proof}

% At this stage, \((P,\Lambda,\zeta)\) satisfies the SRE-J with
% \((\Lambda,\zeta)\) known only to be locally square-integrable. The global
% square-integrability required by \defref{def:strongregular} is provided by
% \lemref{lem:BMO} in the next subsection.

% By \eqref{Psemimartingale_canonical}, \lemref{lem:SREidentify}(a), and
% $P(T)=G$, the triple $(P,\Lambda,\zeta)$ satisfies the SRE-J \eqref{SREJ};
% the global square-integrability of $(\Lambda,\zeta)$ required by
% \defref{def:strongregular} is provided by \lemref{lem:BMO} in the next
% subsection.

\subsection{Proof of \texorpdfstring{\thmref{thm:SREsolve}}{Theorem 5.2}}\label{sec:feedback}
We now complete the proof of \thmref{thm:SREsolve}. 
By \lemref{lem:SREidentify}, the aggregated value process already satisfies the
SRE-J, but its martingale components are only locally square-integrable. 
The next lemma gives the
global square-integrability required in \defref{def:strongregular}.

In the following, for a triple \((\bar P,\bar\Lambda,\bar\zeta)\), let
\(\bar{\widehat S}\) and \(\bar{\widehat R}\) be defined by \eqref{RShat}
with \((P,\Lambda,\zeta)\) replaced by
\((\bar P,\bar\Lambda,\bar\zeta)\). 
Moreover, if \(\bar{\widehat R}\) is invertible,
we set
\(
  \bar\Theta:=-\bar{\widehat R}^{-1}\bar{\widehat S}.
\)

% The
% next lemma upgrades this local estimate to a global one. It is an a priori
% estimate for bounded SRE-J solutions with uniformly positive
% \(\widehat R\); the admissibility of the value feedback will be obtained later
% from the completion-of-squares identity and open-loop optimality.

% The last ingredient is an a priori estimate: the martingale part of any
% bounded solution of the SRE-J with uniformly positive $\widehat R$ is a BMO
% martingale. It upgrades the local square-integrability of $(\Lambda,\zeta)$
% to the global one required by \defref{def:strongregular} and yields the
% well-posedness of the closed-loop system.

% [Square-integrability estimate]

\begin{lem}\label{lem:BMO}
Let \((\bar P,\bar\Lambda,\bar\zeta)\) satisfy the SRE-J \eqref{SREJ}, where
\(\bar P\) is bounded, c\`adl\`ag, and adapted, and
\((\bar\Lambda,\bar\zeta)\) are locally square-integrable.
Assume there is \(\delta>0\) such that 
\(\bar{\widehat R}\) satisfies
\(\bar{\widehat R}(s)\ge\delta I_m\) a.s. for a.e. \(s\in[0,T]\). Then there exists a
constant \(K\), such that
% depending only on \(T,\delta,\lambda(Z)\), the bounds of the
% coefficients, and \(\|\bar P\|_{\infty}\), 
\begin{equation}\label{square_integrability_estimate}
  \mathbb E\bigg[\int_0^T
  \left(
    |\bar\Lambda(s)|^2+
    \int_Z|\bar\zeta(s,z)|^2\lambda(\mathrm dz)
    +|\bar\Theta(s)|^2
  \right)\mathrm ds\bigg]
  \le K .
\end{equation}
In particular,
\((\bar\Lambda,\bar\zeta)\in
L^2_{\mathbb F}(0,T;\mathbb S^n)\times
G^2_{\mathbb F}(0,T,\lambda;\mathbb S^n)\) and
\(\bar\Theta\in L^2_{\mathbb F}(0,T;\mathbb R^{m\times n})\).
\end{lem}

\begin{proof}
\emph{Step 1: Estimate of \(\bar\zeta\).}
Let
\(H(s,z):=\mathbf 1_{\{|\bar\zeta(s,z)|>2\|\bar P\|_{\infty}\}}\) and let \(\mathsf N_T:=N((0,T]\times Z)\).
Since the finite-variation part of the SRE-J is absolutely continuous in
\(\mathrm ds\) and the Brownian part is continuous, at each atom
\((T_j,Z_j)\) of \(N\), one has
\(\bar\zeta(T_j,Z_j)=\Delta\bar P(T_j)\). Hence \(H(T_j,Z_j)=0\) a.s. By the
compensator identity,
\[
  \mathbb E\bigg[\int_0^T\!\!\int_Z H(s,z)\lambda(\mathrm dz)\mathrm ds\bigg]
  =
  \mathbb E\bigg[\int_0^T\!\!\int_Z H(s,z)N(\mathrm ds,\mathrm dz)\bigg]
  =
  \mathbb E\bigg[\sum_{j=1}^{\mathsf N_T} H(T_j,Z_j)\bigg]=0, 
\]
Thus
\(
  |\bar\zeta(s,z)|\le 2\|\bar P\|_{\infty},
  \text{a.s. for a.e. }(s,z)\in[0,T]\times Z .
\)
Since \(\lambda(Z)<\infty\), it follows that \(\bar\zeta\in G^2_{\mathbb F}(0,T,\lambda;\mathbb S^n)\).

\emph{Step 2: Estimate of \(\bar\Lambda\).}
Let \(\bar\Gamma^\circ\) be defined by \eqref{Gammacirc_def}, with
\((P,\Lambda,\zeta)\) replaced by
\((\bar P,\bar\Lambda,\bar\zeta)\), and set
\(\bar\Psi:=\bar{\widehat S}^{\top}\bar{\widehat R}^{-1}\bar{\widehat S}\).
Then \(\bar\Psi\succeq0\) and
\[
  \mathrm d\bar P(s)
  =
  \big(\bar\Psi(s)-\bar\Gamma^\circ(s)\big)\mathrm ds
  +\bar\Lambda(s)\mathrm dW(s)
  +\int_Z\bar\zeta(s,z)\widetilde N(\mathrm ds,\mathrm dz).
\]
By the estimate in Step 1, the boundedness of the coefficients, and
\(\lambda(Z)<\infty\),
\begin{equation}\label{trace-inequ}
  |\operatorname{tr}\bar\Gamma^\circ(s)|
  +|\operatorname{tr}(\bar P(s-)\bar\Gamma^\circ(s))|
  \le K(1+|\bar\Lambda(s)|),
  \qquad
  |\operatorname{tr}(\bar P(s-)\bar\Psi(s))|
  \le K\operatorname{tr}\bar\Psi(s).
\end{equation}
For \(j\ge1\), set the stopping time 
\(
  \nu_j:=T\wedge\inf\left\{r\ge0:
  \int_0^r|\bar\Lambda(s)|^2\,\mathrm ds\ge j\right\}.
\)
Taking expectations in the trace equation for \(\bar P\) on \([0,\nu_j]\) and using \eqref{trace-inequ} yields
\begin{equation}\label{trace-inequ-2}
  \mathbb E\int_0^{\nu_j}\operatorname{tr}\bar\Psi(s)\,\mathrm ds
  \le K+K\mathbb E\int_0^{\nu_j}|\bar\Lambda(s)|\,\mathrm ds .
\end{equation}
On the other hand, applying It\^o's formula to \(|\bar P|^2=\operatorname{tr}(\bar P^2)\) 
on  \([0,\nu_j]\) and using \eqref{trace-inequ}-\eqref{trace-inequ-2} gives
\[
\begin{aligned}
  \mathbb E\int_0^{\nu_j}
  \left(
    |\bar\Lambda(s)|^2+
    \int_Z|\bar\zeta(s,z)|^2\lambda(\mathrm dz)
  \right)\mathrm ds  
  &\le K
  +K\mathbb E\int_0^{\nu_j}\operatorname{tr}\bar\Psi(s)\,\mathrm ds
  +K\mathbb E\int_0^{\nu_j}(1+|\bar\Lambda(s)|)\,\mathrm ds  \\
  &\le K+K\mathbb E\int_0^{\nu_j}|\bar\Lambda(s)|\,\mathrm ds .
\end{aligned}
\]
By Young's inequality, one has 
\(K|\bar\Lambda|\le \frac12|\bar\Lambda|^2+\frac{K^2}{2}\)
and then
\(
  \mathbb E\big[\int_0^{\nu_j}|\bar\Lambda(s)|^2\,\mathrm ds\big]\le K^2+2K
\)
with \(K\) independent of \(j\). Since \(\nu_j\uparrow T\) a.s., it follows from 
the monotone
convergence theorem that \(\bar\Lambda\in L^2_{\mathbb F}(0,T;\mathbb S^n)\). 

\emph{Step 3: Estimate of \(\bar\Theta\).}
By \eqref{RShat} and Step 1,
\(|\bar{\widehat S}(s)|\le K(1+|\bar\Lambda(s)|)\). Since
\(\bar{\widehat R}\ge\delta I_m\),
\[
  |\bar\Theta(s)|^2
  \le K|\bar{\widehat S}(s)|^2
  \le K(1+|\bar\Lambda(s)|^2).
\]
Combining this estimate with Step~2 proves
\eqref{square_integrability_estimate}.
\end{proof}

\begin{proof}[Proof of \thmref{thm:SREsolve}]
(i) By \lemref{lem:H3propagation} and the random-time Hilbert-space
construction in \thmref{thm:Prandomtime}, the random-time problem is uniquely
open-loop solvable and \eqref{randomtime_value_in_SRE_section} holds for every
\(\tau\in\mathcal T[0,T]\).

\smallskip
(ii) \emph{Existence.}
By \proref{prop:Psemimartingale}, the family
\(\{P(\tau)\}_{\tau\in\mathcal T[0,T]}\) is aggregated by a bounded
c\`adl\`ag special semimartingale \(P\), with martingale components
\((\Lambda,\zeta)\). Combining \lemref{lem:Rhat_unifpos} with
\lemref{lem:SREidentify}, we obtain that \((P,\Lambda,\zeta)\) satisfies
the SRE-J \eqref{SREJ} and \(\widehat R\ge\varepsilon I_m\). 
Furthermore, \((\Lambda,\zeta)\in L^2_{\mathbb F}(0,T;\mathbb S^n)\times
G^2_{\mathbb F}(0,T,\lambda;\mathbb S^n)\) by \lemref{lem:BMO}. Hence
\((P,\Lambda,\zeta)\) is a strongly regular solution.

\emph{Maximality.}
Let \((P',\Lambda',\zeta')\) be any strongly regular solution of
\eqref{SREJ}, and set
\(\Theta'=-(\widehat R')^{-1}\widehat S'\). For this triple, 
\(\mathscr R_{P'}=0\). By \remref{rem:complsq_portable}, we obtain,
for every \((t,\xi)\in\mathcal D\) and \(u\in\mathcal U[t,T]\), with state
\(X\),
\begin{equation}\label{prime_verification_identity}
  J(t,\xi;u)
  =
  \mathbb E\langle P'(t)\xi,\xi\rangle
  +\mathbb E\int_t^T
  \big|\widehat R'(s)^{1/2}
  \big(u(s)-\Theta'(s)X(s-)\big)\big|^2\,\mathrm ds .
\end{equation}
Taking the infimum over \(u\) and using part~(i) with \(\tau=t\) gives
\(
  \mathbb E\langle P'(t)\xi,\xi\rangle
  \le V(t,\xi)
  =
  \mathbb E\langle P(t)\xi,\xi\rangle .
\)
By the arbitrariness of \(t\) and \(\xi\), and the c\`adl\`ag property of
\(P\) and \(P'\), we have \(P'(t)\le P(t)\) for all \(t\in[0,T]\), a.s.
Thus \(P\) is maximal. 

\emph{Uniqueness.}
Let \((P',\Lambda',\zeta')\) satisfy the additional assumption in (ii). Then \(X'\) is the state process corresponding to \((t,\xi)\) and the control \(u':=\Theta'X'_-\in\mathcal U[t,T]\). It follows from \eqref{prime_verification_identity} that
\(
  \mathbb E\langle P(t)\xi,\xi\rangle=V(t,\xi)\le J(t,\xi;u')=\mathbb E\langle P'(t)\xi,\xi\rangle .
\)
Thus \(P\le P'\), and hence \(P'=P\). Comparing the martingale parts of the two SRE-Js and using the uniqueness of the Brownian--Poisson martingale representation, we obtain \(\Lambda'=\Lambda\) and \(\zeta'=\zeta\).

% If
% \((\widetilde P,\widetilde\Lambda,\widetilde\zeta)\) is another maximal
% strongly regular solution, then \(\widetilde P\le P\), while maximality of
% \(\widetilde P\) gives \(P\le\widetilde P\). Hence \(\widetilde P=P\). It
% follows from the uniqueness of the Brownian--Poisson martingale
% representation that \(\widetilde\Lambda=\Lambda\) and
% \(\widetilde\zeta=\zeta\).

\smallskip
(iii) Write \(u^*=u^*_{t,\xi}\). By part~(i) with \(\tau=t\), we have $V(t,\xi)=\mathbb E[\langle P(t)\xi,\xi\rangle]$.
Applying the completion of square formula \eqref{complsq_global} to the unique open-loop optimal pair
\((X^*,u^*)\), we obtain
\[
\begin{aligned}
  \mathbb E[\langle P(t)\xi,\xi\rangle]
  = V(t,\xi)
  = J(t,\xi;u^*)=
  \mathbb E[\langle P(t)\xi,\xi\rangle]
  +\mathbb E\bigg[\int_t^T
  \big|\widehat R(s)^{1/2}
  \big(u^*(s)-\Theta(s)X^*(s-)\big)\big|^2\,\mathrm ds\bigg] .
\end{aligned}
\]
Since \(\widehat R\ge\varepsilon I_m\), this gives
\(
  u^*(s)=\Theta(s)X^*(s-),
 \text{a.e. }s\in[t,T],\ \text{a.s.}
\)
Hence \(X^*\) satisfies the closed-loop equation \eqref{closestate}.
Finally, by \thmref{Opensolve},
\(u^*=-\mathcal N_t^{-1}\mathcal L_t\xi\). Moreover, it follows from 
\lemref{lem:H3propagation} that
\(\mathcal N_t\ge\varepsilon I\), while \(\mathcal L_t\) is bounded by
\proref{prop:NtLtbounded}(ii). Hence
\[
  \mathbb E\bigg[\int_t^T|\Theta(s)X^*(s-)|^2\,\mathrm ds\bigg]
  =
  \mathbb E\bigg[\int_t^T|u^*(s)|^2\,\mathrm ds\bigg]
  \le K\mathbb E[|\xi|^2] .
\]
By the state estimate \eqref{FSDEestimate} with control \(u^*\), we have
\(
  \mathbb E[\sup_{s\in[t,T]}|X^*(s)|^2]
  \le K\mathbb E|\xi|^2 .
\)
The proof is complete.
\end{proof}

\begin{remark}\label{rem:why_maximal}
In \thmref{thm:SREsolve}(ii), the uniqueness of strongly regular solutions is established only in the class of solutions whose feedback coefficients generate admissible closed-loop controls.
For a general strongly regular solution \((P',\Lambda',\zeta')\), it is not clear whether \(\Theta'X'_-\) belongs to \(\mathcal U[t,T]\).
In fact, as pointed out in~\cite{sun2021indefinite}, when the coefficients are random, the feedback coefficient determined by the Riccati equation may be unbounded, and the well-posedness of the closed-loop system is not obvious.
Under the standard condition, this difficulty can be overcome by the nonnegativity of \(P'\); see Tang~\cite[Theorem~5.1]{tang2015dynamic} and Zhang, Dong, and Meng~\cite{zhang2020backward}.
In the indefinite case, however, \(P'\) need not be nonnegative, and the closed-loop jump matrix may be singular (see \exaref{exam:counterexample2}).
Whether every strongly regular solution coincides with the maximal one remains open under the present assumptions.
We refer to~\cite{ma2026recursive} for a related result, and to~\cite{fu2026darkpools} for a liquidation model in which uniqueness holds among nonnegative solutions.
\end{remark}

\section{Sufficient Conditions for Uniform Convexity}\label{sec:sufficient}

We present some sufficient conditions on the coefficients of the state equation
and the weighting matrices that guarantee the uniform convexity condition
\eqref{N_0uniformalconvex}. We begin with the standard conditions and then show
that uniform convexity may hold even when neither \eqref{convex1} nor
\eqref{convex2} is satisfied.

% [Standard sufficient conditions]

\begin{prop}\label{prop:basic_unifconv}
 Under conditions \ref{SLQH1}--\ref{SLQH2}, suppose that either \eqref{convex1} or \eqref{convex2} holds. Then the mapping $u \mapsto J(t, 0; u)$ is uniformly convex for all $t \in [0, T)$.
\end{prop}

\begin{proof}
  For $u\in\mathcal U[t,T]$, let $X^{(u)}$ be the solution to the following state equation with initial state $0$:
  \begin{equation}\label{FSDEu}
     \begin{cases}
      \mathrm{d}X^{(u)}(s) =\left[ A(s) X^{(u)}\left( s- \right) +B(s) u(s) \right] \mathrm{d}s+\left[ C(s) X^{(u)}\left( s- \right) +D(s) u(s) \right] \mathrm{d}W(s)\\
      \hspace{5em}+\int_Z{\left[ E\left( s,z \right) X^{(u)}\left( s- \right) +F\left( s,z \right) u(s) \right] \widetilde{N}\left( \mathrm{d}s,\mathrm{d}z \right)},\quad s\in [t,T],\\
      \hspace{0.6em} X^{(u)}(t)=0.\\
    \end{cases} 
\end{equation}

(i) If \eqref{convex1} holds, then $J(t,0;u)\ge\mathbb E\int_t^T\langle R(s)u(s),u(s)\rangle\,\mathrm ds\ge\varepsilon_R\,\mathbb E\int_t^T|u(s)|^2\,\mathrm ds$.
% Under condition \eqref{convex1}, since $Q\ge 0$, $G\ge 0$, $S=0$, and $R\ge \varepsilon_R I_m$:
% \begin{align*}
% J(t,0;u)=&~\mathbb{E}\bigg[ \langle GX^{(u)}(T),X^{(u)}(T)\rangle +\int_t^T{\langle Q(s)X^{(u)}(s),X^{(u)}(s)\rangle +\langle R(s)u(s),u(s)\rangle}\mathrm{d}s \bigg] 
% \\
% \ge &~\mathbb{E}\bigg[\int_t^T{\langle R(s)u(s),u(s)\rangle}\mathrm{d}s\bigg]  \ge \varepsilon_R \mathbb{E}\bigg[\int_t^T{|u(s)|^2}\mathrm{d}s\bigg] .
% \end{align*}

(ii) Suppose \eqref{convex2} holds. Set
\(
  H(s):=D(s)^\top D(s)
  +\int_Z F(s,z)^\top F(s,z)\lambda(\mathrm dz).
\)
By \eqref{convex2}, \(H(s)\ge\varepsilon_0 I_m\), and hence \(H^{-1}\) is bounded. 
Then we can rewrite \eqref{FSDEu} as the following BSDE:
\begin{equation}\label{BSDEu}
\begin{cases}
\mathrm dY(s)=\Big\{
A(s)Y(s-)-B(s)H(s)^{-1}
\Big[\big(D(s)^\top C(s)+\int_ZF(s,z)^\top E(s,z)\lambda(\mathrm dz)\big)Y(s-)\\
\hspace{4em}
-D(s)^\top Z(s)-\int_ZF(s,z)^\top r(s,z)\lambda(\mathrm dz)
\Big]\Big\}\mathrm ds+Z(s)\mathrm dW(s)+\int_Zr(s,z)\widetilde N(\mathrm ds,\mathrm dz),
~ s\in[t,T],\\
Y(T)=X^{(u)}(T),
\end{cases}
\end{equation}
whose adapted solution is given by $(Y,Z,r)=\big(X^{(u)},\,CX^{(u)}_-+Du,\,EX^{(u)}_-+Fu\big)$.
By the standard estimate for linear BSDEs with jumps and bounded coefficients (see, e.g., \cite{tang1994necessary,li2009stochastic}),
% with $Y(s)=X^{(u)}(s)$, $Z(s)=C(s)X^{(u)}(s-)+D(s)u(s)$, $r(s,z)=E(s,z)X^{(u)}(s-)+F(s,z)u(s)$.
% By \lemref{existunique}, the unique adapted solution $(Y,Z,r)\in S_{\mathbb{F}}^{2}(t,T;\mathbb{R}^{n})\times L_{\mathbb{F}}^{2}(t,T;\mathbb{R}^{n})\times G_{\mathbb{F}}^{2}(t,T,\lambda;\mathbb{R}^{n})$ of \eqref{BSDEu} satisfies
\begin{equation}
\begin{aligned}\label{bridge}
\mathbb{E}\bigg[ \mathop {\mathrm{sup}} \limits_{t\le s\le T}|Y(s)|^2+\int_t^T{|}Z(s)|^2\mathrm{d}s+\int_t^T{\int_Z{|r(s,z)|^2}\lambda(\mathrm{d}z)}\mathrm{d}s \bigg] \le K\mathbb{E}\left[ |X^{(u)}(T)|^2 \right],
\end{aligned}
\end{equation}
where $K>0$ is independent of $X^{(u)}(T)$. By \ref{SLQH1}, there exists $\beta>0$ such that $\int_Z{|E(s,z)|^2\lambda(\mathrm{d}z)}+|C(s)|^2\le \beta$. Let $\gamma \in(0,1)$ satisfy $2(\gamma^{-1}-1)T\beta<1$. Using $(a+b)^2\ge (1-\gamma)a^2-(\gamma^{-1}-1)b^2$, we obtain
\begin{equation}\label{inequa1}
\begin{aligned}
\mathbb{E}\bigg[\int_t^T\Big(|Z|^2+\int_Z|r|^2\lambda(\mathrm dz)\Big)\mathrm ds\bigg]
&\ge(1-\gamma)\,\mathbb{E}\bigg[\int_t^T\langle Hu,u\rangle\,\mathrm ds\bigg]-(\gamma^{-1}-1)\beta T\,\mathbb{E}\Big[\sup_{t\le s\le T}|X^{(u)}(s)|^2\Big].
\end{aligned}
\end{equation}
Substituting \eqref{inequa1} into \eqref{bridge} and noting that $(\gamma^{-1}-1)\beta T<1$ and $H\ge\varepsilon_0I_m$, we obtain
$(1-\gamma)\varepsilon_0\mathbb{E}\int_t^T|u(s)|^2\mathrm ds\le K\,\mathbb{E}|X^{(u)}(T)|^2$.
Finally, since $Q,R\ge 0$, $S=0$, and $G\ge \varepsilon_G I_n$,
\[
J(t,0;u)\ge\varepsilon_G\,\mathbb{E}|X^{(u)}(T)|^2
\ge \frac{(1-\gamma)\varepsilon_0\varepsilon_G}{K}\,\mathbb{E}\bigg[\int_t^T|u(s)|^2\,\mathrm{d}s\bigg].
\]
This completes the proof.

% \begin{equation}\label{inequa1}
% \begin{aligned}
% \mathbb{E}\bigg[  \int_t^T{|}Z(s)|^2\mathrm{d}s\bigg]
% &\ge (1-\gamma)\mathbb{E}\bigg[ \int_t^T{|D(s)u(s)|^2}\mathrm{d}s\bigg]-(\gamma^{-1}-1)\beta T\mathbb{E}\bigg[\mathop{\mathrm{sup}}\limits_{t\le s\le T}|X^{(u)}(s)|^2\bigg],
% \end{aligned}   
% \end{equation}
% and 
% \begin{equation}\label{inequa2}
%   \begin{aligned}
%     \mathbb{E}\bigg[ \int_t^T{\int_Z{|r(s,z)|^2}\lambda(\mathrm{d}z)}\mathrm{d}s\bigg]
%     \ge (1-\gamma)\mathbb{E}\bigg[ \int_t^T{\int_Z{|F(s,z)u(s)|^2}\lambda(\mathrm{d}z)}\mathrm{d}s\bigg]-(\gamma^{-1}-1)\beta T\mathbb{E}\bigg[ \mathop{\mathrm{sup}}\limits_{t\le s\le T}|X^{(u)}(s)|^2\bigg].
%   \end{aligned}   
%   \end{equation}
% Combining \eqref{inequa1}--\eqref{inequa2} and noting $D(s)^{\top}D(s)+\int_Z{F(s,z)^{\top}F(s,z)\lambda(\mathrm{d}z)}\ge \varepsilon_0 I_m$ yields
% \begin{align*}
% K\mathbb{E}[|X^{(u)}(T)|^2]
% \ge (1-\gamma)\varepsilon_0\mathbb{E}\bigg[\int_t^T{|}u(s)|^2\mathrm{d}s\bigg].
% \end{align*} 
% Finally, since $Q,R\ge 0$, $S=0$, and $G\ge \varepsilon_G I_n$, we have
% $$
% J(t,0;u)\ge \mathbb{E}[\langle GX^{(u)}(T),X^{(u)}(T)\rangle]\ge \varepsilon_G\mathbb{E}[|X^{(u)}(T)|^2]\ge \frac{(1-\gamma)\varepsilon_0\varepsilon_G}{K}\mathbb{E}\bigg[\int_t^T{|}u(s)|^2\mathrm{d}s\bigg].
% $$
% This completes the proof.
\end{proof}

\begin{thm}\label{thm:unifconvex1}
  Let \ref{SLQH1}--\ref{SLQH2} hold. Suppose that $B(\cdot)=0$, $C(\cdot)=0$, $E(\cdot,\cdot)=0$, $S(\cdot)=0$. Let $\Phi=\{\Phi(s);0\le s\le T\}$ be the solution to the random ODE
  $$
  \mathrm{d}\Phi(s)=A(s)\Phi(s)\mathrm{d}s,\qquad \Phi(0)=I_n.
  $$
  Let $\lambda_G$ and $\lambda_Q(s)$ be the essential infimum of the smallest eigenvalues of $G$ and $Q(s)$, respectively. Assume $\lambda_G\ge0$ and $\lambda_Q(s)\ge0$ a.e. 
  Set
  \(
    H(r):=D(r)^\top D(r)+\int_ZF(r,z)^\top F(r,z)\lambda(\mathrm dz).
  \)
  If there exists $\delta>0$ such that
  \begin{equation}\label{unifconvex_cond1}
    \left[
    \frac{\lambda_G}{\|\Phi(T)^{-1}\|_{\infty}^2}
    +\int_r^T\frac{\lambda_Q(s)}
    {\|\Phi(s)^{-1}\|_{\infty}^2}\mathrm ds
    \right]
    \frac{H(r)}{\|\Phi(r)\|_{\infty}^2}
    +R(r)\ge \delta I_m,
  \end{equation}
  a.s. for a.e. $r\in[0,T]$, then $u\mapsto J(t,0;u)$ is uniformly
  convex for every $t\in[0,T)$.
\end{thm}

\begin{proof}
  Fix $t\in[0,T)$ and $u\in\mathcal U[t,T]$, and let $X=X^{(u)}$ be the
  state with $X(t)=0$. Since \(A\) is bounded, \(\Phi(s)\) is
  invertible for all \(s\), and both \(\Phi\) and \(\Phi^{-1}\) are
  essentially bounded. By the variation of constants formula, we have 
  \(
    X(s)=\Phi(s)\Xi(s),
  \)
  where
  \[
    \Xi(s):=\int_t^s\Phi(r)^{-1}D(r)u(r)\mathrm dW(r)
    +\int_t^s\int_Z\Phi(r)^{-1}F(r,z)u(r)
    \widetilde N(\mathrm dr,\mathrm dz).
  \]
  By the It\^{o} isometry for Brownian and compensated Poisson integrals,
  \[
  \mathbb E[|\Xi(s)|^2]
  =
  \mathbb E\bigg[\int_t^s\Big(
  |\Phi(r)^{-1}D(r)u(r)|^2
  +\int_Z|\Phi(r)^{-1}F(r,z)u(r)|^2\lambda(\mathrm dz)
  \Big)\mathrm dr\bigg] .
  \]
  Hence
  \begin{equation}\label{terminal_H_est}
  \begin{aligned}
  \mathbb E[\langle GX(T),X(T)\rangle]
  \ge
  \frac{\lambda_G}{\|\Phi(T)^{-1}\|_\infty^2}\mathbb E[|\Xi(T)|^2]
  \ge
  \mathbb E\bigg[\int_t^T
  \frac{\lambda_G}{\|\Phi(T)^{-1}\|_\infty^2\|\Phi(r)\|_\infty^2}
  \langle H(r)u(r),u(r)\rangle\mathrm dr \bigg].
  \end{aligned}
  \end{equation}
  Similarly, using Fubini's theorem,
  \begin{equation}\label{running_H_est}
  \begin{aligned}
  \mathbb E\bigg[\int_t^T\langle Q(s)X(s),X(s)\rangle\mathrm ds\bigg] \ge
  \mathbb E\bigg[\int_t^T
  \bigg(
  \int_r^T\frac{\lambda_Q(s)}{\|\Phi(s)^{-1}\|_\infty^2}\mathrm ds
  \bigg)
  \frac{\langle H(r)u(r),u(r)\rangle}{\|\Phi(r)\|_\infty^2}
  \mathrm dr\bigg] .
  \end{aligned}
  \end{equation}
  Combining \eqref{terminal_H_est} and \eqref{running_H_est} with the
  control cost leads to
  \[
  \begin{aligned}
  J(t,0;u)
  &\ge
  \mathbb E\bigg[\int_t^T\bigg\langle\bigg[\bigg(
  \frac{\lambda_G}{\|\Phi(T)^{-1}\|_\infty^2}
  +\int_r^T\frac{\lambda_Q(s)}{\|\Phi(s)^{-1}\|_\infty^2}\mathrm ds
  \bigg)\frac{H(r)}{\|\Phi(r)\|_\infty^2}+R(r)\bigg]u(r),u(r)
  \bigg\rangle\mathrm dr  \bigg]
  \\
  &\ge
  \delta\,\mathbb E\bigg[\int_t^T|u(r)|^2\mathrm dr \bigg].
  \end{aligned}
  \]
  Thus $u\mapsto J(t,0;u)$ is uniformly convex.
\end{proof}

\begin{remark}
  Condition \eqref{unifconvex_cond1} allows $R$ to be negative definite provided  $D^{\top}D+\int_Z{F^{\top}F\lambda(\mathrm{d}z)}$ is sufficiently positive definite. Thus, neither \eqref{convex1} nor \eqref{convex2} need hold.
\end{remark}

\section{Illustrative Examples}\label{sec:example}

We present two examples demonstrating that the SRE-J can be solvable in the indefinite case.

\begin{exam}[Indefinite $G$ and negative $R$]\label{exam:illustrative2}
Let $n=2$, $m=1$, $T>0$, and $\lambda(Z)=1$.
Take
$A=B=C=E=S=Q=0$, $D=F=(1,0)^\top$, $G=\begin{pmatrix}5&0\\0&-1\end{pmatrix}$, and $R=-9.$
Then the state equation is
\[
\mathrm dX(s)=(1,0)^\top u(s)\,\mathrm dW(s)
+\int_Z (1,0)^\top u(s)\,\widetilde N(\mathrm ds,\mathrm dz),
\qquad X(t)=\xi\in\mathbb R^2,
\]
and the cost functional is
\(
J(t,\xi;u)
=\mathbb E\left[
5X_1(T)^2-X_2(T)^2
-9\int_t^T |u(s)|^2\,\mathrm ds
\right].
\)
Since $X_2(s)=\xi_2$ and $\lambda(Z)=1$, by the It\^o isometry,
\(
  \mathbb E|X_1(T)|^2
  =\mathbb E|\xi_1|^2
  +2\mathbb E\int_t^T|u(s)|^2\,\mathrm ds.
\)
Hence,
\[
  J(t,\xi;u)
  =5\mathbb E|\xi_1|^2-\mathbb E|\xi_2|^2
  +\mathbb E\bigg[\int_t^T|u(s)|^2\,\mathrm ds\bigg].
\]
In particular,
\(
  J(0,0;u)=\mathbb E\int_0^T|u(s)|^2\,\mathrm ds,
\)
so \ref{N_0uniformalconvex2-I} holds with $\varepsilon=1$.
For these coefficients, the SRE-J over $[0,T]$ reduces to
\[
\begin{cases}
\mathrm dP(s)=\widehat S(s)^\top\widehat R(s)^{-1}\widehat S(s)\,\mathrm ds
+\Lambda(s)\mathrm dW(s)
+\displaystyle\int_Z\zeta(s,z)\widetilde N(\mathrm ds,\mathrm dz),\\
P(T)=G,
\end{cases}
\]
where
\(
\widehat S(s)=D^\top\Lambda(s)+\int_ZF^\top\zeta(s,z)\lambda(\mathrm dz),
\)
and
\(
\widehat R(s)=R+D^\top P(s-)D
+\int_ZF^\top\big(P(s-)+\zeta(s,z)\big)F\,\lambda(\mathrm dz).
\)

It is straightforward to verify that $(P,\Lambda,\zeta)=(G,0,0)$
solves the SRE-J with $\widehat S=0$ and
\(
\widehat R
=1.
\)
% The unique open-loop optimal
% control is \(u^*(s)=0\). 
Thus the SRE-J is solvable although $G$ is
indefinite and $R<0$.
\end{exam}

% \[
% \eta=\mathbb E\eta+\int_0^T\mathbb E[D_t^W\eta\mid\mathcal F_{t-}]
% \,\mathrm dW(t)
% +\int_0^T\int_Z
% \mathbb E[D_{t,z}^N\eta\mid\mathcal F_{t-}]
% \,\widetilde N(\mathrm dt,\mathrm dz).
% \]
% Let $\mu(t)$ be a c\`adl\`ag modification of
% $\mathbb E[\eta\mid\mathcal F_t]$, and let
% $\lambda_W(t)$ and $\lambda_N(t,z)$ be predictable versions of
% $\mathbb E[D_t^W\eta\mid\mathcal F_{t-}]$ and
% $\mathbb E[D_{t,z}^N\eta\mid\mathcal F_{t-}]$, respectively. Then

\begin{exam}[Random SRE-J solution with a nonzero jump component]\label{exam:malliavin}
Let $n=2$, $m=1$, $T>0$, and $\lambda(Z)=1$. 
Let $\eta\in L_{\mathcal F_T}^{\infty}(\Omega;\mathbb R)$ be Malliavin
differentiable with respect to both $W$ and $\widetilde N$. 
By the Clark--Ocone formula on the Wiener--Poisson space
(see, e.g., Di Nunno, {\O}ksendal, and
Proske~\cite[Theorem~12.20]{di2009malliavin}),
the bounded martingale $\mu(t):=\mathbb E[\eta\,|\,\mathcal F_t]$ admits the
representation
\begin{equation}\label{mu_malliavin_example}
\mu(t)=\mathbb E\eta+\int_0^t\lambda_W(s)\,\mathrm dW(s)
+\int_0^t\int_Z\lambda_N(s,z)\,\widetilde N(\mathrm ds,\mathrm dz),
\qquad t\in[0,T],
\end{equation}
where $\lambda_W$ and $\lambda_N$ are predictable versions of
$\mathbb E[D_t\eta\,|\,\mathcal F_t]$ and
$\mathbb E[D_{t,z}\eta\,|\,\mathcal F_{t-}]$, respectively.

Consider the SLQ problem where the coefficients of the state equation are
given by
\[
A(s)=\begin{pmatrix}0&1\\0&0\end{pmatrix},\quad
B(s)=0,\quad C(s)=0,\quad E(s,z)=0,\quad
D(s)=F(s,z)=\binom{1}{2},
\]
and the weighting matrices in the cost functional are given by $S(s)=0$ and
\[
G=\begin{pmatrix}-(1+T^2)&T\\T&1+T^2\end{pmatrix}
+\eta\begin{pmatrix}4&-2\\-2&1\end{pmatrix},
\quad
Q(s)=\begin{pmatrix}2s&s^2\\s^2&-4s\end{pmatrix}
+4\mu(s-)\begin{pmatrix}0&-1\\-1&1\end{pmatrix},
\quad
R(s)=-(1+s^2).
\]
For this SLQ problem, the associated SRE-J over $[0,T]$ is
\begin{equation}\label{SREJ_example_malliavin}
\begin{cases}
\mathrm dP
=-\big(A^\top P_-+P_-A+Q-\widehat S^\top\widehat R^{-1}\widehat S\big)
\,\mathrm ds
+\Lambda\,\mathrm dW+\displaystyle\int_Z\zeta\widetilde N(\mathrm ds,\mathrm dz),\\
P(T)=G,
\end{cases}
\end{equation}
where
\(
\widehat S(s)=D^\top\Lambda(s)+\int_ZF^\top\zeta(s,z)\lambda(\mathrm dz),
\)
and
\(
\widehat R(s)=R(s)+D^\top P(s-)D
+\int_ZF^\top\big(P(s-)+\zeta(s,z)\big)F\,\lambda(\mathrm dz).
\)
It is straightforward to verify that the adapted solution
$(P,\Lambda,\zeta)$ of \eqref{SREJ_example_malliavin} is given by
\begin{equation}\label{sol_malliavin_example}
P(s)=\begin{pmatrix}-(1+s^2)&s\\s&1+s^2\end{pmatrix}
+\mu(s)\begin{pmatrix}4&-2\\-2&1\end{pmatrix},
\quad
\Lambda(s)=\lambda_W(s)\begin{pmatrix}4&-2\\-2&1\end{pmatrix},
\quad
\zeta(s,z)=\lambda_N(s,z)\begin{pmatrix}4&-2\\-2&1\end{pmatrix}.
\end{equation}

From \eqref{sol_malliavin_example}, we observe that $P$ need not be positive
definite, and the jump component $\zeta$ may be nonzero. Indeed, take
\[
\eta=\frac1{16}\sin(W(T)^2)+\frac18\mathbf 1_{\{\mathsf{N}_T\ge1\}},
\qquad
\mathsf{N}_T:=N((0,T]\times Z).
\]
Then $\|\eta\|_\infty\le 3/16<1/4$, and hence
\(
P_{11}(s)=4\mu(s)-1-s^2<0,
\)
so $P(s)$ is not positive definite.  Moreover,
\(
D_{t,z}\eta=\frac18\mathbf 1_{\{\mathsf N_T=0\}}.
\)
By the independent increments of the Poisson random measure and
$\lambda(Z)=1$,
\[
\lambda_N(t,z)
=\mathbb E[D_{t,z}\eta\mid\mathcal F_{t-}]
=\frac18e^{-(T-t)}\mathbf 1_{\{N((0,t)\times Z)=0\}}.
\]
Thus $\zeta$ is nonzero before the first jump. 

The cost functional of this SLQ problem is uniformly convex. To see this, applying It\^o's formula to $s\mapsto \langle P(s)X(s),X(s)\rangle$ and noting 
\(
\widehat R(s)
=5(1+s^2)+8s
\), we obtain 
\[
J(t,0;u)
=\mathbb E\bigg[\int_t^T\widehat R(s)|u(s)|^2\,\mathrm ds\bigg]
\ge
5\,\mathbb E\bigg[\int_t^T|u(s)|^2\,\mathrm ds\bigg] ,
\]
so the uniform convexity condition \ref{N_0uniformalconvex2-I} holds with $\varepsilon=5$.
\end{exam}

\section{Conclusion}\label{sec:conclusion}

In this paper, we have studied an indefinite stochastic linear-quadratic optimal control problem with random coefficients and Poisson jumps.
The weighting matrices in the cost functional are allowed to be indefinite, and the uniform convexity condition is imposed instead of the standard condition.
We prove that the associated SRE-J admits a maximal strongly regular solution and derive the closed-loop representation of the open-loop optimal control.
A key point is that the Riccati process is constructed from the stochastic value flow, since the global representation used in the diffusion case may fail in the presence of jumps.
Some sufficient conditions for uniform convexity are given, and examples are presented to illustrate the indefinite case.
These results extend the closed-loop representation of open-loop optimal controls
from indefinite SLQ problems driven only by Brownian motion to those with Poisson jumps.
In particular, no invertibility condition is imposed on the jump coefficient $I_n+E$ or on the closed-loop jump matrix $I_n+E+F\Theta$.

Several questions remain open. First, it is not known whether every strongly regular solution of the SRE-J coincides with the maximal one; see \remref{rem:why_maximal}. Second, the closed-loop solvability of Problem (SLQ-J) in the sense of~\cite{sun2016open} deserves further investigation in the case of random coefficients. Third, it would be interesting to extend the value-flow approach to mean-field and game problems with Poisson jumps.
% It remains open whether the SRE-J admits a unique strongly regular solution
% for indefinite SLQ problems.

\subsection*{Statement on AI use}
AI-assisted tools were used to improve the language and readability
of the manuscript, including grammatical correction, and not to
generate scientific content. After using these tools, the authors reviewed and edited the content as needed and take full
responsibility for the content of the published article.

\bibliographystyle{plain}

\bibliography{jump-2}

\end{document}